\documentclass[a4paper,fleqn,leqno]{article}

\usepackage{latexsym,amssymb,amsthm,amsmath,mathtools,epsfig,color,epsf}
\usepackage[english]{babel}
\usepackage[normalem]{ulem}

\usepackage{hyperref}
\usepackage[utf8]{inputenc}
\usepackage{graphicx, tikz, pgfplots,pgfplotstable}
\usepackage{ifthen,tkz-euclide,txfonts}
\usepackage[ampersand]{easylist}
\pgfplotsset{compat=1.5}
\usepackage{caption}
\usepackage{subcaption}

 \makeatletter
 \@addtoreset{equation}{section}
 \makeatother

 \newcounter{extralabel}[section]

 \newtheorem{ittheorem}{Theorem}
 \newtheorem{itlemma}{Lemma}
 \newtheorem{itproposition}{Proposition}
 \newtheorem{itdefinition}{Definition}
 \newtheorem{itcorollary}{Corollary}
 \newtheorem{itconjecture}{Conjecture}
 \newtheorem{itremark}{Remark}
 \newtheorem{itassumption}{Assumption}

 \newenvironment{theorem}{\addtocounter{extralabel}{1}
 \begin{ittheorem}}{\end{ittheorem}}

 \newenvironment{lemma}{\addtocounter{extralabel}{1}
 \begin{itlemma}}{\end{itlemma}}

 \newenvironment{proposition}{\addtocounter{extralabel}{1}
 \begin{itproposition}}{\end{itproposition}}

 \newenvironment{definition}{\addtocounter{extralabel}{1}
 \begin{itdefinition}}{\end{itdefinition}}

 \newenvironment{corollary}{\addtocounter{extralabel}{1}
 \begin{itcorollary}}{\end{itcorollary}}

 \newenvironment{remark}{\addtocounter{extralabel}{1}
 \begin{itremark}}{\end{itremark}}
 
 \newenvironment{assumption}{\addtocounter{extralabel}{1}
 \begin{itassumption}}{\end{itassumption}}

\newcommand{\Rand}[1]{\marginpar{#1}}

\newcommand{\be}[1]{\Rand{\vspace{0,6cm}\tt #1}\begin{equation}\label{#1}}
\newcommand{\beL}[1]{\Rand{\vspace{0,6cm}\tt #1}\begin{lemma}\label{#1}}
\newcommand{\belC}[2]{\Rand{\vspace{0,6cm}\tt #1}\begin{lemma}[#2]\label{#1}}
\newcommand{\beP}[1]{\Rand{\vspace{0,6cm}\tt #1}\begin{proposition}\label{#1}}
\newcommand{\bePC}[2]{\Rand{\vspace{0,6cm}\tt #1}\begin{proposition}[#2]\label{#1}}
\newcommand{\beD}[1]{\Rand{\vspace{0,6cm}\tt #1}\begin{definition}\label{#1}}
\newcommand{\beT}[1]{\Rand{\vspace{0,6cm}\tt #1}\begin{theorem}\label{#1}}
\newcommand{\beC}[1]{\Rand{\vspace{0,6cm}\tt #1}\begin{corollary}\label{#1}}
\newcommand{\bePr}[1]{\Rand{\vspace{0,6cm}\tt #1}\begin{proof}\label{#1}}

\newcommand{\bea}[1]{\Rand{\vspace{0,7cm}\tt #1\vspace{-0,7cm}}\begin{eqnarray}\label{#1}}

\renewcommand{\d}{{\rm d}}
\newcommand{\e}{{\rm e}}
\newcommand{\E}{\mathbb{E}}
\renewcommand{\P}{\mathbb{P}}

 \def\1{{\mathchoice {1\mskip-4mu\mathrm l} 
{1\mskip-4mu\mathrm l}
{1\mskip-4.5mu\mathrm l} {1\mskip-5mu\mathrm l}}}

\def\CA{\mathcal{A}}
\def\CB{\mathcal{B}}

\def\CE{\mathcal{E}}
\def\CF{\mathcal{F}}
\def\CG{\mathcal{G}}

\def\CM{\mathcal{M}}

\def\CL{\mathcal{L}}
\def\CI{\mathcal{I}}
\def\CJ{\mathcal{J}}
\def\CT{\mathcal{T}}

\def\CP{\mathcal{P}}

\def\CV{\mathcal{V}}

\def\CR{\mathcal{R}}

\def\E{\mathbb{E}}
\def\G{\mathbb{G}}

\def\N{\mathbb{N}}
\def\P{\mathbb{P}}

\def\R{\mathbb{R}}
\def\S{\mathbb{S}}

\def\Z{\mathbb{Z}}

\def\suml{\sum\limits}

\DeclareMathSymbol{\varNu}{\mathord}{letters}{78}

\newcommand{\ee}{\end{equation}}
\newcommand{\eea}{\end{eqnarray}}
\newcommand{\bean}{\begin{eqnarray*}}
\newcommand{\eean}{\end{eqnarray*}}

\definecolor{rood}{rgb}{1,0,0}

\usepackage{color}
\usepackage{colortbl}

\newtheorem{xx}{\bf xxx}

\newtheorem{zz}{\bf zzz}

\newtheorem{yy}{\bf yyy}

\begin{document}


\title{Spatial populations with seed-bank:\\
finite-systems scheme}

\author{Andreas Greven$^1$, Frank den Hollander$^2$}

\date{19 September 2022}

\maketitle

\begin{abstract}
This is the third in a series of four papers in which we consider a system of interacting Fisher-Wright diffusions with seed-bank. Individuals carry type $\heartsuit$ or $\diamondsuit$, live in colonies, and are subject to resampling and migration as long as they are \emph{active}. Each colony has a structured seed-bank into which individuals can retreat to become \emph{dormant}, suspending their resampling and migration until they become active again. As geographic space labelling the colonies we consider a countable Abelian group $\G$ endowed with the discrete topology. Our goal is to understand in what way the seed-bank enhances genetic diversity and causes new phenomena. 

In \cite{GdHOpr1} we showed that the system of continuum stochastic differential equations, describing the population in the large-colony-size limit, has a unique strong solution. We further showed that if the system starts from an initial law that is invariant and ergodic under translations with a density of $\heartsuit$ that is equal to $\theta$, then it converges to an equilibrium $\nu_\theta$ whose density of $\heartsuit$ also is equal to $\theta$. Moreover, $\nu_\theta$ exhibits a dichotomy of \emph{coexistence} (= locally multi-type equilibrium) versus \emph{clustering} (= locally mono-type equilibrium). We identified the parameter regimes for which these two phases occur, and found that these regimes are different when the mean wake-up time of a dormant individual is finite or infinite.   

The goal of the present paper is to establish the \emph{finite-systems scheme}, i.e., identify how a finite truncation of the system (both in the geographic space and in the seed-bank) behaves as both the time and the truncation level tend to infinity, properly tuned together. Since the finite system exhibits clustering, we focus on the regime where the infinite system exhibits coexistence, which consists of two sub-regimes. If the wake-up time has finite mean, then the scaling time turns out to be proportional to the volume of the truncated geographic space, and there is a \emph{single universality class} for the scaling limit, namely, the system moves through a succession of equilibria of the infinite system with a density of $\heartsuit$ that evolves according to a \emph{renormalised Fisher-Wright diffusion} and ultimately gets trapped in either $0$ or $1$. On the other hand, if the wake-up time has infinite mean, then the scaling time turns out to grow faster than the volume of the truncated geographic space, and there are \emph{two universality classes} depending on how fast the truncation level of the seed-bank grows compared to the truncation level of the geographic space. For slow growth the scaling limit is the same as when the wake-up time has finite mean, while for fast growth the scaling limit is different, namely, the density of $\heartsuit$ initially remains fixed at $\theta$, afterwards makes \emph{random switches} between $0$ and $1$ on a range of different time scales, driven by individuals in deep seed-banks that wake up, until it finally gets \emph{trapped} in either $0$ or $1$ on the time scale where the individuals in the deepest seed-banks wake up. Thus, the system evolves through a sequence of \emph{partial clusterings} (or partial fixations) before it reaches \emph{complete clustering} (or complete fixation).  



\medskip\noindent 
\emph{Acknowledgements:} 
AG was supported by the Deutsche Forschungsgemeinschaft  (through grant DFG-GR 876/16-2 of SPP-1590), FdH was supported by the Netherlands Organisation for Scientific Research (through NWO Gravitation Grant NETWORKS-024.002.003), and by the Alexander von Humboldt Foundation (during visits to Bonn and Erlangen in the Fall of 2019, 2020 and 2021). The present paper grew out of joint work with Margriet Oomen \cite{GdHOpr1}, \cite{GdHOpr3}, to whom the authors are indebted.   
\end{abstract}

\bigskip

\footnoterule
\noindent
\hspace*{0.3cm} {\footnotesize $^{1)}$ 
Department Mathematik, Universit\"at Erlangen-N\"urnberg, Cauerstrasse 11, D-91058 Erlangen, Germany\\
greven@mi.uni-erlangen.de}\\
\hspace*{0.3cm} {\footnotesize $^{2)}$ 
Mathematisch Instituut, Universiteit Leiden, Niels Bohrweg 1, 2333 CA  Leiden, NL\\
denholla@math.leidenuniv.nl}


\tableofcontents


\section{Introduction}
\label{s.introduct}

Section~\ref{ss.goal} outlines the goal of the paper. Section~\ref{ss.process} recalls the three models introduced in \cite[Section 2]{GdHOpr1}. Section~\ref{ss.obs} contains some key observations made in \cite[Section 2]{GdHOpr1} concerning the choice of model parameters and initial laws, and the role of duality and diffusion function. Section~\ref{ss.coreinf} summarises the core results in \cite[Section 3]{GdHOpr1} and provides a brief outline of the remainder of the paper. Section~\ref{s.scaling} describes our main results for the finite-system scheme. Section~\ref{s.classpres} contains preparations and Sections~\ref{s.fssrhofin}--\ref{s.fssrhoinffast} provide proofs. For background and motivation we refer the reader to \cite[Section 1]{GdHOpr1}.  

Sections~\ref{ss.process}--\ref{ss.coreinf} are largely copied from \cite{GdHOpr1}, but are needed to set the stage.


\subsection{Goal}
\label{ss.goal}

In \cite{GdHOpr1} we considered a system of interacting Fisher-Wright diffusions with seed-bank. Individuals carry type $\heartsuit$ or $\diamondsuit$, live in colonies, and are subject to resampling and migration as long as they are \emph{active}. Each colony has a structured seed-bank into which individuals can retreat to become \emph{dormant}, suspending their resampling and migration until they become active again. As geographic space labelling the colonies we considered a countable Abelian group $\G$ endowed with the discrete topology. We showed that the system of continuum stochastic differential equations, describing the population in the large-colony-size limit, has a unique strong solution. We further showed that if the system starts from an initial law that is invariant and ergodic under translations with a density of $\heartsuit$ that is equal to $\theta$, then the system converges to an equilibrium $\nu_\theta$ whose density of $\heartsuit$ also is equal to $\theta$. Moreover, $\nu_\theta$ exhibits a dichotomy of \emph{coexistence} (= locally multi-type equilibrium) versus \emph{clustering} (= locally mono-type equilibrium). We identified the parameter regimes for which these two phases occur, and found that these regimes are qualitatively different when the mean wake-up time of an individual is finite or infinite.   

The goal of the present paper is to establish the \emph{finite-systems scheme}, i.e., identify how a finite truncation of the system (both in the geographic space and in the seed-bank) behaves as both the time and the truncation level tend to infinity, properly tuned together. Since the finite system exhibits clustering, we focus on the regime where the infinite system exhibits coexistence. To allow for a proper truncation limit, we additionally assume that $\G$ is \emph{profinite}, i.e., $\G$ is (isomorphic to) the limit of a projective system of finite groups endowed with the discrete topology. Our main findings are the following:
\begin{itemize}
\item
If the wake-up time has finite mean, then the scaling time is proportional to the volume of the truncated geographic space and there is a \emph{single universality class} for the scaling limit, namely, the system moves through a succession of equilibria of the infinite system with a density of $\heartsuit$ that evolves according to a \emph{renormalised Fisher-Wright diffusion} and ultimately gets trapped in either $0$ or $1$ (i.e., the system ultimately clusters). 
\item
If the wake-up time has infinite mean, then the scaling time grows faster than the volume of the truncated geographic space, and there are \emph{two universality classes} depending on how fast the truncation level of the seed-bank grows compared to the truncation level of the geographic space. For slow growth the scaling limit is the same as when the wake-up time has finite mean, while for fast growth the scaling limit is different, namely, the density of $\heartsuit$ initially remains fixed at $\theta$, afterwards makes \emph{random switches} between $0$ and $1$ on a range of different time scales, driven by individuals in deeper seed-banks that wake up, until it finally gets \emph{trapped} in either $0$ or $1$ on the time scale where the individuals in the deepest seed-banks wake up. Thus, the system evolves through a sequence of \emph{partial clusterings} (or partial fixations) before it reaches \emph{complete clustering} (or complete fixation).  
\end{itemize}
The finite-systems scheme underpins the relevance of systems with an infinite geographic space and an infinite seed-bank for the description of systems with a large but finite geographic space and a large but finite seed-bank per colony. 


\subsection{Migration, resampling and seed-bank: three models}
\label{ss.process}

In this section, which is largely copied from \cite[Section 2]{GdHOpr1}, we restrict ourselves to recalling the definition of the three models introduced in \cite{GdHOpr1}. In \cite[Appendix A]{GdHOpr1} we explained how the system of continuum stochastic differential equations that is our object of study arises as the large-colony-size limit of a sequence of discrete individual-based systems. In what follows we will interpret properties of the continuum system in terms of the underlying discrete systems in order to provide the proper intuition.  

We consider populations of individuals that can be of two types: $\heartsuit$ and $\diamondsuit$. These populations are located in a \emph{geographic space} $\G$ with a group structure, namely, $(\G,+)$ is a countable Abelian group with $+$ a group action. We are interested in the frequencies of types in the various locations. To describe the migration we use a {\em migration kernel} $a(\cdot,\cdot)$, which is a $\G \times \G$ matrix of transition rates that determines a continuous-time random walk and satisfies:

\begin{assumption}{\bf [Migration kernel]}
\label{ass.migration}
\begin{equation}
\label{gh10}
a(i,j) = a(0,j-i) \quad \forall\,i,j \in \G, \qquad \sum_{i \in \G} a(0,i) < \infty, \qquad a(\cdot,\cdot) \text{ is irreducible.} 
\hfill\Box
\end{equation}
\end{assumption}

In each of the three models to be described below, the population at a location consist of an \emph{active part} and a \emph{dormant part}. The seed-bank at a given location is the repository of the dormant population at that location (and for two of the models has an internal structure that regulates the wake-up time). For each active individual that becomes dormant a randomly chosen dormant individual becomes active, i.e., the active and the dormant population \emph{exchange} individuals (see Fig.~\ref{fig:Model1}). This guarantees that the sizes of the active and the dormant population stay fixed over time. 


\paragraph{Model 1: single-layer seed-bank.}

For $i \in \G$ and $t \geq 0$, let $x_i(t)$ denote the fraction of individuals in colony $i$ of type $\heartsuit$ that are active at time $t$, and $y_i(t)$ the fraction of individuals in colony $i$ of type $\heartsuit$ that are dormant at time $t$. These fractions evolve according to the SSDE 
\begin{eqnarray}
\label{gh1}
\d x_i(t) &=& \sum_{j \in \G} a(i,j)\, [x_j (t) - x_i(t)]\,\d t
+ \sqrt{d x_i(t)[1-x_i(t)]}\,\d w_i(t)\\ \nonumber
&&+\,Ke\,[y_i(t) - x_i(t)]\,\d t,\\[0.2cm]
\label{gh2}
\d y_i (t) &=& e\,[x_i(t) - y_i(t)]\,\d t,
\end{eqnarray}
where  $(w_i(t))_{t \geq 0}$, $i \in \G$, are independent standard Brownian motions. The first term in \eqref{gh1} describes the \emph{migration} of individuals (at rate $a(i,j)$ from $j$ to $i$), the second term in \eqref{gh1} describes the \emph{resampling} of individuals (at rate $d \in (0,\infty)$ for all $i$). The third term in \eqref{gh1} together with the term in \eqref{gh2} describe the \emph{exchange} of active and dormant individuals (at rate $e \in (0,\infty)$ for all $i$). See Fig.~\ref{fig:Model1} for an illustration.

\medskip\noindent
\vspace{-.5cm}
\begin{figure}[htbp]
\begin{center}
\begin{tikzpicture}
\draw [fill=red!20!blue!20!,ultra thick] (0,0) rectangle (2,4);
\draw [fill=red!20!blue!20!,ultra thick] (4,0) rectangle (6.5,4);
\node  at (1,2) {\text{\Large A}};
\node  at (5.25,2) {\text{\Large D}};
\draw[ultra thick,<-](2.25,2)--(3.75,2);
\draw[ultra thick,->](2.25,3)--(3.75,3);
\draw[ultra thick,<-](-1.5,2)--(0,2);
\draw[ultra thick,->](-1.5,3)--(0,3);
\node at (3,4.3) {\text{exchange}};
\node at (1,5.3) {\text{resampling}};
\node at (-1,4.3) {\text{migration}};
\draw [ultra thick] (1,3.75)to [out=180,in=180](1, 4.75) ;
\draw [ultra thick,->] (1,4.75)to [out=0,in=60](1.1, 3.75) ;
\node[above]  at (3,3) {$Ke$};
\node[above]  at (3,2) {$e$};
\node[above]  at (-0.75,3) {};
\node[above]  at (-0.75,2) {};
\node[left]  at (1.7,4.5) {$d$};
\end{tikzpicture}
\vspace{0.2cm}
\caption{\small The evolution in a single colony in Model 1. Individuals are subject to migration (see Fig.~\ref{fig:spatial}), resampling and exchange with the seed-bank.}
\label{fig:Model1}
\end{center}
\vspace{-.5cm}
\end{figure}
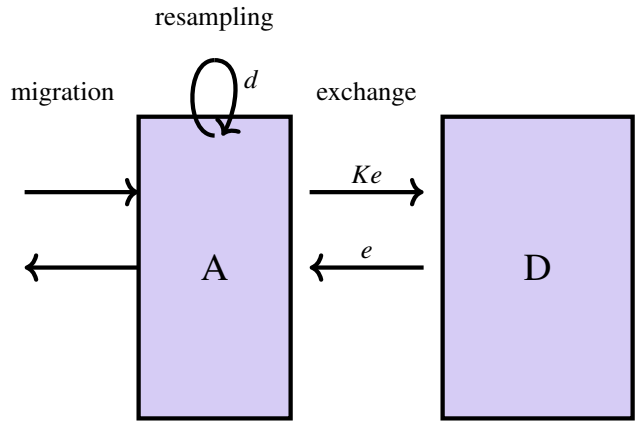

The factor $K \in (0,\infty)$ allows for an \emph{asymmetry} between the sizes of the active and the dormant population. Indeed, because we are tracking \emph{fractions} of individuals of type $\heartsuit$, we have 
\begin{equation}
\label{ratio1}
K = \frac{\text{size dormant population}}{\text{size active population}}, 
\end{equation}
and this ratio is the same for all colonies. 

The \emph{state space} of the system is
\begin{equation}
\label{Edef}
E= [0,1]^\S, \qquad \S=\G\times\{A,D\}, 
\end{equation} 
endowed with the product topology, where $A$ denotes the residence of the active population and $D$ the repository of the dormant population. Alternatively, we may also think of two populations, one active and one dormant, and write $E = ([0,1] \times [0,1])^\G$. Accordingly, the configuration of the system at time $t$ is written as
\begin{equation}
Z(t) = \big(z_{u}(t)\big)_{u \in \S}
\end{equation} 
with $z_{u}(t)=x_i(t)$ if $u=(i,A)$ and $z_{u}(t) = y_i(t)$ if $u=(i,D)$, respectively, 
\begin{equation}
\label{e409}
Z(t) = \big(z_i(t)\big)_{i\in\G}
\end{equation} 
with $z_i(t) = (x_i(t),y_i(t))$. Via duality, our SSDE can be understood in terms of tuples of random walks on $\G$ with internal states $A$ and $D$ \cite[Section 2]{GdHOpr1}.  


\paragraph{Model 2: multi-layer seed-bank.}

\begin{figure}[htbp]
\begin{center}
\begin{tikzpicture}
\draw [fill=red!20!blue!20!,ultra thick] (0,0) rectangle (2,4);
\draw [fill=red!20!blue!20!,ultra thick] (4,0) rectangle (6.5,4);
\draw [fill=yellow!50!, ultra thick] (4.25,3.15) rectangle (6.25,3.85); 
\draw [fill=red!50!, ultra thick] (4.25,2.15) rectangle (6.25,2.85); 
\draw [fill=green!50!, ultra thick] (4.25,0.55) rectangle (6.25,1.25); 
\node  at (1,2) {\text{\Large A}};
\node  at (5.25,3.45) {\text{\Large $D_0$}};
\node  at (5.25,2.45) {\text{\Large $D_1$}};
\node  at (5.25,0.85) {\text{\Large $D_m$}};
\draw[ultra thick,<-](2.25,3.4)--(3.75,3.4);
\draw[ultra thick,->](2.25,3.7)--(3.75,3.7);
\draw[ultra thick,<-](2.25,2.4)--(3.75,2.4);
\draw[ultra thick,->](2.25,2.7)--(3.75,2.7);
\draw[ultra thick,<-](2.25,0.8)--(3.75,0.8);
\draw[ultra thick,->](2.25,1.1)--(3.75,1.1);
\node at (3,4.3) {\text{exchange}};
\node at (1,5.3) {\text{resampling}};
\draw [ultra thick] (1,3.75)to [out=180,in=180](1, 4.75) ;
\draw [ultra thick,->] (1,4.75)to [out=0,in=60](1.1, 3.75) ;
\foreach \x in {1.9,1.7,1.5,0.35,0.15}
\draw[fill] (5.25,\x) circle [radius=0.05];
\node[above]  at (3,3.6) {$K_0 e_0$};
\node[below]  at (3,3.49) {$e_0$};
\node[above]  at (3,2.6) {$K_1 e_1$};
\node[below]  at (3,2.49) {$e_1$};
\node[above]  at (3,1.0) {$K_me_m$};
\node[below]  at (3,0.89) {$e_m$};
\node[left]  at (1.7,4.5) {$d$};
\node at (-1,4.3) {\text{migration}};
\draw[ultra thick,<-](-1.5,2)--(0,2);
\draw[ultra thick,->](-1.5,3)--(0,3);	
\end{tikzpicture}
\vspace{.3cm}
\caption{\small The evolution in a single colony in Model 2. Individuals are subject to migration (see Fig.~\ref{fig:spatial}), resampling and 
exchange with the seed-bank, as in Model 1. Additionally, when individuals become dormant 
they get a colour and when they become active they loose their colour.}
\label{fig:Model2}
\end{center}
\end{figure}
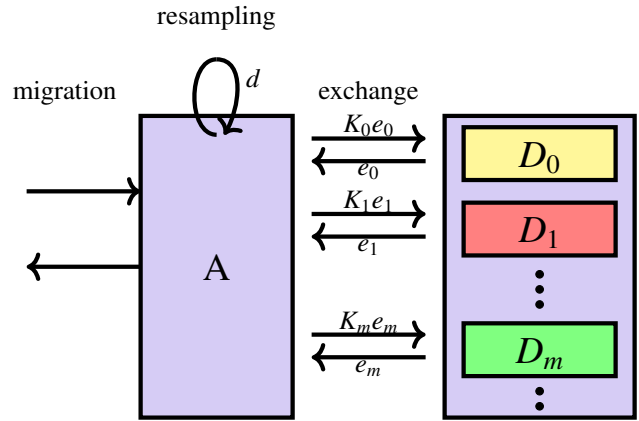

In order to allow for fat tails in the wake-up times of individuals and still preserve the Markov property, we enrich the state space. Namely, we allow individuals to become dormant with a colour that is drawn randomly from an infinite sequence of colours, labelled by $\N_0$. See Figs.~\ref{fig:Model2} and \ref{fig:spatial} for an illustration. 

As before, let $x_i(t)$ denote the fraction of individuals in colony $i$ of type $\heartsuit$ that are active at time $t$, but now let $y_{i,m}(t)$ denote the fraction of individuals in colony $i$ of type $\heartsuit$ that are dormant with colour $m$ at time $t$. Suppose that active individuals exchange with dormant individuals with colour $m$ at rate $e_m \in (0,\infty)$. Then the SSDE in \eqref{gh1}--\eqref{gh2} is replaced by
\begin{eqnarray}
\label{gh1*}
\d x_i(t) &=& \sum_{j \in \G} a(i,j)\,[x_j (t) - x_i(t)]\,\d t
+ \sqrt{d x_i(t)[1-x_i(t)]}\,\d w_i(t)\\ \nonumber
&&+\,\sum_{m\in\N_0} K_me_m\, [y_{i,m}(t) - x_i(t)]\,\d t,\\[0.2cm]
\label{gh2*}
\d y_{i,m} (t) &=& e_m\,[x_i(t) - y_{i,m}(t)]\,\d t, \qquad m\in\N_0,
\end{eqnarray}
where the factor $K_m \in (0,\infty)$ captures the asymmetry between the size of the active population and the $m$-dormant population, i.e., similarly as in \eqref{ratio1}, 
\begin{equation}
\label{ratio2}
K_m = \frac{\text{size $m$-dormant population}}{\text{size active population}},
\qquad m\in\N_0, 
\end{equation}
where $K_m \in (0,\infty)$ is the same for all colonies. The state space is
\begin{equation}
\label{Edef*}
E= [0,1]^\S, \qquad \S = \G \times \{A,(D_m)_{m\in\N_0}\},
\end{equation} 
endowed with the product topology, where $A$ denotes the residence of the active population and $D_m$ the repository of the $m$-dormant population. Alternatively, we may think of infinitely many populations, one active and all the others dormant, and write $E = ([0,1] \times [0,1]^{\N_0})^\G$. Accordingly, the configuration of the system at time $t$ is written as
\begin{equation}
Z(t) = \big(z_{u}(t)\big)_{u\in\S} 
\end{equation}
with $z_{u}(t) = x_i(t)$ if $u=(i,A)$ and $z_{u}(t) = y_{i,m}(t)$ if $u=(i,D_m)$ for some $m\in\N_0$, respectively,
\begin{equation}
\label{e510}
Z(t) = \big(z_i(t)\big)_{i \in \G}
\end{equation}
with $z_i(t) = (x_i(t),(y_{i,m}(t))_{m \in \N_0})$. 

\paragraph{Model 3: multi-layer seed-bank with displaced seeds.}
We can extend the mechanism of Model 2 by allowing individuals that move into a seed-bank to do so in a randomly chosen colony. This amounts to introducing a sequence of irreducible \emph{displacement kernels} $a_m(\cdot,\cdot)$, $m \in \N_0$, satisfying 
\begin{equation}
\label{pmdef}
a_m(i,j) = a_m(0,j-i) \quad \forall\,i,j \in \G,\, m \in\N_0,
\qquad \sup_{m \in\N_0} \sum_{i \in \G} a_m(0,i) < \infty,
\end{equation}  
and replacing \eqref{gh1*}--\eqref{gh2*} by
\begin{eqnarray}
\label{gh1**}
\d x_i(t) &=& \sum_{j \in \G} a(i,j)\,
[x_j (t) - x_i(t)]\,\d t
+ \sqrt{d x_i(t)[1-x_i(t)]}\,\d w_i(t)\\ \nonumber
&&+\, \sum_{j\in\G} \sum_{m\in\N_0} K_me_m\,a_m(j,i)\,[y_{j,m}(t) - x_i(t)]\,\d t,\\[0.2cm]
\label{gh2**}
\d y_{i,m} (t) &=& \sum_{j \in \G} e_m\,a_m(i,j)\,[x_j(t) - y_{i,m}(t)]\,\d t, \qquad m\in\N_0.
\end{eqnarray}
Here, the third term in \eqref{gh1**} together with the term in \eqref{gh2**} describe the \emph{migration} of individuals and the \emph{switch of colony} of individuals during the exchange between active to dormant (at rate $a_m(i,j)$ between $i$ and $j,m$). The state space and the configuration at time $t$ are the same as in \eqref{Edef*}. Also \eqref{ratio2} remains the same.

\begin{figure}
\begin{center}
\includegraphics[width=1\textwidth]{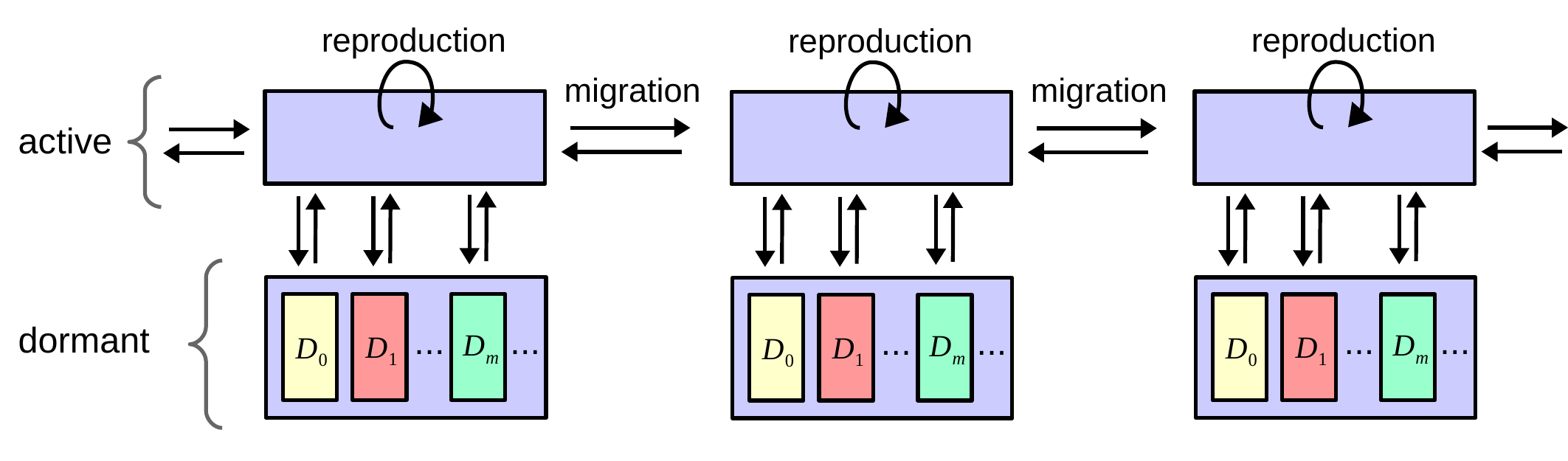}
\end{center}
\caption{\small The spatial evolution in Model 2. Each colony has its own seed-bank. Resampling occurs within the active population in each        colony. Migration occurs between active populations in different colonies (top layer). Exchange between the active and the dormant population occurs in each colony (transitions between the top layer and the bottom layer).}
\label{fig:spatial}
\end{figure}


\paragraph{Duality.}

As shown in \cite[Section 2]{GdHOpr1}, the three models have a tractable dual, in which paths of individuals (see Fig.~\ref{fig:spatial})  are \emph{reversed in time} to become ancestral lineages of individuals (see Fig.~\ref{fig:Duality}). In the dual, which plays a crucial role in the analysis, lineages evolve as independent continuous-time Markov processes with state space
\begin{equation}
\label{lineages1}
\text{$\S=\G \times \{A,D\}$ (Model 1), \quad $\S=\G \times \{A, (D_m)_{m\in\N_0}\}$ (Models 2--3)},
\end{equation}
and with transition kernel 
\begin{equation}
\label{lineages2}
\text{$b^{(1)}(\cdot,\cdot)$ in \eqref{mrw1}, $b^{(2)}(\cdot,\cdot)$ in \eqref{mrw2}, $b^{(3)}(\cdot,\cdot)$ in \eqref{mrw3},}
\end{equation}
and coalesce at rate $d$ when they are at the same site in $\G$ and are both active. Unlike in the original system, there is \emph{no exchange} between active and dormant lineages in the dual, which makes the dual easier to work with. The duality relations are \emph{mixed-moment relations} linking the original process to the dual process and vice versa. The dual is therefore called a \emph{moment dual}.   

For $g \neq dg_{\text{FW}}$ no tractable dual is available and we need to argue by comparison with models that have a tractable dual.  

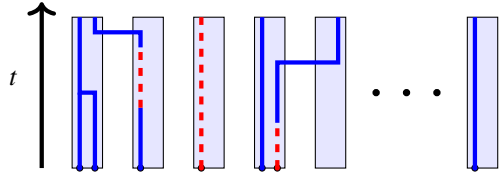
\begin{figure}[htbp]
\begin{center}
\vspace{.2cm}
\setlength{\unitlength}{.5cm}
		\begin{tikzpicture}[scale=0.4]
		\draw [fill=blue!10!] (0,-0.5) rectangle (1,4.5);
		\draw [fill=blue!10!] (2,-0.5) rectangle (3,4.5);
		\draw [fill=blue!10!] (4,-0.5) rectangle (5,4.5);
		\draw [fill=blue!10!] (6,-0.5) rectangle (7,4.5);
		\draw [fill=blue!10!] (8,-0.5) rectangle (9,4.5);
		\draw [fill=black, ultra thick] (10,2) circle [radius=0.05] ;
		\draw [fill=black, ultra thick] (11,2) circle [radius=0.05] ;
		\draw [fill=black, ultra thick] (12,2) circle [radius=0.05] ;
		\draw [fill=blue!10!] (13,-0.5) rectangle (14,4.5);
		\draw [ultra thick, ->] (-1,-.5) -- (-1,5);
		\node[left] at (-1.5,2.5) {$t$};
		\draw [fill=blue] (.25,-.5) circle [radius=0.1] ;
		\draw [fill=blue] (.75,-.5) circle [radius=0.1] ;
		\draw[blue,ultra thick] (.25,-.5)--(.25,2);	
		\draw[blue,ultra thick] (.75,-.5)--(.75,2)--(.25,2)--(.25,4.5);
		\draw [fill=blue] (2.25,-.5) circle [radius=0.1] ;
		\draw[blue,ultra thick] (2.25,-.5)--(2.25,1.5);
		\draw[red,ultra thick, dashed] (2.25,1.5)--(2.25,3.5);
		\draw[blue,ultra thick] (2.25,3.5)--(2.25,4)--(.75,4)--(.75, 4.5);
		\draw [fill=red] (4.25,-.5) circle [radius=0.1] ;
		\draw[red,ultra thick, dashed] (4.25,-.5)--(4.25,4.5);	
		\draw [fill=blue] (6.25,-.5) circle [radius=0.1] ;
		\draw[blue,ultra thick] (6.25,-.5)--(6.25,4.5);
		\draw [fill=red] (6.75,-.5) circle [radius=0.1] ;
		\draw[red,ultra thick, dashed] (6.75,-.5)--(6.75,1);
		\draw[blue,ultra thick] (6.75,1)--(6.75,3)--(8.75,3)--(8.75,4.5);
		\draw [fill=blue] (13.25,-.5) circle [radius=0.1] ;
		\draw[blue,ultra thick] (13.25,-.5)--(13.25,4.5);	
		\end{tikzpicture}
\vspace{0.2cm}
\caption{\small Picture of the evolution of lineages in the spatial coalescent. The purple blocks depict the colonies, the blue lines the active lineages, and the red dashed lines the dormant lineages. Blue lineages can migrate and become dormant, (i.e., become red dashed lineages). Two blue lineages can coalesce when they are at the same colony. Red dashed dormant lineages first have to become active (blue) before they can coalesce with other blue and active lineages or migrate. Note that the dual runs \emph{backwards in time}. The collection of all lineages determines the genealogy of the system.}
\label{fig:Duality}
\end{center}
\end{figure}
 

\subsection{Observations: $\rho<\infty$ versus $\rho=\infty$}
\label{ss.obs} 
 
In this section, which is also largely copied from \cite[Section 2]{GdHOpr1}, we recall a number of important observations regarding the models introduced in Section~\ref{ss.process}.

   
\paragraph{Two key quantities.}

In Models 2 and 3 we must assume that
\begin{equation}
\label{emcond}
\chi = \sum_{m\in\N_0} K_me_m < \infty
\end{equation}
in order for active \emph{lineages} in the dual not to become dormant instantly. The advantage of having infinitely many colours is that it allows us to have \emph{wake-up times with fat tails} and at the same time \emph{preserve the Markov property for the evolution} of the system. Indeed, in Model 2 the rate for an active lineage to become dormant is $\chi$, the go-to-sleep time $\sigma$ of an active lineage has law
\begin{equation}
\label{sigmagotosleep}
\mathbb{P}(\sigma>t) = \e^{-\chi t}, \qquad t\geq 0,
\end{equation}
while the wake-up time $\tau$ of a dormant lineage has law
\begin{equation}
\label{tauwakeup}
\mathbb{P}(\tau>t) = \sum_{m\in\N_0} \frac{K_me_m}{\chi}\,\e^{-e_m t}, \qquad t \geq 0,
\end{equation} 
where $K_me_m/\chi$ is the probability that a dormant lineage has colour $m$ (when the system is in equilibrium). It is possible to talk about paths in the forward time direction rather than about lineages in the backward time direction, but this requires an enrichment of the state spaces and the introduction of historical processes. We will not elaborate on this extension.

Recall that in the continuum limit the size of the active population is scaled to $1$. Note that
\begin{equation}
\label{taumean}
\mathbb{E}[\sigma] = \frac{1}{\chi}, \qquad \mathbb{E}[\tau] = \frac{\rho}{\chi}
\end{equation} 
with
\begin{equation}
\label{rhodef}
\rho = \sum_{m\in\N_0} K_m = \frac{\text{size dormant population}}{\text{size active population}}.
\end{equation}
We saw in \cite{GdHOpr1} that $\rho<\infty$ (`finite-size seed-bank') and $\rho=\infty$ (`infinite-size seed-bank') represent \emph{different regimes} for the long-time behaviour of the system (see Section~\ref{ss.coreinf} below). In particular, the criterion whether coexistence or clustering prevails is different. In the present paper we will see that also the scaling in the finite-systems scheme is different. 


\paragraph{Initial laws.}

We need to specify the law $\mu(0)$ from which the \emph{initial configuration} is drawn. Let $\CP(E)$ denote the set of probability measures on $E$. Define
\begin{equation}
\label{laws}
\begin{aligned}
\CT &= \big\{\mu\in\CP(E)\colon\, \mu \text{ is invariant under translations in } \G\big\},\\
\CT^\mathrm{erg} &= \big\{\mu\in\CT\colon\, \mu \text{ is ergodic under translations in } \G\big\},
\end{aligned}
\end{equation}
where translation stands for group action (recall \eqref{e409} and \eqref{e510}). For Model 1, as well as for Models 2 and 3 with $\rho<\infty$, $\mu(0)$ can be any element of $\CT^\mathrm{erg}$. However, for Models 2--3 with $\rho=\infty$ an extra restriction is needed, namely, we additionally require that $\mu(0)$ is \emph{colour regular}:
\begin{equation}
\label{covcond1}
\theta^\bullet_{\mu(0)} = \lim_{m\to\infty} \E_{\mu(0)}[y_{0,m}] \text{ exists.}
\end{equation}
This condition gives us control over the \emph{deep seed-banks}, which play a dominant role when $\rho=\infty$. (For $\mu(0) \in \CT \setminus \CT^\mathrm{erg}$ \eqref{covcond1} is replaced by the requirement that $\lim_{m\to\infty} \E_{\mu(0)}[y_{0,m} \mid \CB]$ exists $\mu(0)$-a.s.\ with $\CB$ the translation invariant sigma-field.) Accordingly, we define
\begin{equation}
\label{lawscr}
\begin{aligned}
\CT^\bullet &= \big\{\mu\in\CT\colon\,\mu \text{ is colour regular}\big\},\\
\CT^\mathrm{erg,\bullet} &= \big\{\mu\in\CT^\mathrm{erg}\colon\,\mu \text{ is colour regular}\big\}.
\end{aligned}
\end{equation}

\begin{remark}
\label{rem:topology}
{\rm Note that the set $\CT^\bullet$ is not closed in the weak topology as a subset of $\CT$. To remedy this problem, consider the set $\CT^\bullet \times [0,1]$ endowed with the metric that is the sum of the metric for weak convergence on $\CT$ and the Euclidean metric on $[0,1]$. Consider the subset $D = \{(\mu,\theta_\mu^\bullet)\colon\,\mu \in \CT^\bullet\}$. Clearly, $D$ has a countable dense subset, consisting of sequences of truncated seed-banks that are extended to infinite seed-banks by repeating the value $\theta_\mu^\bullet$ in the seed-bank direction. Since $D$ is also complete and metric, it is a Polish space in the stronger topology.} \hfill $\Box$ 
\end{remark}

We write $\mu(t)$ to denote the law evolved from $\mu(0)$ at time $t$, and further define
\begin{equation}
\label{lawinv}
\CI = \big\{\mu\in\CT\colon\, \mu \text{ is invariant under the evolution}\big\}.
\end{equation}


\paragraph{Topologies.}

On the set $\CT$ we use the topology of weak convergence. On the set $\CT^\bullet$, however, we need a stronger topology, which we call the \emph{topology uniform weak convergence} and which is defined on the subset 
\begin{equation}
\label{CTcov}
\CT^{\bullet,*}  = \Big\{\mu \in \CT^\bullet\colon\, \lim_{|m-m'| \to\infty} \mathrm{Cov}_\mu(y_{0,m},y_{0,m'}) = 0\Big\}.
\end{equation}
In \cite{GdHOpr1} we showed that 
\begin{equation}
\label{hatthetadef}
\hat{\theta} = \lim_{M\to\infty} \frac{x_0 + \sum_{m=0}^M K_my_{0,m}}{1+\sum_{m=0}^M K_m}
\end{equation}
exists in $L_2(\mu)$ for every $\mu \in \CT^{\bullet,*}$, which is the limiting density of $\heartsuit$. 


\paragraph{Examples.}

Three natural examples of Abelian groups and migration kernels are: 
\begin{itemize}
\item
$\G = \Z^d$ the Euclidean lattice of dimension $d$, and
\begin{equation}
\label{rwmzfv}
a(\cdot,\cdot) \text{ irreducible}, \quad \sum_{i \in \Z^d} i\,a(0,i) = 0, \quad \sum_{i \in \Z^d} |i|^2 a(0,i) < \infty. 
\end{equation}
Transient migration corresponds to $d \geq 3$ (see \cite[Section II.8]{Sp64}).
\item
$\G = \Z$, and
\begin{equation}
\label{rwtail}
a(\cdot,\cdot) \text{ symmetric}, \quad a(0,i) \sim Q\,|i|^{-1-q}, |i| \to \infty, \quad Q \in (0,\infty),\,q \in (0,2).  
\end{equation}
Transient migration corresponds to $q \in (0,1)$ (see \cite[Section II.8]{Sp64}). 
\item
$\G = \Omega_N$ the hierarchical group of order $N$ given by
\begin{equation}
\Omega_N = \Big\{i = (i_k)_{k\in\N_0} \in \{0,\ldots,N-1\}^{\N_0}\colon\, \sum_{k\in\N_0} i_k < \infty\Big\},
\end{equation} 
endowed with the hierarchical distance $d(i,j) = \inf\{k \in \N_0\colon\,i_l=j_l\,\forall\,l \geq k\}$, and 
\begin{equation}
\label{rwhier}
a(i,j) = \suml_{k \geq d(i,j)} \frac{c_{k-1}}{N^{k-1}}\frac{1}{N^k}, 
\quad i,j \in \Omega_N, i \neq j, \qquad a(i,i)=0, \quad i \in \Omega_N,
\end{equation}
where $(c_k)_{k \in \N_0}$ are coefficients satisfying $\limsup_{k \to \infty} \frac{1}{k}\,\log c_k < \log N$. The latter is the kernel of transition rates for the \emph{hierarchical random walk} that, for each $k \in \N$, at rate $c_{k-1}N^{-(k-1)}$ chooses the $k$-ball around it current location and moves to a uniformly random location in this ball. Transient migration corresponds to $\sum_{k\in \N_0} (1/c_k) < \infty$ (see \cite{DGW04a}, \cite{DGW05}).  
\end{itemize}


\paragraph{General diffusion function.}

In \eqref{gh1}, \eqref{gh1*} and \eqref{gh1**} we can replace the diffusion functions $dg_{\text{FW}}$, $d \in (0,\infty)$, with $g_{\text{FW}}(x) = x(1-x)$, $x \in [0,1]$, the Fisher-Wright diffusion function, by a general diffusion function in the class $\CG$ defined by 
\begin{equation}
\label{gh6}
\CG = \Big\{g\colon\,[0,1] \to [0,\infty)\colon\,g(0) = g(1) = 0, \, g(x) >0 \,\,\forall\,x \in (0,1), 
\, g \mbox{ Lipschitz}\Big\}.
\end{equation}
This class is appropriate because the diffusion stays confined to $[0,1]$, yet can go everywhere in [0,1] (\cite[Section 16.7]{B68}). Picking $g\neq d g_{\text{FW}}$ amounts to allowing the resampling rate to be \emph{state-dependent}, which is an important extension from a biological perspective. The resampling rate in state $x$ equals $g(x)/x(1-x)$, $x \in (0,1)$. An example is the Ohta-Kimura diffusion function $g(x)=[x(1-x)]^2$, $x \in [0,1]$, for which the resampling rate is equal to the genetic diversity of the colony \cite{ohtakimura1973}.


\paragraph{Trapping time for finite geographic space.}

Let
\begin{equation}
\label{e1233}
\begin{array}{lll}
&\text{Model 1}\colon &H = \inf\left\{t \geq 0\colon\,(x(t),y(t)) \in \{(0,0)^\G,(1,1)^\G\}\right\},\\[0.2cm]
&\text{Models 2--3}\colon\, &H = \inf\left\{t \geq 0\colon\,(x(t),y(t)) \in \{(0,0^{\N_0})^\G,(1,1^{\N_0})^\G\}\right\},
\end{array}
\end{equation}
be the time until trapping in one of the mono-type states. We note that if the geographic space $\G$ would be \emph{finite} and the seed-bank would be labeled by a \emph{finite} set rather than $\N_0$, then  
\begin{equation}
\label{trapcond}
\int_{[0,1]} \d x\,\,\frac{x(1-x)}{g(x)}<\infty \quad \Longrightarrow \quad
\P(H<\infty) =1.
\end{equation}
The integral criterion makes $\{0,1\}$ accessible for the Fisher-Wright diffusion $g_{\mathrm{FW}}$ with $g_{\mathrm{FW}}(x)=x(1-x)$, $x \in [0,1]$, but not for the Ohta-Kimura diffusion $g_{\mathrm{OK}}$ with $g_{\mathrm{OK}}(x)=(x(1-x))^2$, $x \in [0,1]$ (see \cite[Chapter 16, Section 7]{B68}).


\subsection{Core results for the infinite system}
\label{ss.coreinf}
 
In this section, which is largely copied from \cite[Section 3]{GdHOpr1}, we summarise the core results for the infinite system in order to set the stage for the results for the finite system that will be presented in Section~\ref{s.scaling}.   

In \cite{GdHOpr1} we showed that, for each of the three models introduced in Section~\ref{ss.process}, the system converges to a \emph{unique equilibrium} $\nu_\theta$ that depends on a single parameter $\theta$, namely, the density of type $\heartsuit$ in the population (active or dormant) under the initial law $\mu(0)$ (see \eqref{thetadef} and \eqref{thetadefalt}--\eqref{thetadefaltalt} below). This initial law is assumed to be an element of $\CT^\mathrm{erg}$ (recall \eqref{laws}) for $\rho<\infty$ and an element of $\CT^{\mathrm{erg},\bullet}$ (recall \eqref{lawscr}) for $\rho=\infty$. Moreover, we showed that $\nu_\theta$ exhibits a dichotomy of \emph{coexistence} (= locally multi-type equilibrium) versus \emph{clustering} (= locally mono-type equilibrium), and identified the \emph{parameter regimes} for which coexistence, respectively, clustering occurs. 


\paragraph{Regularity conditions.}
  
For $\rho=\infty$ we required certain \emph{regularity conditions} on top of Assumption~\ref{ass.migration}, which we list next.  

\begin{assumption}{\bf [Regularity conditions for $\rho=\infty$]} 
\label{ass.reginf}
$\mbox{}$\\
{\rm (i) The migration is \emph{symmetric}, i.e.,
\begin{equation}
\label{sym}
a(i,j) = a(j,i) \qquad \forall\, i,j \in \G,
\end{equation} 
and the time-$t$ transition kernel satisfies
\begin{equation}
\label{ass2}
t \mapsto a_t(0,0) \text{ is regularly varying at infinity}.
\end{equation} 
(ii) The seed-bank coeffcients are \emph{polynomial}, i.e.,
\begin{equation}
\label{ass3}
\begin{aligned}
&K_m \sim A\,m^{-\alpha}, \quad e_m \sim B\,m^{-\beta}, \quad m \to \infty,\\
&A,B \in (0,\infty), \quad \alpha,\beta \in \R\colon\,\alpha \leq 1 < \alpha +\beta.
\end{aligned}
\end{equation}
Consequently, 
\begin{equation}
\label{Ptail}
\mathbb{P}(\tau>t) \sim C\,t^{-\gamma}, \quad t\to\infty,
\end{equation}
with $\gamma = \frac{\alpha+\beta-1}{\beta}$ and $C = \frac{A}{\beta\chi}\,B^{1-\gamma}\,\Gamma(\gamma)$, where $\Gamma$ is the Gamma-function.}\hfill$\Box$
\end{assumption}

\noindent
Examples satisfying \eqref{ass2} can be found in \cite[Chapter 3]{H95}. The conditions on $\alpha,\beta$ in \eqref{ass3} guarantee that $\rho=\infty$, $\chi<\infty$.


\paragraph{Dichotomy.}

In \cite[Sections 5--7]{GdHOpr1} we showed that 
\begin{itemize}
\item
coexistence occurs if and only if the \emph{average total joint activity time} of two lineages in the dual \emph{without coalescence}, both starting from $(0,A)$, is finite,
\end{itemize} 
where joint activity means that the two lineages are active and are at the same site. (see also Appendix~\ref{appA*}). We further identified the \emph{parameter regime} for coexistence, namely, for $\rho<\infty$ we found that, subject to Assumption~\ref{ass.migration}, coexistence occurs if and only if 
\begin{equation}
\label{crfin}
I_{\hat{a}} = \int_1^\infty \hat{a}_t(0,0)\,\d t < \infty,
\end{equation} 
while for $\rho=\infty$, subject to Assumptions~\ref{ass.migration}--\ref{ass.reginf}, coexistence occurs if and only if
\begin{equation}
\label{crinf} 
I_{\hat{a},\gamma} = \int_1^\infty t^{-(1-\gamma)/\gamma} \hat{a}_t(0,0)\, \d t < \infty.
\end{equation} 
Here, $\hat{a}_t(0,0)$ is the probability that the migration with symmetrised kernel $\hat{a}(i,j)= \tfrac12[a(i,j)+a(j,i)]$, $i,j\in\G$, starting from $0$ is back at $0$ at time $t$, and $\gamma$ is the exponent in \eqref{Ptail}. For $\rho=\infty$ we showed that the claim may fail without the symmetry assumption $a(i,j)=a(j,i)$, $i,j\in\G$. Remarkably, \eqref{crfin} depends on the migration kernel only, while \eqref{crinf} depends on the migration kernel and on the asymptotics of the seed-banks coefficients $K_m$ and $e_m$. The diffusion function $g$ plays no role: either there is coexistence for all $g \in \CG$ or for no $g \in \CG$.   

Clearly, \eqref{crinf} shows that there is an interesting competition between migration and seed-bank, while a comparison of \eqref{crfin}--\eqref{crinf} shows that the seed-bank \emph{enhances genetic diversity}. The criterion in \eqref{crfin} corresponds to the symmetrised migration being \emph{transient}. The criterion in \eqref{crinf} (for symmetric migration) is less stringent and may even be met when the migration is \emph{recurrent}. For instance, it holds as soon as $\gamma \in (0,\tfrac12)$, irrespective of the migration (because $\hat{a}_t(0,0) \leq 1$ for all $t$), while for $\gamma \in [\tfrac12,1]$ it holds for certain classes of recurrent migration (see the examples below). 


\paragraph{Examples.}

To illustrate \eqref{crinf} we consider three examples.    
\begin{itemize}
\item
On $\G=\Z^d$, if $\hat{a}(\cdot,\cdot)$ satisfies \eqref{rwmzfv}, then $\hat{a}_t(0,0) \sim Ct^{-d/2}$, $t \to \infty$, for some $C \in (0,\infty)$ \cite[Chapter II, Section 7]{Sp64} and so \eqref{crinf} amounts to the requirement that
\begin{equation}
\label{gdcond}
\frac{1-\gamma}{\gamma} + \frac{d}{2} > 1,   
\end{equation}
i.e, $\gamma \in (0,1]$ when $d \geq 3$ (transient migration), $\gamma \in (0,1)$ when $d = 2$ (critically recurrent migration), and $\gamma \in (0,\tfrac23)$ when $d=1$ (strongly recurrent migration). 
\item
On $\G = \Z$, if $\hat{a}(\cdot,\cdot)$ satisfies \eqref{rwtail}, then $\hat{a}_t(0,0) \asymp t^{-1/q}$ \cite[Section 3.3.5]{H95}, and so \eqref{crinf} amounts to the requirement that 
\begin{equation}
\label{gdcondalt}
\frac{1-\gamma}{\gamma} + \frac{1}{q} > 1,
\end{equation}
i.e., $\gamma \in (0,1]$ when $q \in (0,1)$ (transient migration), $\gamma \in (0,1)$ when $q = 1$ (critically recurrent migration)  and $\gamma \in (0,\frac{q}{2q-1})$ when $q \in (1,2)$ (strongly recurrent migration). 
\item
On $\G = \Omega_N$, if $\hat{a}(\cdot,\cdot)$ satisfies \eqref{rwhier} with $c_k = c^k$, then $\hat{a}_t(0,0) \asymp t^{-1-\delta_N}$ \cite[Section 3.2]{GdHOpr3}, and so \eqref{crinf} amounts to the requirement that 
\begin{equation}
\label{gdcondaltalt}
\frac{1-\gamma_N}{\gamma_N} + \delta_N > 0,
\end{equation}
i.e., $\gamma_N \in (0,1]$ when $\delta_N \in (0,\infty)$ (transient migration), $\gamma_N \in (0,1)$ when $\delta_N = 0$ (critically recurrent migration), and $\gamma_N \in (0,\frac{1}{1-\delta_N})$ when $\delta_N \in (-\infty,0)$ (strongly recurrent migration). The inequality in \eqref{gdcondaltalt} holds if and only if $\log N \times \log(Kc) > \log c \times \log (K^2e)$ (recall that $c < N$ and $Ke < N$). When $Kc=1$, the latter condition holds for all $N$ if and only if $K>1$ and $Ke^2>1$ (recall that $K \geq 1$ is needed to ensure that $\rho=\infty$). For $N$ large enough it holds when $Kc > 1$ and fails when $Kc<1$.
\end{itemize}

\begin{remark}{{\bf [Dichotomy with modulation]}}
\label{r.modulate}
{\rm As shown in \cite[Sections 3.2 and 6.6]{GdHOpr1}, it is possible to modulate \eqref{Ptail}--\eqref{ass3} by slowly varying functions. Namely, if \eqref{Ptail} is replaced by $\mathbb{P}(\tau>t) \sim C\,\phi(t)\,t^{-\gamma}$, $t\to\infty$, with $\phi$ slowly varying at infinity, then $I_{\hat{a},\gamma}$ is replaced by
\begin{equation}
I_{\hat{a},\gamma,\phi} = \int_1^\infty \hat{\phi}(t)^{-1/\gamma}\,t^{-(1-\gamma)/\gamma} \hat{a}_t(0,0)\, \d t,
\end{equation} 
where $\hat{\phi}(t) = \phi(t)$ when $\gamma \in (0,1)$ and $\hat{\phi}(t) = \int_1^t \phi(s)\,s^{-1}\,\d s$ when $\gamma=1$.
} \hfill $\Box$
\end{remark}

\begin{remark}{{\bf [Dichotomy for hierarchical seed-bank]}}
{\rm As argued in \cite[Section 2]{GdHOpr3}, on $\G=\Omega_N$ it is natural to consider a \emph{hierarchical seed-bank}, which amounts to replacing $e_m$ by $e_m/N^m$ in \eqref{gh1*}--\eqref{gh2*} , and to assume that, instead of \eqref{ass3},
\begin{equation}
\label{asshier}
K_m = K^m, \quad e_m = e^m, \qquad m \in \N_0,\, K,e \in (0,\infty).
\end{equation}
(To guarantee that $\chi<\infty$ in \eqref{emcond} and $\rho=\infty$ in \eqref{rhodef}, we need $Ke < N$ and $K \geq 1$.) As shown in \cite[Section 3.2]{GdHOpr3}, the parameters in Remark~\ref{r.modulate} become 
\begin{equation}
\label{scal1}
\gamma_N = \frac{\log(N/Ke)}{\log(N/e)}, \quad \varphi(t) \asymp 1, \quad \hat\varphi(t) \asymp \left\{\begin{array}{ll}
1, &K \in (1,\infty),\\[0.2cm]
\log t, &K = 1,
\end{array}
\right.
\end{equation}
while, if $c_k = c^k$, $k \in \N_0$, $c \in (0,\infty)$, then
\begin{equation}
\label{scal2}
a_t(0,0) \asymp t^{-1-\delta_N},
\end{equation} 
where
\begin{equation}
\label{deltaNc}
\delta_N = \frac{\log c}{\log (N/c)}.
\end{equation}
(To guarantee that $\sum_{i\in\Omega_N} a(0,i) < \infty$ in \eqref{gh10}, we need $c<N$.) Note that $\gamma_N =1$ for all $N$ when $K=1$, while $\gamma_N<1$ for all $N$ when $K \in (1,\infty)$, but with $\gamma_N \to 1$ as $N\to\infty$. Further note that $\delta_N \to 0$ as $N\to\infty$ for all $c \in (0,\infty)$. Thus, the hierarchical mean-field limit $N\to\infty$ corresponds to a \emph{critically infinite mean wake-up time for the seed-bank} and a \emph{critically recurrent migration}.
} \hfill $\Box$
\end{remark}


\paragraph{Equilibria.}

In \cite{GdHOpr1} we also analysed the \emph{equilibria} of the infinite system, i.e., the  family of extremal invariant measures  
\begin{equation}
\label{equifam}
(\nu_\theta)_{\theta \in [0,1]}
\end{equation}
parametrised by the density of $\heartsuit$. This family depends on all the parameters of the model, i.e., the migration kernel, the seed-bank coefficients and  the diffusion function. We showed that $\nu_\theta$ is associated and mixing for all $\theta \in [0,1]$, and that $\E_{\nu_\theta}[x_0] = \theta$ and $\E_{\nu_\theta}[y_{0,m}] = \theta$ for all $m\in\N_0$ (so that $\nu_\theta$ is colour regular). We showed that for $\rho<\infty$ the deep seed-banks are \emph{random} under $\nu_\theta$ with a strictly positive variance, while (subject to \eqref{ass3} and \eqref{covcond1}) for $\rho=\infty$ the deep seed-banks are asymptotically \emph{deterministic} under $\nu_\theta$, i.e., converge in law to $\theta$ in every colony. 

We found that the seed-bank \emph{reduces volatility}, i.e., the active components have a smaller variance than in the model without seed-bank. We further showed that $\theta \mapsto \nu_\theta$ is continuous in the weak topology for Model 1 and for Model 2 with $\rho<\infty$, and  is continuous in the uniform weak topology for Model 2 with $\rho=\infty$.  


\paragraph{Domains of attraction.}

In \cite[Sections 5 and 6]{GdHOpr1} we showed that the sets $\CR^{(1)}_\theta$, $\CR^{(2)}_\theta$ and $\CR^{(2),\bullet}_\theta$ defined below are the \emph{domains of attraction} of $\nu_\theta$, $\theta \in [0,1]$, in Model 1, Model 2 with $\rho<\infty$ and Model 2 with $\rho=\infty$, respectively. See \cite[Definitions 5.6 and 6.5]{GdHOpr1}. We recall the notation
\begin{equation}
\begin{array}{lll}
&\text{Model 1}\colon &\S = \G \times \{A,D\},\\  
&\text{Model 2}\colon &\S = \G\times\{A,(D_m)_{m\in\N_0}\}, 
\end{array}
\end{equation}
and, for $u \in \S$,
\begin{equation}
\label{zudef}
z_u = \left\{\begin{array}{ll} 
x_i, &u=(i,A),\\
y_i, &u=(i,D),
\end{array}
\right. , \qquad 
z_u = \left\{\begin{array}{ll}
x_i, &u=(i,A),\\
y_{i,m}, &u=(i,D_m).
\end{array}
\right.
\end{equation}
We write $b^{(1)}_t(\cdot,\cdot)$ and $b^{(2)}_t(\cdot,\cdot)$ to denote the time-$t$ transition kernels of the dual (whose transition kernel is defined in \eqref{mrw1} and \eqref{mrw2}).

\begin{definition}{\bf [Liggett conditions]}
\label{Rtheta}
$\mbox{}$\\
{\rm (A) Model 1: $\CR^{(1)}_\theta$ is the set of measures $\mu \in \CT^{\mathrm{erg}}$ satisfying:
\begin{equation}
(\ast)_\theta
\begin{array}{lll}
&(1)\,\,\forall\, u_1 \in \S \colon\\
&\lim\limits_{t\to\infty} \sum\limits_{u_2 \in \S} 
b^{(1)}_t(u_1,u_2)\,\E_\mu[z_{u_2}] =\theta,\\[0.4cm]
&(2)\,\,\forall\,u_1,u_2 \in \S \colon\\
&\lim\limits_{t\to \infty} \sum\limits_{u_3,u_4 \in \S} 
b^{(1)}_t(u_1,u_3)\,b^{(1)}_t(u_2,u_4)\,\E_\mu[z_{u_3}z_{u_4}]=\theta^2.
\end{array}
\end{equation}
(B) Model 2 with $\rho<\infty$: $\mathcal{R}^{(2)}_\theta$ is the set of measures $\mu \in \CT^{\mathrm{erg}}$ satisfying:
\begin{equation}
(\ast)_\theta
\begin{array}{lll}
&(1)\,\,\forall\,u_1 \in \S\colon\\
&\lim\limits_{t\to \infty} \sum\limits_{u_2 \in \S} 
b^{(2)}_t(u_1,u_2)\,\E_\mu[z_{u_2}] =\theta,\\[0.4cm]
&(2)\,\,\forall\,u_1, u_2 \in \S\colon\\
&\lim\limits_{t\to \infty} \sum\limits_{u_3,u_4\in \S} 
b^{(2)}_t(u_1,u_3)\,b^{(2)}_t(u_2,u_4)\,\E_\mu[z_{u_3}z_{u_4}]=\theta^2.
\end{array}
\end{equation}
(C) Model 2 with $\rho=\infty$: $\mathcal{R}^{(2),\bullet}_\theta$ is the set
\begin{equation}
\CR_\theta^{(2),\bullet} = \big\{\mu \in \CR_\theta^{(2)}\colon\, \mu \text{ is colour regular}\big\}. 
\hfill\Box
\end{equation}
}
\end{definition}

\noindent
The Liggett conditions are equivalent to saying that, for $k=1,2$ and $u_1 \in \S$, $\sum_{u_2 \in \S} b^{(k)}_t(u_1,u_2)\,z_{u_2}$ converges in $L_2(\mu)$ to $\theta$ as $t\to\infty$.


\paragraph{Outline.}

Section~\ref{s.scaling} states our main theorems: a detailed description of the finite-systems scheme for Model 1, Model 2 with $\rho<\infty$ and Model 2 with $\rho=\infty$. Section~\ref{s.classpres} contains preparatory lemmas and observations that play a central role in the paper. Sections~\ref{s.fssrhofin}--\ref{s.fssrhoinffast} are devoted to the proofs of the main theorems. In Appendix~\ref{appA*} we show that the seed-bank reduces the volatility of the active components. In Appendix~\ref{appB*} we recall an \emph{abstract scheme} from \cite{CG94*} that lists general conditions for the existence of a finite-systems scheme, as well as a method from \cite{DGV95} to check some of these conditions in concrete settings. The proofs in Sections~\ref{s.fssrhofin}--\ref{s.fssrhoinfslow} rely on this abstract scheme. The proofs in Section \ref{s.fssrhoinffast} follow a separate route. Appendix~\ref{appC*} computes the asymptotic fraction of time spent in the active state by a lineage in the dual, which controls the various scaling regimes. In Appendix~\ref{appD*} we speculate about how in Model 2 with $\rho=\infty$ and fast-growing seed-bank the crossover from partial clustering to complete clustering may take place.


\section{Results: Finite-systems scheme}
\label{s.scaling}
 
This section contains our main results for the finite-system scheme. In Section~\ref{ss.setting} we provide the general setting. In Section~\ref{ss.model1} we focus on Model 1 (Theorems~\ref{T.finsys1} and \ref{thm:1657}). In Section~\ref{ss.model2} we prepare for Model 2 by introducing two regimes: slow growing seed-bank and fast-growing seed-bank. In Sections~\ref{ss.model2fin}--\ref{ss.model2inf} we analyse Model 2 for $\rho<\infty$ (Theorem~\ref{T.finsys2fin}), respectively, $\rho=\infty$ (Theorems~\ref{T.finsys2inf(1)}--\ref{T.finsys2inf(2)}). We will not consider Model 3 because it behaves similarly as Model 2. In Section~\ref{ss.open} we list a few open problems. 

 
\subsection{General setting} 
\label{ss.setting}
 
What does the dichotomy of the infinite system imply for \emph{large but finite systems}, for which clustering always prevails? In the \emph{coexistence regime} this question can be analysed by establishing what is called the \emph{finite-systems scheme} \cite{CG90}, \cite{CGSh95},  i.e., by identifying how a finite truncation of the system -- both in the geographic space and in the seed-bank -- behaves as both the time and the truncation level tend to infinity, properly tuned together. Our target will be to show that the finite system adopts the equilibrium of the infinite system but with a \emph{random density} of $\heartsuit$ (recall \eqref{hatthetadef}), and that the law of the latter evolves, on the proper time scale, as a diffusion on $[0,1]$ with a \emph{renormalised diffusion function}.     


\paragraph{Projective system.} 

Throughout the sequel the following assumption is in force on the geographic space:

\begin{assumption}{\bf [Profinite geographic space]}
\label{ass.profinite} 
{\rm $\G$ is \emph{profinite}, i.e., 
\begin{equation}
\label{profinite}
\begin{aligned}
&\text{$\G$ is (isomorphic to) the limit of a projective system $(\G_n)_{n\in\N}$}\\ 
&\text{of finite groups endowed with the discrete toplogy.}
\end{aligned}
\hfill\Box
\end{equation} 
}
\end{assumption}

\noindent
Key examples are
\begin{itemize}
\item
$\G=\Z^d$ and $\G_n = [-n,n]^d \cap \Z^d\, (\text{mod } 2n)$ the $n$-torus (viewed as a quotient group).  
\item
$\G=\Omega_N$ and $\G_n = (\Omega_N)_n$ the $n$-ball (viewed as a subgroup). 
\end{itemize}
We restrict the $\G$ to $\G_n$, i.e., we keep \eqref{gh1}--\eqref{gh2} and \eqref{gh1*}--\eqref{gh2*} but replace the migration kernel by
\begin{equation}
\label{perkernel}
a^n(i,j)=\sum_{ {k \in \G} \atop {k \downarrow j} } a(i,k), \qquad i,j \in \G_n,
\end{equation}
where $k \downarrow j$ means that $k$ is projected onto $j$. 

\begin{definition}{\bf [Mixing time]}
{\rm The \emph{mixing time} of the truncated migration is defined as the minimal sequence $(\psi_n)_{n\in\N}$ with $\lim_{n\to\infty} \psi_n = \infty$ such that
\begin{equation}
\label{pnmix}
 \lim_{s\to\infty} \limsup_{n\to\infty} \sup_{i \in \G_n} \Big|\, |\G_n|\, a^n_{s\psi_n}(0,i) - 1\Big| = 0, 
\end{equation}
where $a^n_t(\cdot,\cdot)$ is the time-$t$ transition kernel associated with \eqref{perkernel}.}\hfill$\Box$
\end{definition}

\noindent
In words, the migration on $\G_n$ mixes on time scale $\psi_n$. Note that 
\begin{equation}
a^n_t(0,i) \geq a_t(0,i) \qquad  \forall\,n \in \N,\, i \in \G_n,\, t \geq 0.
\end{equation} 
Also note that, for symmetric random walk (recall \eqref{sym}), 
\begin{equation}
\label{centralbds}
a^n_t(0,i) \leq a^n_t(0,0), \quad  a_t(0,i) \leq a_t(0,0), \qquad \forall\,n \in \N,\, i \in \G_n,\,t \geq 0.
\end{equation}

\begin{assumption}{\bf [Relations between the transitions kernels]}
\label{ass.trker}
$\mbox{}$\\
{\rm (i) Prior to the mixing time the random walk on $\G_n$ is close to the random walk on $\G$:
\begin{equation}
\label{afininfcompalt}
\lim_{n\to\infty} \sup_{t = o(\psi_n)} \sup_{i\in\G_n} \frac{a^n_t(0,i) - a_t(0,i)}{a_t(0,i)} = 0.
\end{equation}
(ii) Until the mixing time the random walk on $\G_n$ is comparable to the random walk on $\G$:  
\begin{equation}
\label{afininfcomp}
\exists\,C<\infty\colon\,C^{-1} a_t(0,i) \leq a^n_t(0,i) \leq C a_t(0,i) \quad \forall\,i \in \G_n\,\forall\, 0 \leq t \leq \psi_n.
\hfill\Box
\end{equation}
}
\end{assumption}

\begin{remark}{{\bf [Three examples of mixing time]}}
\label{r.1274}
$\mbox{}$
{\rm 
\begin{itemize}
\item[(1)]
On $\G = \Z^d$, for $\hat{a}(\cdot,\cdot)$ irreducible with $\sum_{i\in\Z^d} i\,\hat{a}(0,i) = 0$ and $\sum_{i\in\Z^d} |i|^2 \hat{a}(0,i)<\infty$, it is known that $\psi_n \asymp n^2$ (`diffusive mixing'; see \cite[Corollary 22.3]{BR76}, \cite[Proposition 2.8]{C89}, \cite[Lemma 2.1]{CGSh95}). In particular, $\psi_n \ll |\G_n|$ when $d \geq 3$, $\psi_n \asymp |\G_n|$ when $d=2$, and $\psi_n \gg |\G_n|$ when $d=1$. The crossover occurs at $d=2$ (critically recurrent migration). 
\item[(2)]
On $\G = \Z$, for $\hat{a}(\cdot,\cdot)$ symmetric with $\hat{a}(0,i) \sim Q |i|^{-1-q}$, $|i| \to \infty$, $q \in (0,2)$, $Q \in (0,\infty)$, it is known that $\psi_n \asymp n^\delta$ (the proof is an exercise in Fourier analysis). In particular, $\psi_n \ll |\G_n|$ when $q \in (0,1)$, $\psi_n \asymp |\G_n|$ when $q = 1$, and $\psi_n \gg |\G_n|$ when $q \in (1,2)$. The crossover occurs at $q = 1$ (critically recurrent migration).
\item[(3)]
On $\G = \Omega_N$, for the hierarchical random walk with coefficients $(c_k)_{k \in \N_0}$, it is known that $\psi_n \asymp  N^{n-1}/c_{n-1} = |\G_n|/Nc_{n-1}$ (the average time until the migration chooses horizon $n$). In particular, $\psi_n \ll |\G_n|$ when $\lim_{k\to\infty} c_k = \infty$, $\psi_n \asymp |\G_n|$ when $\lim_{k\to\infty} c_k \in (0,\infty)$, and $\psi_n \gg |\G_n|$ when $\lim_{k\to\infty} c_k = 0$. If $c_k = c^k$, then the crossover occur at $c=1$ (critically recurrent migration).
\hfill $\Box$
\end{itemize}
}
\end{remark} 

\begin{remark}{{\bf [Open questions]}}
\label{r.conj}
{\rm Three interesting questions in the general setting are: 
\begin{itemize}
\item[(I)] 
Is it true that $\psi_n = o(|\G_n|)$ for transient random walk?
\item[(II)] 
Is it true that $\hat{a}_{\psi_n}(0,0) \asymp |\G_n|^{-1}$?
\item[(III)]
Do \eqref{afininfcompalt}--\eqref{afininfcomp} hold always?
\end{itemize} 
The answer is yes for the three examples in Remark~\ref{r.1274}, but appears not to be known in general.} \hfill $\Box$
\end{remark}


\paragraph{Initial laws.}

Note that the state spaces for the finite system are 
\begin{equation}
\begin{array}{ll}
E_n = ([0,1] \times [0,1])^{\G_n} &\text{Model 1},\\
E_n = E_{M_n,n} = ([0,1] \times [0,1])^{M_n+1})^{\G_n} &\text{Models 2--3},
\end{array}
\end{equation} 
when the seed-bank is truncated at colour $M_n$. We need to decide how to choose the initial law $\mu_n$ on $\E_n$ in such a way that it properly links up with the initial law $\mu$ on $E$ (recall \eqref{Edef} and \eqref{Edef*}), which is assumed to be translation invariant (= invariant under the group action; recall \eqref{laws}). One way to go about is to let $\mu_n$ be the restriction of $\mu$ to $E_n$. This works perfectly well for $\G=\Omega_N$, because $\G_n=(\Omega_N)_n$ is a subgroup of $\G$, but does not work for $\G=\Z^d$, because $\mu_n$ is not translation invariant on $\G_n = \Z^d \cap [-n,n]^d$ (mod $2n$). A standard way out is to pick $\mu_n$ translation invariant such that $\mu_n$ converges to $\mu$ weakly as $n\to\infty$. A natural way to achieve this is by picking 
\begin{equation}
\label{choicelawn}
\mu_n((x^n,y^n)) = |\G_n|^{-1} \sum_{i\in\G_n} \mu(\tau_i(x^n,y^n)), \qquad (x^n,y^n) = (x_i(t),y_i(t))_{i \in \G_n},
\end{equation}
where $\tau_i$, $i \in \G_n$, is the \emph{shift} on $E_n$ defined by $(\tau_i (x^n,y^n))_j = (x^n,y^n)_{i+j}$, $i,j \in \G_n$. (We use the upper index $n$ to indicate that $(x^n,y^n)$ lives on $E_n$, but suppress this index from the single components.) 

Note that \eqref{choicelawn} reads as 
\begin{equation}
\label{choicelawnalt}
\mu_n = \mu \circ \Phi _n
\end{equation} 
with $\Phi_n\colon E \to E_n$ the translation invariant \emph{restriction operator} defined by $\Phi_n(x,y) = |\G_n|^{-1} \sum_{i\in\G_n} \tau_i(x^n,y^n)$. We will write \eqref{choicelawnalt} as $\mu_n = \Phi_n \mu$ by using the same symbol $\Phi_n$ for the restriction operator acting on $\CT$, and put
\begin{equation}
\label{trlinvn}
\CT_n = \{\Phi_n\mu\colon\,\mu \in \CT\}.
\end{equation}
Conversely, we write
\begin{equation}
\label{extop}
\mu = \mu_n \circ \tilde\Phi_n
\end{equation}
with $\tilde\Phi_n\colon E_n \to E$ the translation invariant \emph{extension operator} defined by $\tilde\Phi_n(x^n,y^n) = |\G_n|^{-1} \sum_{i\in\G_n} \tilde\tau_i(x^n,y^n)$, where $\tilde\tau_i$ is the periodic continuation of $\tau_i(x^n,y^n)$ from $E_n$ to $E$ in the geographic coordinate and the constant continuation of $\tau_i(x^n,y^n)$ from $\{0,\ldots,M_n\}$ to $\N_0$ in the seed-bank coordinate by the value
\begin{equation}
\frac{x_i+\sum_{m=0}^{M_n} K_m y_{i,m}}{1+\sum_{m=0}^{M_n} K_m}.
\end{equation}
We will write \eqref{extop} as $\mu = \tilde\Phi_n \mu$ by using the same symbol $\tilde\Phi_n$ for the extension operator acting on $\CT_n$.


\subsection{Model 1}
\label{ss.model1}


\paragraph{Set-up.}

We denote the system restricted to $E_n$ by
\begin{equation}
\label{gh16}
\big(x^n(t),y^n(t)\big)_{t \geq 0}.
\end{equation}
The \emph{empirical measure} of the system at time $t$ is defined by
\begin{equation}
\label{FdH1*}
\CE^n(t)
= \frac{1}{|\G_n|} \sum_{i \in \G_n} \delta_{\tau_i (x^n,y^n)}.
\end{equation}
We consider the empirical process 
\begin{equation}
\label{e1297}
\Big(\CE^n(s\beta_n)\Big)_{s \geq 0},
\end{equation}
where $\beta_n$ is a large time scale that needs to be properly chosen. We write $\mu_n(s\beta_n)$ to denote the law of $(x^n(s\beta_n),y^n(s\beta_n))$ when we choose $\mu_n(0)$ as in \eqref{choicelawn}.  


\paragraph{Macroscopic variable.} 

For large $n$ we expect that $\CE^n(s\beta_n)$ and $\mu^n(s\beta_n)$ are controlled by a random process on the set $\{\nu_\theta\colon\,\theta \in [0,1]\}$ of \emph{extremal invariant measures} of the infinite system that is determined by a \emph{macroscopic variable} $\hat\theta^n(s\beta_n)$ of the configuration on $E_n$, whose limit as $n\to\infty$ is a \emph{conserved quantity} of the infinite system, i.e., is itself a consistent estimator (recall \eqref{hatthetadef}). This macroscopic variable is given by
\begin{equation}
\label{e1311}
\hat\theta^n(t) = \frac{1}{|\G_n|} \sum_{i \in \G_n} \frac{x_i(t) + K y_i(t)}{1+K}, \qquad t \geq 0,
\end{equation}
which evolves as
\begin{eqnarray}
\label{mo2alt}
\d \hat\theta^n(t) = \frac{1}{|\G_n|\,(1+K)}
\sum_{i\in\G_n} \sqrt{g\big(x_i(t)\big)}\,\d w_i(t),
\end{eqnarray}
where the migration cancels out because of the averaging over the geographic space, and the exchange with the seed-bank cancels out because of the averaging over the seed-bank. Thus, in particular,
\begin{equation}
\label{martalt}
\big(\hat\theta^n(t)\big)_{t \geq 0} \text{ is a bounded and continuous martingale} 
\end{equation}
with increasing process
\begin{equation}
\label{incrpr1}
\left\langle \hat\theta^n(t)\right\rangle_{t \geq 0} 
= \left(\int_0^t \d s\,\frac{1}{|\G_n|\,(1+K)} \sum_{i \in \G_n} g\big(x_i(s))\big)\right)_{t \geq 0}.
\end{equation} 


\paragraph{Time scale.}

The goal is to identify $\beta_n$ such that ($\CL$ denotes probability law, convergence is always in the weak topology) 
\begin{equation}
\label{gh17}
\lim_{n\to\infty} \CL \left[\CE^n(s\beta_n)\right] 
= \lim_{n\to\infty} \CL\left[\mu^n(s\beta_n)\right] = \CL \left[ \nu_{\Theta(s)}\right], \qquad s > 0,
\end{equation}
and
\begin{equation}
\label{gh26}
\lim_{n\to\infty} \CL \left[\left(\hat \theta^n (s \beta_n)\right)_{s > 0}\right]
= \CL\left[\big(\Theta(s)\big)_{s > 0}\right],
\end{equation}
where along the way we also need to identify $(\Theta(s))_{s \geq 0}$ as a \emph{Markov process} on $[0,1]$ (via an appropriate martingale problem or an associated SDE), and show that $\Theta(0)=\theta$ and $\Theta(\infty) \in \{0,1\}$. Accordingly, we want to show that 
\begin{equation}
\label{gh17alt}
\lim_{n\to\infty} \CL \left[\CE^n(\tilde\beta_n)\right] 
= \lim_{n\to\infty} \CL\left[\mu^n(\tilde\beta_n)\right] = \nu_\theta
\end{equation}
for any time scale $\tilde\beta_n \ll \beta_n$, and 
\begin{equation}
\label{gh17altalt}
\lim_{n\to\infty} \CL \left[\CE^n(\tilde\beta_n)\right] 
= \lim_{n\to\infty} \CL\left[\mu^n(\tilde\beta_n)\right] = \theta\,\delta_{(1,1)^\G} + (1-\theta)\,\delta_{(0,0)^\G}
\end{equation}
for any time scale $\tilde\beta_n \gg \beta_n$. In the above scenario, the local configuration is controlled by the macroscopic variable in \eqref{e1311}, in the sense that it approaches the corresponding equilibrium of the \emph{infinite} system locally. The macroscopic variable itself follows an \emph{autonomous} diffusion process, and \emph{conditional} on $\hat \theta^n(s\beta_n)$ being close to some $\theta' \in [0,1]$ for some $s>0$, the system converges in distribution to $\nu_{\theta'}$. 


\paragraph{Scaling limit.}

In the model without seed-bank \cite{CGSh95}, $(\Theta(s))_{s \geq 0}$ turned out to be the diffusion 
\begin{equation}
\label{gh18a}
\d \Theta(s) = \sqrt{(\CF g)(\Theta(s)}\,\d w(s), \qquad \Theta(0) = \theta, 
\end{equation}
with diffusion function $\CF g$ given by 
\begin{equation}
\label{gh19a}
(\CF g)(\theta) = \E_{\nu_\theta}[g(x_0)] = \int_E g(x_0)\,\nu_\theta(\d x), \qquad \theta \in [0,1],
\end{equation}
and that if $g\in\CG$, then also $\CF g \in \CG$. For the special case where $g=dg_{\text{FW}}$, $d \in (0,\infty)$, it turned out that $\CF g = d^\ast g_{\text{FW}}$ with $d^\ast=d/(d+\hat G(0,0))$, where $\hat G(0,0)$ is the Green function at the origin of the symmetrised migration kernel $\hat a(\cdot,\cdot)$. In the presence of the seed-bank the role of $\hat G(0,0)$ is taken over by the quantity
\begin{equation}
\begin{array}{lll}
\hat{B}(0,0) &= &\text{the mean total joint activity time of two independent Markov processes}\\ 
&&\text{with transition kernel $b^{(1)}(\cdot,\cdot)$ in \eqref{mrw1} both starting from $(0,A)$}.
\end{array} 
\end{equation}
The key observation is that $\hat{B}(0,0) < \infty$ if and only if the system is in the \emph{coexistence regime}. In Appendix~\ref{appA*} we will see that $\CF$ is non-trivial \emph{only} in the coexistence regime, and that throughout this regime $\CF g \in \CG$ for all  $g\in\CG$. (In the clustering regime $\CF g \equiv 0$ for all $g\in\CG$, because the equilibrium lives on the states $x^n = y^n \equiv 1$ and $x^n = y^n \equiv 0$.)  


\paragraph{Results.}

Abbreviate
\begin{equation}
\label{kappadef}
\kappa = (1+K)^2.
\end{equation}
Recall \eqref{laws}. For $\mu \in \CT^{\mathrm{erg}}$, define
\begin{equation}
\label{thetadef}
\theta_\mu = \E_\mu\left[ \frac{x_0 + Ky_0}{1+K}\right]
\end{equation}
and, for $\theta \in [0,1]$, 
\begin{equation}
\label{erglaw1}
\CT^{\mathrm{erg}}_\theta = \left\{\mu\in\CT^{\mathrm{erg}}\colon\,\theta_\mu = \theta\right\}.
\end{equation}

\begin{theorem}{{\bf [Finite-systems scheme: Model 1]}}
\label{T.finsys1}
Suppose that Assumptions~\ref{ass.migration}, \ref{ass.profinite} and \ref{ass.trker} are in force. Suppose that $\mu(0) \in \CT^\mathrm{erg}_\theta$, and that $\mu_n(0)$ is given by \eqref{choicelawn}. Then, in the coexistence regime, the following are true:
\begin{itemize}
\item[{\rm (a)}] (Convergence on macroscopic time scale) For every $s>0$, \eqref{gh17}--\eqref{gh17altalt} hold with time scale
\begin{equation}
\label{e1347}
\beta_n = \kappa |\G_n|
\end{equation}
and $(\Theta(s))_{s \geq 0}$ the diffusion on $[0,1]$ given by
\begin{equation}
\label{gh18}
\d \Theta(s)=\sqrt{(\CF g)(\Theta(s))}\,\d w(s), \qquad \Theta(0)=\theta,
\end{equation}
with $\CF g$ given by (recall \eqref{equifam})
\begin{equation}
\label{gh19}
(\CF g)(\theta)=E_{\nu_\theta}[g(x_0)] = \int_E g(x_0)\,\nu_\theta(\d x,\d y), \qquad \theta \in [0,1].
\end{equation}
\item[{\rm (b)}] (Fisher-Wright diffusion) 
If $g= dg_{\text{FW}}$, $d \in (0,\infty)$, then
\begin{equation}
\label{gh20}
\CF g = d^\ast g_{\text{FW}} 
\end{equation}
with $d^\ast = d/(1+d\hat{B}(0,0))$.
\end{itemize}
\end{theorem}

\noindent
Thus, we find the same behaviour as in the system without seed-bank, except that the macroscopic time scale $\beta_n$ runs slower by a factor $\kappa=(1+K)^2$. One factor $1+K$ arises from the time that is lost in the seed-bank, while the other factor $1+K$ arises from the fact that the seed-bank brings in no volatility during the time that is lost (see \eqref{incrpr1}, and Section \ref{s.fssrhofin} for further details). Fig.~\ref{fig:WFren} draws a typical realisation of $(\Theta(s))_{s \geq 0}$. 

\begin{figure}[htbp]
\centering 
\includegraphics[width=.4\textwidth]{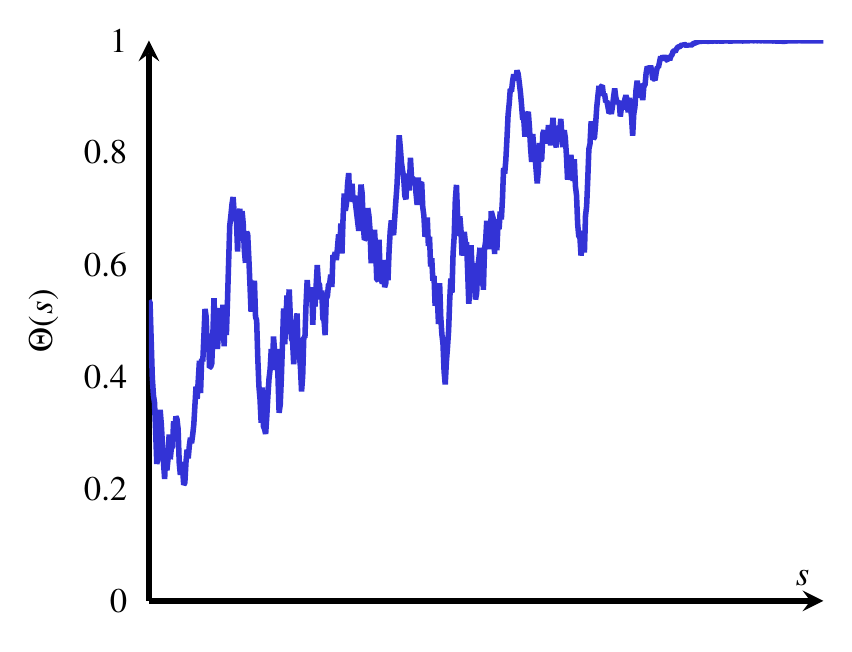}
\vspace{0.2cm}
\caption{\small Picture of the diffusion $(\Theta(s))_{s \geq 0}$. The boundary set $\{0,1\}$
is hit in finite time when $0$ and $1$ are accessible, which is the case for instance when 
$\CF g$ is a multiple of the Fisher-Wright diffusion function. The blue curve is the path of
the diffusion.}
\label{fig:WFren}
\end{figure}

\begin{remark}{{\bf [Clustering regime not universal]}}
\label{r.1267}
{\rm Theorem~\ref{T.finsys1} \emph{only holds in the coexistence regime}, because it captures on what time scale the finite system feels the boundary and begins to cluster. To derive an analogue of Theorem~\ref{T.finsys1} in the clustering regime we would need to compare the different modes of clustering, which would be more difficult (see \cite{CG91} for systems without seed-bank on $\G=\Z^2$). The outcome would \emph{not be as universal} and would depend on the \emph{degree} of recurrence of the migration kernel.} \hfill $\Box$
\end{remark}


\paragraph{Comments.}

The scaling of time by $|\G_n|$ is natural when we think of the dual. On $\G_n$, two random walks on average need $|\G_n|$ moves until they meet. Indeed, the coexistence regime corresponds to transient migration (recall \eqref{crfin}), for which the mixing time is $\psi_n=o(|\G_n|)$ (recall Remarks~\ref{r.1274} and \ref{r.conj}(1)). Hence, at time $s\beta_n=s\kappa|\G_n|$ the two random walks are more or less uniformly distributed on $\G_n$, and have probability $|\G_n|^{-1}$ to be at the same site. If we consider two Markov processes with transition kernel given by \eqref{mrw1}, then we get similar behaviour and the \emph{slow-down factor} $1/\kappa$ in \eqref{kappadef} is the \emph{fraction of time that in the dual two lineages are jointly active}. This is natural because only active lineages in the dual can move and coalesce (i.e., forward in time only active individuals can move and exchange genetic information).   

The reason why $\CF g$ appears in \eqref{gh18} is that, given the value of the macroscopic variable $\hat{\theta}^n$, the $|\G_n|$ constituent components equilibrate on a time scale that is fast with respect to the time scale $\beta_n$ on which $\hat{\theta}^n$ fluctuates. Consequently, the volatility of $\hat{\theta}^n$ is close to the expectation of the volatility of the constituent components in the quasi-equilibrium $\nu_{\hat{\theta}^n}$, as expressed by \eqref{gh19} (see Fig.~\ref{fig-renorm}).
 
We may think of $\CF$ as a \emph{renormalisation map} acting on the class  $\CG$ of diffusion functions defined in \eqref{gh6}. According to \eqref{gh19}, $\CF$ is a \emph{non-linear integral transform}, with the non-linearity arising from the fact that $\nu_\theta$ depends on $g$ (apart from $e,K$). In general no explicit formula is available for $\CF$. However, the result in \eqref{gh20} says that if $g$ is a multiple of the Fisher-Wright diffusion, then so is $\CF g$, in which case $\{0,1\}$ is accessible at both ends. 

\vspace{0.5cm}
\begin{figure}[htbp]
\begin{center}
\setlength{\unitlength}{0.4cm}
\begin{picture}(16,6)(-7,-1)
{\thicklines
\qbezier(-8,0)(-4,2)(0,4)
\qbezier(-6,0)(-3,2)(0,4)
\qbezier(-4,0)(-2,2)(0,4)
\qbezier(4,0)(2,2)(0,4)
\qbezier(6,0)(3,2)(0,4)
\qbezier(8,0)(4,2)(0,4)
} 
\put(-8,0){\circle*{.25}}
\put(-6,0){\circle*{.25}}
\put(-4,0){\circle*{.25}}
\put(4,0){\circle*{.25}}
\put(6,0){\circle*{.25}}
\put(8,0){\circle*{.25}}
\put(0,4){\circle*{.25}}
\put(-.4,0){$\dots$}
\put(-8.3,-1){$1$}
\put(-6.2,-1){$2$}
\put(4.5,-1){$|\G_n|-1$}
\put(7.6,-1){$|\G_n|$}
\put(0,4.7){\small macroscopic variable}
\end{picture}
\end{center}
\caption{Pictorial representation of the averaging procedure behind the renormalisation map. The volatility of the empirical average of the $|\G_n|$ components converges as $n\to\infty$ to the average of the volatility of the single components with respect to their limiting equilibrium distribution.}
\label{fig-renorm}
\end{figure}
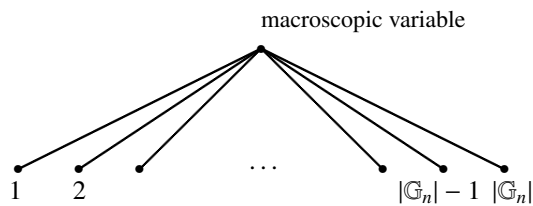

\begin{remark}{{\bf [Hierarchical mean-field limit]}}
\label{r.1261}
{\rm In \cite{GdHOpr3} we consider the case where the geographic space $\G$ is the hierarchical group $\Omega_N$. In the so-called \emph{hierarchical mean-field limit} $N\to\infty$ we are able to analyse $\CF$ in detail. Without seed-bank the \emph{mean-field finite-systems scheme} on $\Omega_N$ was established in \cite{DG93b}, \cite{DGV95}.} \hfill $\Box$
\end{remark}


\paragraph{Trapping time.}
 
Because $(\Theta(s))_{s \geq 0}$ is a bounded martingale, $\lim_{s\to\infty} \Theta(s)$ exists $\P$-a.s. Since $\CF g \in \CG$, the limit lies in $\{0,1\}$. Let
\begin{equation}
\mathfrak{T}=\inf\{s \geq 0\colon\,\Theta(s) \in \{0,1\}\}.
\end{equation} 
If $\{0,1\}$ is \emph{accessible} at both ends, i.e., 
\begin{equation}
\label{e1675}
\int_{[0,1]} \d x\,\, \frac{x(1-x)}{(\CF g)(x)} <\infty,
\end{equation} 
then $\P(\mathfrak{T}<\infty)=1$ (see Fig.~\ref{fig:WFren}). We will show that the hitting time of the traps in the finite system
\begin{equation}
H_n = \inf\left\{t \geq 0\colon\,(x^n(t),y^n(t)) \in \{(0,0)^{\G_n},(1,1)^{\G_n}\}\right\} 
\end{equation}
(compare with \eqref{e1233}) satisfies the following.

\begin{theorem}{{\bf [Scaling of hitting time: Model 1]}}
\label{thm:1657}
Under the conditions in Theorem~\ref{T.finsys1}, for all diffusion functions $g \in \CG$ satisfying~\eqref{e1675},
\begin{equation}
\label{e1437}
\lim_{n\to\infty} \CL [H_n/\beta_n] = \CL [\mathfrak{T}].
\end{equation}
\end{theorem}

\noindent 
Note that Theorem~\ref{thm:1657} does \emph{not} follow directly from the path convergence in \eqref{gh26}. The relation in \eqref{e1437} implies that at time $t=[1+o(1)]\,\mathfrak{T}\beta_n$ the finite system \emph{clusters}. If, on the other hand, $\{0,1\}$ is \emph{not accessible} at both ends, then $\mathfrak{T}=\infty$ and the finite system only approximately clusters on time scale $\beta_n$.

 
\subsection{Model 2: Scaling}
\label{ss.model2}


\paragraph{Truncation.}

In Model 2 the seed-bank is truncated from $\N_0$ to $\{0,\ldots,M_n\}$ with 
\begin{equation}
\label{e1702}
\lim_{n\to\infty} M_n = \infty.
\end{equation}
Different choices for $M_n$ make sense in biological applications, since they capture \emph{how many seeds can accumulate in a single colony}. The truncated state space is $E_n=E_{M_n,n} = ([0,1] \times [0,1]^{M_n+1})^{\G_n}$, the truncated system 
\begin{equation}
\big(x^{M_n,n}(t),y^{M_n,n}(t)\big) = \big(x_i(t),(y_{i,m}(t))_{0 \leq m \leq M_n}\big)_{i\in\G_n}
\end{equation} 
evolves as
\begin{eqnarray}
\label{gh1b*}
\d x_i(t) &=& \sum_{j \in \G_n} a^n(i,j)\,[x_j (t) - x_i(t)]\,\d t
+ \sqrt{g(x_i(t))}\,\d w_i(t)\\ \nonumber
&&+\,\sum_{m=0}^{M_n} K_me_m\, [y_{i,m}(t) - x_i(t)]\,\d t,\\[0.2cm]
\label{gh2b*}
\d y_{i,m} (t) &=& e_m\,[x_i(t) - y_{i,m}(t)]\,\d t, \qquad 0\leq m \leq M_n,
\end{eqnarray}
where $a^n(i,j)$ is defined in \eqref{perkernel}. 


\paragraph{Truncated macroscopic variable.}

The analogue of \eqref{e1311} is the \emph{macroscopic variable} 
\begin{equation}
\label{mo3}
\hat{\theta}^{\,M_n,n}(t) = \frac{1}{|\G_n|} \sum_{i\in\G_n} 
\frac{x_i(t)+\sum_{m=0}^{M_n} K_m y_{i,m}(t)}{1+\sum_{m=0}^{M_n} K_m},
\end{equation}
the analogue of \eqref{FdH1*} is the \emph{empirical measure} 
\begin{equation}
\label{e1721}
\CE^{M_n,n}(t) = \frac{1}{|\G_n|} \sum_{i \in \G_n} \delta_{\tau_i (x^{M_n,n}(t),y^{M_n,n}(t))},
\end{equation}
and we write $\mu^{M_n,n}(t)$ to denote the law of the truncated system at time $t$. 

In order to understand the behaviour of the macroscopic variable $\hat{\theta}^{\,M_n,n}(t)$ defined in \eqref{mo3}, we must keep track of the variance of the average of the active components, and look for a time scale in which this variance remains positive in the limit as $n \to \infty$. By \eqref{gh1b*} and \eqref{gh2b*}, we find that $\hat{\theta}^{M,n}(t)$ evolves according to the SDE
\begin{eqnarray}
\label{mo2}
\d \hat{\theta}^{\,M_n,n}(t) = \frac{1}{|\G_n|\,(1+\sum_{m=0}^{M_n}K_m)}
\sum_{i\in\G_n} \sqrt{g\big(x_i(t)\big)}\,\d w_i(t),
\end{eqnarray}
where the migration cancels out because of the averaging over the geographic space, and the exchange with the seed-bank cancels out because of the averaging over the seed-bank. Thus, in particular,
\begin{equation}
\label{mart}
\big(\hat{\theta}^{\,M_n,n}(t)\big)_{t \geq 0} \text{ is a bounded and continuous martingale}
\end{equation} 
with increasing process
\begin{equation}
\label{incpr}
\left\langle \hat{\theta}^{\,M_n,n}(t)\right\rangle_{t \geq 0} 
= \left(\int_0^t \d s\,\frac{1}{|\G_n|\,(1+\sum_{m=0}^{M_n}K_m)} \sum_{i \in \G_n} g\big(x_i(s))\big)\right)_{t \geq 0}.
\end{equation}


\paragraph{Time scale.}

To identify the proper time scale, we speed up time. Using the scaling properties $W(a^2t) \cong a W(t)$ and $\sqrt{a}\,W(t)+\sqrt{b}\,W^\prime(t)\cong\sqrt{a+b}\,W^{\prime\prime}(t)$, and putting
\begin{equation}
\label{betaKndef}
\beta_n = \kappa_{M_n} |\G_n|, \qquad \kappa_{M_n} = \left(1+\sum_{m=0}^{M_n} K_m\right)^2, 
\end{equation}
we see that \eqref{mo2} becomes
\begin{equation}
\label{e1732}
\d \hat{\theta}^{\,M_n,n}(s\beta_n)
= \sqrt{\frac{1}{|\G_n|}\sum_{i\in\G_n} g\big(x_i(s\beta_n)\big)}\,\d w_i(s).
\end{equation}
Thus, important quantities for the evolution of \eqref{mo3} are the \emph{active macroscopic variable} 
\begin{equation}
\label{act}
\hat{\theta}_x^{\,M_n,n}(t) = \frac{1}{|\G_n|} \sum_{i\in\G_n} x_i(t)
\end{equation}
and the \emph{active empirical measure}
\begin{equation}
\CE_x^{\,M_n,n}(t) = \frac{1}{|\G_n|} \sum_{i \in \G_n} \delta_{\tau_i (x^{M_n,n})}.
\end{equation}

Recall \eqref{gh18a}--\eqref{gh19a}. We might expect that again
\begin{equation}
\lim_{n\to\infty} \mathcal{L}\left[\big(\hat{\theta}^{\,M_n,n}(s\beta_n)\big)_{s\geq0}\right]
= \mathcal{L}\left[(\Theta(s))_{s\geq0}\right] 
\end{equation}
and
\begin{equation}
\label{e1748}
\lim_{n\to\infty}\frac{1}{|\G_n|} \sum_{i\in\G_n} g\big(x_i(s\beta_n)\big)
= (\mathcal{F}g)\big({\Theta}(s)\big).
\end{equation}
Moreover, we might expect that again 
\begin{equation}
\lim_{n\to\infty} \CL\left[\big(\mu^{M_n,n}(s\beta_n)\big)_{s \geq 0}\right] = \CL[(\nu_{\Theta(s)})_{s \geq 0}],
\end{equation}
i.e., at time $s\beta_n$ the law of the finite system converges as $n\to\infty$ to a \emph{random mixture} of equilibria for the infinite system, selected through the random variable $\Theta(s)$. 

We will see in Sections~\ref{ss.model2fin}--\ref{ss.model2inf} that the above heuristics is correct for $\rho<\infty$, and also for $\rho=\infty$ when the growth of $M_n$ is slow, but fails for $\rho=\infty$ when the growth of $M_n$ is fast. 


\subsection{Model 2: $\rho<\infty$}
\label{ss.model2fin}
  

\paragraph{Results.}

For $\rho < \infty$ the behaviour is qualitatively similar to what we saw in Model 1. Let $\mu(0) \in \CT^\mathrm{erg}$. For $\mu \in \CT^{\mathrm{erg}}$, define
\begin{equation}
\label{thetadefalt}
\theta_\mu = \E_\mu\left[\frac{x_0 + \sum_{m\in\N_0} K_m\,y_{0,m}}{1+\rho}\right]
\end{equation}
and, for $\theta \in [0,1]$,
\begin{equation}
\label{erglaw2fin}
\CT^{\mathrm{erg}}_\theta = \left\{\mu \in \CT^\mathrm{erg}\colon\,\theta_\mu = \theta\right\}.
\end{equation}

\begin{theorem}{{\bf [Finite-systems scheme: Model 2 with $\rho<\infty$]}}
\label{T.finsys2fin}
Suppose that Assumptions~\ref{ass.migration}, \ref{ass.profinite} and \ref{ass.trker} are in force. Suppose that $\mu(0) \in \CT^\mathrm{erg}_\theta$, and that $\mu_n(0)$ is given by \eqref{choicelawn}. Put
\begin{equation}
\label{e2033alt}
\beta_n = \kappa |\G_n|, \qquad \kappa=(1+\rho)^2.
\end{equation}
Then, in the coexistence regime, for any $M_n$ satisfying \eqref{e1702}, the same formulas as in \eqref{gh17}--\eqref{gh17altalt} and \eqref{gh18}--\eqref{gh19} apply, with $(0,0)^\G$ and $(1,1)^\G$ replaced by $(0,0^{\N_0})^\G$ and $(1,1^{\N_0})^\G$ in \eqref{gh17altalt}.
\end{theorem}


\paragraph{Comments.}

Theorem~\ref{T.finsys2fin} shows that for $\rho<\infty$ the behaviour is similar as in Theorem~\ref{T.finsys1}. The macroscopic time scale $\beta_n$ runs slower by a factor $\kappa=(1+\rho)^2$, with one factor $1+\rho$ coming from the time that is lost in the seed-bank and the other factor $1+\rho$ coming from the fact that the seed-bank brings in no volatility during the time that is lost (again see Section \ref{s.fssrhofin} for further details). The scaling of time by $|\G_n|$ is again natural, because for $\rho<\infty$ the coexistence regime again corresponds to transient migration (recall \eqref{crfin}), for which the mixing time of the migration is $o(|\G_n|)$ (recall Remarks~\ref{r.1274}--\ref{r.conj}). The slow-down factor $1/\kappa$ in \eqref{kappadef} is again the fraction of time the two lineages are jointly active.


\paragraph{Trapping time.}

Theorem~\ref{thm:1657} carries over verbatim.


\subsection{Model 2: $\rho=\infty$}
\label{ss.model2inf}

For $\rho=\infty$, the situation is much more delicate. It turns out that there are \emph{two regimes}, which we define in Section~\ref{sss.tworeg} and refer to as \emph{slow growing} seed-bank and \emph{fast growing} seed-bank. Here, slow and fast refer to the time scale on which two active lineages in the dual coalesce, which depends on $|\G_n|$ and on $\gamma$, the exponent of the tail of the wake-up time defined in \eqref{Ptail}. We identify the finite-systems scheme in these regimes in Sections~\ref{sss.slowgrow}, respectively, \ref{sss.fastgrow}, and find that they exhibit completely different behaviour. For slow growing seed-bank there is a \emph{gradual approach} towards complete clustering as for $\rho<\infty$, while for fast growing seed-bank there is a \emph{bursty approach} towards complete clustering, with \emph{random switches} between partially clustered states, driven by dormant lineages that gradually wake up on time scales that exceed the time scale on which two active lineages in the dual coalesce.   

For $\mu \in \CT^{\mathrm{erg},\bullet}$ (recall \eqref{laws}), define
\begin{equation}
\label{thetadefaltalt}
\theta_\mu = \lim_{M\to\infty}  \E_\mu\left[\frac{x_0 + \sum_{m=0}^M K_m\,y_{0,m}}
{1+\sum_{m=0}^M K_m}\right].
\end{equation}
and, for $\theta \in [0,1]$,
\begin{equation}
\label{erglaw2inf}
\CT^{\mathrm{erg},\bullet}_\theta = \left\{\mu \in \CT^{\mathrm{erg},\bullet}\colon\,\theta_\mu = \theta\right\}.
\end{equation}
The limit in \eqref{thetadefaltalt} exists because $\mu(0)$ is colour regular (recall \eqref{covcond1}). 


\subsubsection{Two growth regimes}
\label{sss.tworeg}

Recall \eqref{Ptail}--\eqref{ass3}. We introduce \emph{two time scales}: 
\begin{equation}
\label{scales}
\beta^*_n = M_n^{\,\beta}, \qquad \beta^{**}_n = \left\{\begin{array}{ll}
|\G_n|^{1/(2\gamma-1)}, &\gamma \in (\tfrac12,1],\\[0.2cm]
\e^{|\G_n|}, &\gamma = \tfrac12,\\[0.2cm]
\infty, &\gamma \in (0,\tfrac12).     
\end{array}
\right.
\end{equation}  
The interpretation of these time scales is as follows (recall \eqref{lineages1}--\eqref{lineages2}):
\begin{itemize}
\item
$\beta^*_n$ is the time scale on which a single dormant lineage in the dual starting from the deepest seed-bank (colour $M_n$) becomes active.   
\item
$\beta^{**}_n$ is the time scale on which two active lineages in the dual coalesce (on the active layer $\G_n$). 
\end{itemize}
Note that $\beta^*_n$ only depends on the size of the seed-bank, while $\beta^{**}_n$ only depends on the size of the geographic space and the exponent of the tail of the wake-up time.  

The two regimes of interest are
\begin{equation}
\label{Mncond}
\begin{array}{lllll}
&{\rm (I)} &\beta^*_n \ll \beta^{**}_n &\text{ (`slow growth')},\\[0.2cm]
&{\rm (II)} &\beta^{**}_n \ll \beta^*_n &\text{ (`fast growth')}.
\end{array}
\end{equation}
(For the critical value $\gamma=\tfrac12$, the two regimes are actually $\log \beta^*_n \ll \log \beta^{**}_n$, respectively, $\log \beta^{**}_n \ll \log \beta^*_n$. Note that for $\gamma \in (0,\tfrac12)$ only regime (I) is possible.) In regime (I) the system behaves as if $\rho<\infty$ (see Section~\ref{s.fssrhoinfslow}). In regime (II) the system essentially behaves like a \emph{hybrid} system in which the geographic space is truncated but the seed-bank is not (see Section~\ref{s.fssrhoinffast}). In the crossover regime 
\begin{equation}
\label{Mncondcrit}
\beta^*_n \asymp \beta^{**}_n,
\end{equation} 
which we will not consider, the behaviour is more involved.

We will see that the two regimes exhibit different behaviour. In other words, the limits $|\G_n| \to \infty$ and $M_n \to \infty$ \emph{cannot be interchanged}. Note that both in regime (I) and regime (II) the time scale $\kappa_{M_n}|\G_n|$ falls between the time scales $\beta^{*}_n$ and $\beta^{**}_n$.


\subsubsection{Slow growing seed-bank}
\label{sss.slowgrow}


\paragraph{Results.}

In regime (I) the behaviour is similar as for $\rho<\infty$ except for a change of time scale.

\begin{theorem}{{\bf [Finite-systems scheme: Model 2 with $\rho=\infty$ in regime (I)]}}
\label{T.finsys2inf(1)}
Suppose that Assumptions~\ref{ass.migration}, \ref{ass.reginf}, \ref{ass.profinite} and \ref{ass.trker} are in force. Suppose that $\mu(0) \in \CT^\mathrm{erg,\bullet}_\theta$, and that $\mu_n(0)$ is given by \eqref{choicelawn}. Put 
\begin{equation}
\beta_n = \kappa_{M_n} |\G_n|.
\end{equation}
Then, in the coexistence regime, the same formulas as in \eqref{gh17}--\eqref{gh17altalt} and \eqref{gh18}--\eqref{gh19} apply.
\end{theorem}


\paragraph{Comments.}

Theorem~\ref{T.finsys2inf(1)} shows that for $\rho=\infty$ the macroscopic time scale $\beta_n$ needs to be speeded up in order to see fluctuations of the macroscopic variable. The fraction of time that two lineages in the dual are jointly active scales like $1/\kappa_{M_n}$. Hence the time scale must be multiplied by $\kappa_{M_n}$, but otherwise the behaviour in regime (I) is similar as for $\rho<\infty$. Note that
$\kappa_{M_n} \sim [A/(1-\alpha)]^2 M_n^{2(1-\alpha)}$ for $\alpha \in (-\infty,1)$ and $\kappa_{M_n} \sim [A \log M_n]^2$ for $\alpha = 1$ (recall \eqref{ass3}). 


\paragraph{Trapping time.}

Theorem~\ref{thm:1657} again carries over verbatim.


\subsubsection{Fast growing seed-bank} 
\label{sss.fastgrow}


\paragraph{Results.}

In regime (II) the behaviour is different than in regime (I), both in terms of time scale and scaling limit. In the following we analyse what happens on time scale $\bar{\beta}_n$ in \emph{three different ranges}: 
\begin{equation}
\label{3ranges}
\begin{array}{lll}
&(1) &1 \ll \bar{\beta}_n \ll \beta^{**}_n,\\
&(2) &\beta^{**}_n \ll \bar{\beta}_n \ll \beta^*_n,\\
&(3) &\bar{\beta}_n \gg \beta^*_n. 
\end{array}
\end{equation}
In order to state our result, we need to fix a \emph{depth} $L_n$ until which the seed-bank is being monitored. The proper choice turns out to be any $L_n$ satisfying 
\begin{equation}
1 \ll L_n \ll \bar{\beta}_n^{1/\beta},
\end{equation}
where $\beta>0$ is the exponent in \eqref{ass3}.
 
Recall from \eqref{Edef} that $\S = \G \times \{A,(D_m)_{m\in\N_0}\}$. Abbreviate $\S_n = \G_n \times \{A,(D_m)_{0 \leq m \leq M_n}\}$ and $\S^{L_n}_n = \G_n \times \{A,(D_m)_{0 \leq m \leq L_n}\}$, $0 \leq L_n \leq M_n$. Given a sequence of laws $(\mu_n)_{n\in\N}$ in $\CP([0,1]^\S)$ and a law $\mu \in \CP([0,1]^\S)$, we say that  $\lim^{L_n}_{n\to\infty} \mu_n = \mu$ when
\begin{equation}
\lim_{n\to\infty} \sup_{ {A_n \subset [0,1]^{\S^{L_n}_n}} \atop {\text{measurable}}} 
\left|\,\int_{A_n} f \d\mu_n - \int_{A_n} f \d\mu\,\right| = 0
\qquad \forall\,f \in C_b([0,1]^\S;\R),
\end{equation}
which we refer to as \emph{weak convergence to depths} $(L_n)_{n \in \N_0}$.  
 
\begin{theorem}{{\bf [Finite-systems scheme: Model 2 with $\rho=\infty$ in regime (II)]}}
\label{T.finsys2inf(2)}
Suppose that Assumptions~\ref{ass.migration}, \ref{ass.reginf}, \ref{ass.profinite} and \ref{ass.trker} are in force, and that $\psi_n = o((\beta^{**}_n)^\gamma)$. Suppose that $\mu(0) \in \CT^\mathrm{erg,\bullet}_\theta$, and that $\mu_n(0)$ is given by \eqref{choicelawn}. Then, in the coexistence regime, for the three ranges in \eqref{3ranges}:
\begin{itemize}
\item[{\rm (1)}]
Equilibrium:
\eqref{gh17alt} holds with $\beta_n$ replaced by $\bar{\beta}_n$ and $\lim_{n\to\infty}$ replaced by $\lim^{L_n}_{n\to\infty}$. 
\item[{\rm (2)}]
Partial clustering:
\begin{equation}
\hspace{-1cm}
\begin{array}{lll}
&\lim_{n\to\infty} \CL\left[\big(\hat{\theta}^{M_n,n}(\bar{\beta}_n),\hat{\theta}^{M_n,n}_x(\bar{\beta}_n)\big)\right] 
= \delta_\theta \otimes \delta_{\bar\Upsilon},\\[0.2cm]
&\lim^{L_n}_{n\to\infty} \CL\left[\big(\CE^{M_n,n}(\bar{\beta}_n)\right]
= \lim^{L_n}_{n\to\infty} \CL\left[\big(\mu^{M_n,n}(\bar{\beta}_n)\right]  
= \bar{\Upsilon}\,\delta_{(1,1^{\N_0})^{\G}} + [1-\bar{\Upsilon}]\,\delta_{(0,0^{\N_0})^{\G}},
\end{array}
\end{equation}
where $\bar{\Upsilon} \in \{0,1\}$ is a random variable with law 
\begin{equation}
\CL[\bar{\Upsilon}] = \theta\,\delta_1 + (1-\theta)\,\delta_0. 
\end{equation}
Moreover, for any two time scales $\beta^{**}_n \ll \bar{\beta}_n \ll \tilde{\beta}_n \ll \beta^*_n$,  $\bar{\Upsilon}$ and $\tilde{\Upsilon}$ are independent. 
\item[{\rm (3)}]
Complete clustering: \eqref{gh17altalt} holds with $\beta_n$ replaced by $\bar{\beta}_n$ and $\lim_{n\to\infty}$ replaced by $\lim^{M_n}_{n\to\infty}$. 
\end{itemize}
\end{theorem}


\paragraph{Comments.}

The behaviour in regime (II) is strikingly different from that in regime (I). 
\begin{itemize}
\item[(1)]
Before time scale $\beta^{**}_n$ \emph{local convergence to equilibrium occurs}, as in the infinite system, up to seed-bank depth $L_n = o(\bar{\beta}_n^{1/\beta})$. 
\item[(1-2)]
On time scale $\beta^{**}_n$, \emph{partial clustering} sets in that is \emph{global in the geographic space but local in the seed-bank up to depth} $L_n = o((\beta^{**}_n)^{1/\beta})$, i.e., the active population and the seed-banks up to depth $L_n$ gradually move towards one of the clustered states in which they are either all $1$ or all $0$. (The latter corresponds to the geographic space partitioning into two parts, where all the active individuals and all the dormant individuals in the seed-banks up to depth $L_n$ stem from a single ancestor.) There is no movement yet towards the equilibria of the finite system, i.e., towards the completely clustered states, because the deeper seed-banks have not yet made themselves felt. 
\item[(2)]
After time scale $\beta^{**}_n$ but before time scale $\beta^*_n$, deeper seed-banks come into play, and \emph{partial clustering} occurs up to depth $L_n = o((\bar{\beta}_n)^{1/\beta})$. Since the initial mean is $\theta$, and the mean is preserved under the evolution, the value that is taken in the partial clustering is the \emph{random variable} $\bar{\Upsilon} \in \{0,1\}$ with mean $\theta$, i.e., $\bar{\Upsilon} = 0$ with probability $\theta$ and $\bar{\Upsilon} = 1$ with probability $1-\theta$. The deeper seed-banks make the active population and the seed-banks up to depth $L_n$ undergo \emph{random switches} between the two partially clustered states: on a larger time scale $\tilde{\beta_n}$ the density is equal to a freshly sampled random variable $\tilde{\Upsilon} \in \{0,1\}$ with mean $\theta$, i.e., \emph{the deeper seed-banks overrule the shallower seed-banks}.  
\item[(2-3)]
Only when time scale $\beta^*_n$ is reached do the deepest seed-banks come into play, partial clustering occurs up to depth $L_n = o((\beta^*_n)^{1/\beta})$, with $(\beta^*_n)^{1/\beta} = M_n$, after which \emph{complete clustering} sets in that is \emph{global in the geographic space and in the seed-bank}, i.e., the active population and all the seed-banks gradually move towards one of the clustered states in which they are either all $1$ or all $0$, and complete clustering occurs up to depth $L_n = M_n$. (The latter corresponds to the geographic space partitioning into two parts in which all the individuals, both active and dormant, stem from a single ancestor.) 
\item[(3)]
After time scale $\beta^*_n$ \emph{complete fixation} has been achieved. 
\end{itemize}
Note that in regime (II) the time scale $\kappa_{M_n}|\G_n|$ of regime (I) is no longer relevant. In fact, \emph{the role of $\beta_n$ is taken over by $\beta^*_n$, the time scale at which complete clustering sets in}. The fact that the macroscopic variable tends to $0$ or $1$ on time scale $\beta^*_n$ and not on time scale $\kappa_{M_n}|\G_n|$ is due to the partial clustering, which makes the right-hand side of \eqref{mo2} small.

It remains open to identify what exactly happens in the crossover regimes $\bar{\beta}_n \asymp \beta^{**}_n$ and $\bar{\beta}_n \asymp \beta^*_n$. We expect that on time scale $\beta^{**}_n$ the macroscopic variable associated with the active population and the seed-banks up to depth $L_n = o((\beta^{**}_n)^{1/\beta})$ move towards partial fixation according to a \emph{jump process}, i.e., it follows a piecewise constant path that ends in $0$ or $1$. We expect that on time scale $\beta^*_n$ the macroscopic variable associated with the full population, i.e., the active population and all the seed-banks up to depth $M_n$, moves towards fixation according to a jump process in the same manner, modulo a constant multiple of time. In both instances the behaviour is \emph{different} from the diffusion in Fig.~\ref{fig:WFren} with diffusion function $\CF g$. Nonetheless, we expect that this diffusion still plays a role in the background. See Appendix~\ref{appD*} for speculations.


\paragraph{Explanation.}

In regime (II) we may pretend that the active and dormant time lapses of a lineage in the dual behave in the same way as when $M_n\equiv\infty$. Indeed, if $\tau_n$ denotes the wake-up time of a lineage in the dual for the finite system after it has become dormant, then (compare with \eqref{tauwakeup})
\begin{equation}
\P(\tau_n > t) = \frac{1}{\chi} \sum_{m=0}^{M_n} K_me_m\,\e^{-e_m t}, \qquad t \geq 0.
\end{equation} 
Inserting \eqref{ass3}, approximating the sum by an integral and passing to the new variable $x=e_mt$, we find that
\begin{equation}
\label{tautail}
\P(\tau_n > t) \sim \frac{A}{\beta \chi}B^{1-\gamma} t^{-\gamma} \int_{Bt/\beta^*_n}^{Bt} \d x\, x^{\gamma-1} \e^{-x},
\qquad t \to\infty.
\end{equation}
Thus, we see that the asymptotics found for the infinite system in \eqref{Ptail} prevails in regime (II) because $\bar{\beta}_n/\beta^*_n \downarrow 0$ as $n\to\infty$. 

On $\G_n$, the mean total joint activity time up to time $T$ for two lineages in the dual \emph{without coalescence}, starting anywhere in $\G_n$ and being both active, equals the \emph{hazard}
\begin{equation}
\label{HnT}
H_n(T) \asymp \int_1^T t^{-2(1-\gamma)}\,|\G_n|^{-1}\, \d t, \qquad T \to\infty.
\end{equation}
Indeed, by \eqref{tautail}, the probability for each lineage to be active at time $t$ falls off like $t^{-(1-\gamma)}$, while the probability for the two lineages to be at the same site is $|\G_n|^{-1}$, provided these lineages are well mixed on $\G_n$ at times of order $T$. From \eqref{HnT} we see that 
\begin{equation}
H_n(T) \asymp |\G_n|^{-1} \times \left\{\begin{array}{ll}
T^{2\gamma-1}, &\gamma \in (\tfrac12,1],\\[0.2cm]
\log T, &\gamma = \tfrac12,\\[0.2cm]
C, &\gamma \in (0,\tfrac12),
\end{array}
\right.
\qquad T \to\infty, 
\end{equation}
and so $H_n(T)$ starts to diverge on time scale $\beta^{**}_n$ (recall \eqref{scales}).  Up to time $T$, each lineage takes $\asymp T^\gamma$ migration steps, and so we require that $\psi_n = o((\beta_n^{**})^\gamma)$ to get proper mixing on time scale $\beta_n^{**}$, which explains the assumption on $\psi_n$ in Theorem~\ref{T.finsys2inf(2)}. This assumption is met for all three examples in Remark~\ref{r.1274} when $\gamma \in (\tfrac12,1]$. For the first example we have $|\G_n|=n^d$ and $\psi_n \asymp n^2$, and so we require that $2 < d\gamma/(2\gamma-1)$, which is precisely the condition for coexistence in \eqref{gdcond}. For the second example we have $|\G_n| = n$ and $\psi_n \asymp n^\delta$, and so we require that $\delta < \gamma/(2\gamma-1)$, which is precisely the condition for coexistence in \eqref{gdcondalt}. For the third example we have $|\G_n| = N^n$ and $\psi_n \asymp (N/c)^{n-1}$, and so we require that $N/c < N^{\gamma_N/(2\gamma_N-1)}$, which is precisely the condition for coexistence in \eqref{gdcondaltalt}. The assumption is trivially met for $\gamma \in (0,\tfrac12)$ (recall \eqref{scales}) and $\gamma=\tfrac12$ (recall  Remark~\ref{r.conj})(1)). 

\begin{remark}{{\bf [Crossover for the hazard]}}
\label{r.fininf}
{\rm The integral in \eqref{crinf} arises as the mean total joint activity time on $\G$  
\begin{equation}
I_{\hat{a},\gamma} \asymp \int_1^\infty  t^{-2(1-\gamma)}\,\hat{a}_{t^\gamma}(0,0)\, \d t. 
\end{equation}
In the coexistence regime we have $I_{\hat{a},\gamma}<\infty$, while $H_n(T)$ in \eqref{HnT} starts to diverges on scale $T=\beta^{**}_n$. The reason why this is possible is related to the observation made in Remark~\ref{r.conj}(2), namely, because $|\G_n|^{-1} \asymp \hat{a}_{\psi_n}(0,0)$ and $\psi_n = o(T^\gamma)$, we have $|\G_n|^{-1} \gg \hat{a}_{T^\gamma}(0,0)$.   
} \hfill $\Box$
\end{remark}


\subsection{Open problems}
\label{ss.open}

We close by listing a few open problems. 
\begin{itemize}
\item[(A)]
How can we refine Theorem~\ref{T.finsys2inf(2)}? In particular, what happens at times of order $\beta^{**}_n$ and $\beta^*_n$? On the shorter time scale $\beta^{**}_n$ the system undergoes clustering in the geographic space but not in the seed-bank, while on the longer time scale $\beta^*_n$ the system undergoes clustering in the seed-bank (and hence complete clustering). In Section~\ref{s.fssrhoinffast} we will argue that in regime (II), unlike in regime (I), clustering in the geographic space is \emph{not diffusive}, but rather proceeds in \emph{random bursts}, i.e., on time scale $\beta^{**}_n$ the active macroscopic variable follows a random jump process that lives on $(0,1)$ and ends at $0$ or $1$.
\item[(B)]
What is the analogue of Theorem~\ref{thm:1657} in regime (II)? We expect the trapping time to be of order $\beta^*_n$, but not to be controlled by a diffusion because of the random bursts.
\item[(C)]
What happens in the crossover regime in \eqref{Mncondcrit}? We expect diffusive scaling of the macroscopic variable, but driven by a diffusion function different from $\CF g$.    
\end{itemize}


\section{Preparation: Preservation and convergence}
\label{s.classpres}

The strategy of the proof of the theorems stated in Section~\ref{s.scaling} is to make the following simple ideas rigorous, where for notational convenience we focus on Model 1. Let $S=(S(t))_{t \geq 0}$ denote the semigroup associated with the Markov process $Z=(Z(t))_{t \geq 0}$ on $\G$ defined in \eqref{e409}, and let $S^n,Z^n$ be their finite-system counterparts on $\G_n$. Suppose that we have identified the time scale $\beta_n$ on which the estimator process $\hat\theta^n=(\hat\theta^n(s\beta_n))_{s > 0}$ is tight in $D((0,\infty),[0,1])$. Then, by the Markov property of $Z^n$, we can write
\begin{equation}
\label{heu1}
\CL[Z^n(s\beta_n)] = \int_{[0,1]} \P\big(\hat\theta^n(s\beta_n-L_n) \in \d\theta'\big)\,\mu^n_{\theta'}S^n(L_n),
\qquad s>0,
\end{equation}   
where $L_n$ is chosen such that $\lim_{n\to\infty} L_n = \infty$ and $\lim_{n\to\infty} L_n/\beta_n = 0$, and  $\mu^n_{\theta'} = \CL[Z^n(s\beta_n-L_n) \mid \hat\theta^n = \theta']$ (regular version of the conditional probability). From the tightness assumption on $\hat\theta^n$ we know that the law in \eqref{heu1} barely changes during the time interval $[s\beta_n-L_n,s\beta_n]$. Suppose that, along some subsequence, $\hat\theta^n$ converges to some limit process $\Theta=(\Theta(s))_{s > 0}$ in $D((0,\infty),[0,1])$. Then we should have that
\begin{equation}
\label{heu2}
\mu^n_{\theta'}S^n(L_n) \approx \mu^{n,\uparrow}_{\theta'}(L_n) \Longrightarrow \nu_{\theta'}, \qquad n \to \infty,
\end{equation}
where $\uparrow$ denotes the extension to a law on $\G$ instead of $\G_n$. Combining \eqref{heu1}--\eqref{heu2}, we get 
\begin{equation}
\label{heu3}
\CL[Z^n(s\beta_n)] \Longrightarrow \int_{[0,1]} \P\big(\Theta(s) \in \d\theta'\big)\,\nu_{\theta'}, \qquad n \to \infty, \qquad s>0. 
\end{equation}
At the same time, we should be able to identify $\Theta$ by taking the increasing process of $\hat\theta^n$ (recall \eqref{incrpr1}), 
\begin{equation}
\label{heu4}
\langle \hat\theta^n(s)\rangle = \frac{1}{1+K} \int_0^s \d u\,\frac{1}{|\G_n|}\sum_{i \in \G_n} g\big(x_i(u\beta_n)\big),
\qquad s>0,
\end{equation}
to conclude that, by the law of large numbers, the volatility function of $\Theta$ is given by 
\begin{equation}
\label{heu5}
\theta \mapsto \frac{1}{1+K}\,\E_{\nu_\theta}[g(x_0(\cdot))]. 
\end{equation}

The question is how to transform the \emph{heuristic argument} in \eqref{heu1}--\eqref{heu5} into a \emph{rigorous argument} (one of the difficulties being that limits are being exchanged at several places). It was shown in \cite{CG94*} that, subject to a collection of assumptions listed as (A1)--(A10) in Appendix \ref{appA*}, the above steps can indeed be made rigorous. Our task will therefore be to verify these assumptions. In the present section we collect some properties that are needed to complete this task. These properties concern approximations of the infinite system by finite systems on time scales of order 1, ergodic theorems for the infinite system, and regularity properties of the equilibria for the infinite system as a function of underlying paramaters.  The verification of the assumptions is carried out in Sections \ref{s.fssrhofin} and \ref{s.fssrhoinfslow} for $\rho<\infty$ and $\rho=\infty$, respectively.
   
Before we proceed we need the following important observations concerning the set of initial laws. In Section~\ref{ss.class} we introduce a \emph{class of initial laws} $\CR_\theta$ parametrised by the density $\theta \in [0,1]$, for which the system decorrelates in space over time intervals of length $o(\beta_n)$, where $\beta_n$ is the \emph{macroscopic time scale}, which must be properly chosen (see Appendix~\ref{appB*}, part $III$). In fact, the proper choices are:
\begin{equation}
\beta_n = \left\{\begin{array}{lll}
\kappa|\G_n| &\text{Model 1} &\kappa = (1+K)^2,\\
\kappa|\G_n| &\text{Model 2 with $\rho<\infty$} &\kappa = (1+\rho)^2,\\
\kappa_{M_n}|\G_n| &\text{Model 2 with $\rho=\infty$, regime (I)} &\kappa_{M_n} = (1+\sum_{m=0}^{M_n} K_m)^2,\\
\beta^*_n &\text{Model 2 with $\rho=\infty$, regime (II)}. &
\end{array}
\right.
\end{equation}

The initial laws chosen for Model 1 and Model 2 all fall in $\CR_\theta$. In Section~\ref{ss.pres} we show that the evolved law of the system stays inside the class $\CR_\theta$ over time intervals of length $o(\beta_n)$. In Section~\ref{ss.LLN} we show that the macroscopic variable converges to $\theta$ over time intervals of length $o(\beta_n)$. In Section~\ref{ss.macroconv} we use this fact to show that on time scale $\beta_n$ the law of the system conditional on the macroscopic variable being $\theta'$ falls in the class $\CR_{\theta'}$. We will see in Sections~\ref{s.fssrhofin}--\ref{s.fssrhoinfslow} that the latter is the key to the proof that the finite system locally converges to an equilibrium of the infinite system, parametrised by the instantaneous value of the macroscopic variable. 

In Sections~\ref{ss.pres}--\ref{ss.macroconv} we \emph{exclude} Model 2 with $\rho=\infty$ in regime (II). In Section~\ref{ss.deep} we show that the results do carry over to this case as well, but \emph{only} for time intervals of length $o(\beta^{**}_n)$ rather than $o(\beta_n)$.

Throughout this section, $b^{(1),n}(\cdot,\cdot)$ and $b^{(2),n}(\cdot,\cdot)$ denote the Markov process transition kernels defined in \eqref{mrw1} and \eqref{mrw2} in Appendix~\ref{appA*}, which describe the motion of the \emph{lineages} of the individuals in the spatial population with seed-bank. Depending on the model under consideration, we write
\begin{equation}
\CR_\theta = \left\{\begin{array}{ll}
\CR^{(1)}_\theta, &\text{Model 1},\\[0.2cm] 
\CR^{(2)}_\theta, &\text{Model 2 with $\rho<\infty$},\\[0.2cm] 
\CR^{(2),\bullet}_\theta, &\text{Model 2 with $\rho=\infty$}.
\end{array}
\right.
\end{equation} 


\subsection{Classes of initial laws}
\label{ss.class}

For each model a specific class of initial laws on $\G$ and their restrictions to $\G_n$ defined by \eqref{choicelawn} play an important role. The following are adaptations of Definition~\ref{Rtheta}, where we write $\S_n = \G_n \times \{A,D\}$ in Model 1 and $\S_n = \G_n \times \{A,(D_m)_{0 \leq m \leq M_n}\}$ in Model 2, and we recall the definition of $z_u$ in \eqref{zudef}. 
 
\begin{definition}{\bf [Model 1: Class of initial laws]}
\label{def:class1}
For $\theta \in [0,1]$, let 
\begin{equation}
\label{class1}
\CR^{(1)}_\theta = \left\{\mu \in \CT^{\mathrm{erg}}\colon\, (\ast)_\theta \text{ holds}\right\}, 
\end{equation}
with $(\ast)_\theta$ the requirement that, for every sequence of times $(t_n)_{n\in\N}$ satisfying $\lim_{n\to\infty} t_n = \infty$ and $\lim_{n\to\infty}$ $t_n/\beta_n = 0$, 
\begin{equation}
\label{astcond}
(\ast)_\theta
\begin{array}{lll}
&(1) \,\,\forall\,u_1 \in \S_n\colon\\
&\lim\limits_{n\to\infty} \sum\limits_{u_2 \in \S_n} 
b^{(1),n}_{t_n}(u_1,u_2)\,\E_{\Phi_n\mu}[z_{u_2}] = \theta,\\[0.4cm] 
&(2) \,\,\forall\,u_1,u_2 \in \S_n\colon\\
&\lim\limits_{n\to\infty} \sum\limits_{u_3,u_4 \in \S_n}  
b^{(1),n}_{t_n}(u_1,u_3)\, b^{(1),n}_{t_n}(u_2,u_4)\,\E_{\Phi_n\mu}[z_{u_3}z_{u_4}] = \theta^2, 
\end{array}
\end{equation}
where $\Phi_n\mu$ is the restriction of $\mu$ defined in \eqref{choicelawn}--\eqref{choicelawnalt}, $b^{(1),n}_{t_n}(\cdot,\cdot)$ is the time-$t_n$ transition kernel of the Markov process $b^{(1)}(\cdot,\cdot)$ in \eqref{mrw1} restricted to $\G_n$.       
\end{definition}

\begin{definition}{\bf [Model 2: Class of initial laws]}
\label{def:class2}
For $\theta \in [0,1]$, let 
\begin{equation}
\label{class2}
\begin{array}{lll}
&\rho<\infty\colon
&\CR^{(2)}_\theta = \left\{\mu \in \CT^{\mathrm{erg}}\colon\,(\ast)_\theta \text{ holds}\right\},\\[0.3cm] 
&\rho=\infty\colon
&\CR^{(2),\bullet}_\theta = \left\{\mu \in \CT^{\mathrm{erg}}\colon\,(\ast)_\theta \text{ holds and } 
\mu \text{ is colour regular}\right\},
\end{array} 
\end{equation}
where in $(\ast)_\theta$ in \eqref{astcond} $b^{(1),n}_{t_n}(\cdot,\cdot)$ is replaced by $b^{(2),n}_{t_n}(\cdot,\cdot)$, the time-$t_n$ transition kernel of the Markov process $b^{(2)}(\cdot,\cdot)$ in \eqref{mrw2} restricted to $\S_n$. 
\end{definition}

\noindent
The two properties in $(\ast)_\theta$ in \eqref{astcond}, which we refer to as \emph{Liggett conditions} (see \cite[Remark 2.15]{GdHOpr1}), say that over time intervals of length $o(\beta_n)$ the following are true: (1) the average of the component seen at time $t_n$ by the Markov process with transition kernel $b^{(1),n}(\cdot,\cdot)$ and $b^{(2),n}(\cdot,\cdot)$ converges to the initial density $\theta$ defined in \eqref{thetadefalt}; (2) the covariance of the components seen at time $t_n$ by two independent Markov processes with transition kernel $b^{(1),n}(\cdot,\cdot)$ and $b^{(2),n}(\cdot,\cdot)$ converges to zero. Consequently, the component seen by such a Markov process converges to $\theta$ in $L_2$. 

The following lemma shows that all laws in $\CR^{(1)}_\theta$, $\CR^{(2)}_\theta$ and  $\CR^{(2),\bullet}_\theta$ are invariant and ergodic under translations, have density $\theta$, and are colour regular (recall \eqref{erglaw1}, \eqref{erglaw2fin} and \eqref{erglaw2inf}).  

\begin{lemma}{\bf [Ergodicity]}
\label{lem:erg}
(a) For every $\theta \in [0,1]$, $\CR^{(1)}_\theta \subset \CT^{\mathrm{erg}}_\theta$.\\
(b) For every $\theta \in [0,1]$, $\CR^{(2)}_\theta \subset \CT^{\mathrm{erg}}_\theta$, respectively, $\CR^{(2),\bullet}_\theta \subset \CT^{\mathrm{erg},\bullet}_\theta$.
\end{lemma}

\begin{proof}
The claim will be immediate from the proof of Lemma~\ref{lem:pres} below. For the finite-systems scheme, we will work with $\CT^\bullet = \cup_{\theta \in [0,1]} \CT^\bullet_\theta$ as the set of initial laws (recall \eqref{covcond1}--\eqref{lawscr}) and use that $\theta^\bullet_\mu = \lim_{m\to\infty} \E_\mu[y_{0,m}] = \theta$ for $\mu \in \CT^\bullet_\theta$.
\end{proof}

Note that (recall \eqref{laws} and \eqref{lawscr})
\begin{equation}
\begin{aligned}
\CT^\mathrm{erg} = \cup_{\theta \in [0,1]} \CT^\mathrm{erg}_\theta,
\qquad \CT^\mathrm{erg,\bullet} = \cup_{\theta \in [0,1]} \CT^\mathrm{erg,\bullet}_\theta.
\end{aligned}
\end{equation}


\subsection{Preservation of the classes}
\label{ss.pres}

The classes introduced in Definitions~\ref{def:class1}--\ref{def:class2} are preserved over time intervals of length $o(\beta_n)$. In the following lemmas $\CR_\theta$ stands for either $\CR^{(1)}_\theta$, $\CR^{(2)}_\theta$ and $\CR^{(2),\bullet}_\theta$. Recall the restriction operator $\Phi_n$ defined in \eqref{choicelawnalt}.  

\begin{lemma}{\bf [Preservation properties]}
\label{lem:pres}
For any $\theta \in [0,1]$ and any sequence of times $(t_n)_{n\in\N}$ satisfying $\lim_{n\to\infty} t_n = \infty$ and $\lim_{n\to\infty} t_n/\beta_n = 0$, the following holds for, respectively, Model 1, Model 2 with $\rho<\infty$, and Model 2 with $\rho=\infty$ in regime (I):
\begin{itemize}
\item[{\rm (a)}] 
If $\mu(0) = \mu \in \CR_\theta$, then $\mu(t_n) \in \CR_\theta$. 
\item[{\rm (b)}] 
If $\mu(0) = \mu \in \CR_\theta$, then $\nu \in \CR_\theta$ for any weak limit point $\nu$ of $(\Phi_n\mu(t_n))_{n\in\N}$.  
\end{itemize}
\end{lemma}

\begin{proof}
(a) The claim is obvious after we replace $t_n$ by $2t_n$ in $(\ast)_\theta$.

\medskip\noindent
(b) The proof proceeds in 3 Steps, each based on a lemma. Before we start we note that the macroscopic variable $\hat{\theta}^n$ in Model 1 (defined in \eqref{e1311}) satisfies 
\begin{equation}
\label{meanpres}
\E_{\Phi_n\mu(t_n)}[\hat{\theta}^n] = \E_{\Phi_n\mu}[\hat{\theta}^n] = \theta,
\end{equation}
because $(\hat{\theta}^n(t))_{t \geq 0}$ is a martingale under the law $\Phi_n\mu$ (recall \eqref{mart}). Therefore, also for any weak limit point $\nu$ of $(\Phi_n\mu(t_n))_{n\in\N}$,
\begin{equation}
\label{meanpresalt}
\E_{\Phi_n\nu(t_n)}[\hat{\theta}^n] = \E_{\Phi_n\nu}[\hat{\theta}^n] = \theta.
\end{equation}
The same observation applies to the macroscopic variable $\hat{\theta}^{M_n,n}$ in Model 2 (defined in \eqref{mo3}). 


\paragraph{Step 1.}

The following lemma says that under the law $\Phi_n\nu$ all components have mean $\theta$. Via \eqref{meanpresalt} this implies that property $(\ast)_\theta(1)$ holds and that $\nu$ is colour regular.
 
\begin{lemma}{\bf [Constant means]}
\label{lem:allmeans}
Let $\mu(0) = \mu \in \CR_\theta$. Let $(t_n)_{n\in\N}$ satisfy $\lim_{n\to\infty} t_n = \infty$ and $\lim_{n\to\infty} t_n/\beta_n = 0$. Let $\nu$ be any weak limit point of $(\Phi_n\mu(t_n))_{n\in\N}$. Then $\E_{\Phi_n\nu}[z_{(i,R_i)}] = \theta$ for all $n\in\N$, all $i \in \G_n$, and all $R_i \in \{A,D\}$ or $\{A,(D_m)_{0 \leq m \leq M_n}\}$.
\end{lemma}

\begin{proof} 
Since $\nu$ is a weak limit point of $(\Phi_n\mu(t_n))_{n\in\N}$, there exists $(t_{n_k})_{k\in\N}$ satisfying $\lim_{k\to\infty} t_{n_k} = \infty$ and $\lim_{k\to\infty} t_{n_k}/\beta_{n_k} = 0$ such that $\nu = \lim_{k\to\infty} \Phi_{n_k}\mu(t_{n_k})$. Hence, for Model 1, for all $u_1 \in \S_n$,
\begin{equation}
\E_{\nu}[z_{u_1}] = \lim_{k\to\infty} \E_{\Phi_{n_k}\mu(t_{n_k})} [z_{u_1}]
= \lim_{k\to\infty} \sum_{u_2 \in \S_{n_k}} b^{(1),n}_{t_{n_k}}\big(u_1,u_2\big)\,\E_{\Phi_{n_k}\mu}[z_{u_2}] 
= \theta,
\end{equation}
where the last equality holds because $\mu \in \CR^{(1)}_\theta$. The same holds for Model 2 with $b^{(1),n}_{t_{n_k}}(\cdot,\cdot)$ replaced by $b^{(2),n}_{t_{n_k}}(\cdot,\cdot)$. Since $\E_{\nu}[z_{u_1}] = \theta$, it follows that $\E_{\Phi_n\nu}[z_{u_1}] = \theta$ for all $n\in\N$ (use \eqref{choicelawn} and the fact that $\nu$ is invariant under translations). 
\end{proof}


\paragraph{Step 2.} 
To settle property $(\ast)_\theta(2)$, we follow the covariance computations for the infinite system carried out in \cite[Sections 5--6]{GdHOpr1}. Together with property $(\ast)_\theta(1)$, the following settles property $(\ast)_\theta(2)$.

\begin{lemma}{\bf [Decaying correlations]}
\label{lem:deccor}
$\mbox{}$\\
Let $\mu(0) = \mu \in \CR_\theta$. Let $(t_n)_{n\in\N}$ satisfy $\lim_{n\to\infty} t_n = \infty$ and $\lim_{n\to\infty} t_n/\beta_n = 0$. Let $\nu$ be any weak limit point of $(\Phi_n\mu(t_n))_{n\in\N}$. Then
\begin{equation}
\label{covdec}
\lim_{|i-j| \to \infty} \lim_{n\to\infty} \Big|\text{{\rm Cov}}_{\Phi_n\nu}(z_{(i,R_i)},z_{(j,R_j)})\Big| = 0
\quad \text{ uniformly in } R_i,R_j.
\end{equation}
\end{lemma}

\begin{proof}
For Model 1, we use It\^o's formula for the first and second moment to write (see \cite[Lemma 5.1]{GdHOpr1}) 
\begin{equation}
\label{covexpr}
\begin{aligned}
&\text{Cov}_{\Phi_n\mu(t_n)}(z_{u_1},z_{u_2}) = \int_{0}^{t_n} \d s\, \sum_{k \in \G_n} 
b^{(1),n}_{t_n-s}\big(u_1,(k,A)\big)\,b^{(1),n}_{t_n-s}\big(u_2,(k,A)\big)\,\E_{\Phi_n\mu}[g(x_k(s))]\\
&\quad + \sum_{u_3,u_4 \in \S_n} b^{(1),n}_{t_n}(u_1,u_3)\,b^{(1),n}_{t_n}(u_2,u_4)\,
\text{Cov}_{\Phi_n\mu}(z_{u_3},z_{u_4}).
\end{aligned}
\end{equation}
The same expression holds for Model 2, with $b^{(1),n}$ replaced by $b^{(2),n}$ (see \cite[Lemma 6.1]{GdHOpr1}). In what follows we focus on Model 2. Model 1 has no truncation of the seed-bank: $M_n=0$ for all $n\in\N$.

\medskip\noindent
{\bf 1.}
The second term in \eqref{covexpr} tends to zero as $n\to\infty$ by \eqref{astcond}, because $\nu \in \CR^{(2)}_\theta$ by Lemma~\ref{lem:pres}(b). To estimate the first term, which we abbreviate by 
\begin{equation}
\label{Itndef}
I_{t_n}(u_1,u_2) = \int_{0}^{t_n} \d s\, \sum_{k \in \G_n} 
b^{(2),n}_{t_n-s}(u_1,(k,A))\,b^{(2),n}_{t_n-s}(u_2,(k,A))\,\E_{\Phi_n\mu}[g(x_k(s))], 
\end{equation}
we proceed as follows (in which we no longer need that $\nu \in \CR^{(2)}_\theta$). Let $\tau^\uparrow$ denote the first time the two Markov processes are \emph{jointly active}, which is a random variable whose law depends on $u_1,u_2$. Because the building up of joint activity time can only start after time $\tau^*$, we have
\begin{equation}
\label{IADcomp}
I_{t_n}(u_1,u_2\big) \leq I_{t_n}\big((I,A),(J,A)\big) 
\end{equation}
where $I,J \in \G_n$ are the random locations of the two Markov processess at time $\tau^\uparrow$. For $s \geq 0$, let $\CE(t)$ and $\CE^\prime(t)$ be the events that the respective random walks are active at time $t$, and let 
\begin{equation}
T(t) = \int_0^t \d s\,1_{\CE(s)}, \qquad T^\prime(t) = \int_0^t \d s\,1_{\CE^\prime(s)},
\end{equation}
be their total activity time up to time $t$. Write $\P_n$ to denote the law of $(\CE(s),\CE^\prime(s))_{s \geq 0}$ given that the two random walks start in the active state.  Note that this process is independent of the migration of the Markov processes. Estimate (henceforth $\|g\|$ denotes the supremum norm of $g$)
\begin{equation}
\label{covcomp}
\begin{aligned}
I_{t_n}\big((i,A),(j,A)\big)
&\leq \|g\| \int_{0}^{t_n} \d s\, \sum_{k\in\G_n} 
b^{(2),n}_{t_n-s}\big((i,A),(k,A)\big)\,b^{(2),n}_{t_n-s}\big((j,A),(k,A)\big)\\
&= \|g\| \int_{0}^{t_n} \d s\, \sum_{k\in\G_n} 
\E_n\left[a^n_{T(s)}(i,k)\,1_{\CE(s)}\,a^n_{T^\prime(s)}(j,k)\,1_{\CE^\prime(s)}\right]\\
&= \|g\| \int_{0}^{t_n} \d s\,\,\E_n\left[\hat{a}^n_{T(s)+T^\prime(s)}(i-j,0)
\,1_{\CE(s)}\,1_{\CE^\prime(s)}\right]. 
\end{aligned}
\end{equation}
We rewrite and split the integral in the right-hand side into three parts:
\begin{equation} 
\label{Jdefs}
\begin{aligned}
J^1_{t_n} &= |\G_n|^{-1} \int_{0}^{t_n} \d s\,\,\E_n\left[1_{\{T(s)+T^\prime(s) > \psi_n\}}\,
1_{\CE(s)}\,1_{\CE^\prime(s)}\right],\\
J^2_{t_n}(i,j) &= |\G_n|^{-1} \int_{0}^{t_n} \d s\,\,\E_n\left[\left(|\G_n|\,\hat{a}^n_{T(s)+T^\prime(s)}(i-j,0)-1\right)\,
1_{\{T(s)+T^\prime(s) > \psi_n\}}\,1_{\CE(s)}\,1_{\CE^\prime(s)}\right],\\ 
J^3_{t_n}(i,j) &= \int_{0}^{t_n} \d s\,\,\E_n\left[\hat{a}^n_{T(s)+T^\prime(s)}(i-j,0)\,
1_{\{T(s)+T^\prime(s) \leq \psi_n\}}\,1_{\CE(s)}\,1_{\CE^\prime(s)}\right],
\end{aligned}
\end{equation}
where $\psi_n$ is the mixing time defined in \eqref{pnmix}. 

\medskip\noindent
{\bf 2.}
First we look at $J^3_{t_n}(i,j)$. Define the event
\begin{equation}
\label{Andef}
\CA_n = \big\{\text{until time } t_n \text{ neither of the two Markov processes visits a seed-bank with colour } m > M_n\big\}.
\end{equation}
Since $\P_n(\CA_n)=1$, we have
\begin{equation}
\label{J3bd1}
\begin{aligned}
&\E_n\left[\hat{a}^n_{T(s)+T^\prime(s)}(i-j,0)\,
1_{\{T(s)+T^\prime(s) \leq \psi_n\}}\,1_{\CE(s)}\,1_{\CE^\prime(s)}\right]\\
&\qquad = \E_n\left[\hat{a}^n_{T(s)+T^\prime(s)}(i-j,0)\,1_{\CE(s)}\,
1_{\{T(s)+T^\prime(s) \leq \psi_n\}}\,1_{\CE^\prime(s)}\,1_{\CA_n}\right].
\end{aligned}
\end{equation}
On the event $\CA_n$, the law of $\{\CE(s),\CE'(u)\}_{u \in [0,t_n]}$ is the same under $\P_n$ as under $\P$, where the former refers to the truncated system and the latter refers to the non-trunctated system. Hence 
\begin{equation}
\label{J3bd2}
\begin{aligned}
&\E_n\left[\hat{a}^n_{T(s)+T^\prime(s)}(i-j,0)\,
1_{\{T(s)+T^\prime(s) \leq \psi_n\}}\,1_{\CE(s)}\,1_{\CE^\prime(s)}\,1_{\CA_n}\right]\\
&\qquad = \E\left[\hat{a}^n_{T(s)+T^\prime(s)}(i-j,0)\,
1_{\{T(s)+T^\prime(s) \leq \psi_n\}}\,1_{\CE(s)}\,1_{\CE^\prime(s)}\,1_{\CA_n}\right].
\end{aligned}
\end{equation}
On the event $\{T(s)+T^\prime(s) \leq \psi_n\}$, we can estimate 
\begin{equation}
\label{J3bd3}
\begin{aligned}
&\E\left[\hat{a}^n_{T(s)+T^\prime(s)}(i-j,0)\,
1_{\{T(s)+T^\prime(s) \leq \psi_n\}}\,1_{\CE(s)}\,1_{\CE^\prime(s)}\,1_{\CA_n}\right]\\
&\qquad \leq C\,\E\left[\hat{a}_{T(s)+T^\prime(s)}(i-j,0)\,1_{\{T(s)+T^\prime(s) \leq \psi_n\}}\,1_{\CE(s)}\,1_{\CE^\prime(s)}\right]
\leq C\,\E\left[\hat{a}_{T(s)+T^\prime(s)}(0,0)\,1_{\CE(s)}\,1_{\CE^\prime(s)}\right],
\end{aligned}
\end{equation}
where the first inequality uses \eqref{afininfcomp} and the second inequality uses that $\hat{a}_t(k,0) \leq \hat{a}_t(0,0)$ for all $k \in \G$ and $t \geq 0$. The right-hand side integrated over $s \in [0,\infty)$ is the mean total joint activity time for two Markov processes on $\G$ both starting from $(0,A)$, which is finite if and only if the system is in the \emph{coexistence regime} (see \cite[Section 2.5]{GdHOpr1}). Since $\lim_{|k| \to \infty} \hat{a}_t(k,0) = 0$ uniformly in $t \geq 0$, it follows from the third line of \eqref{Jdefs} via \eqref{J3bd1}--\eqref{J3bd3} and dominated convergence that
\begin{equation}
\label{est1}
\lim_{|i-j|\to\infty}  J^3_{t_n}(i,j) = 0 \text{ uniformly in } (t_n)_{n\in\N}.
\end{equation}
Hence, to show that $J^3_{t_n}(I,J)$ gives a vanishing contribution to \eqref{IADcomp} as $|i-j| \to \infty$, we must show that the latter implies that $|I-J| \to \infty$ with $\P_n$-probability tending to 1 as $n\to\infty$. But this again trivially follows from the fact that $\lim_{|k|\to\infty} a_t(0,k) = 0$ uniformly in $t \geq 0$.

\medskip\noindent
{\bf 3.}
Next we look at $J^1_{t_n}$ and $J^2_{t_n}(i,j)$. Estimate, with the help of \eqref{pnmix},
\begin{equation}
\label{est2}
\forall\,\,i,j \in \G_n\colon\quad |J^2_{t_n}(i,j)| \leq C' J^1_{t_n}, \qquad 
C' = \sup_{n\in\N} \sup_{t \geq \psi_n} \sup_{k \in \G_n} 
\Big|\, |\G_n|\, \hat{a}^n_t(0,k) - 1\Big| < \infty.
\end{equation}
Further estimate 
\begin{equation}
\label{est3}
J^1_{t_n} \leq \tilde{J}_{t_n} = |\G_n|^{-1} \int_{0}^{t_n} \d s\,\P_n(\CE(s))\,\P_n(\CE^\prime(s)).
\end{equation}
We will show that 
\begin{equation}
\label{Jtildelim}
\lim_{n\to\infty} \tilde{J}_{t_n} = 0.
\end{equation}
Together with \eqref{est1}--\eqref{est3} this will imply \eqref{covdec}. 

\medskip\noindent
{\bf 4.}
For the proof of \eqref{Jtildelim} we distinguish between Model 1 and 2.


\paragraph{Model 1.}

Since $\P_n=\P$ and $\lim_{s\to\infty} \P(\CE(s)) = \lim_{s\to\infty} \P(\CE^\prime(s)) = (1+K)^{-1}$, the integral in \eqref{est3} scales like $t_n/\kappa$ with $\kappa=(1+K)^2$. Since $t_n = o(\beta_n)$ and $\beta_n=\kappa |\G_n|$, we get \eqref{Jtildelim}. 


\paragraph{\bf Model 2.}

For $\rho<\infty$, we have $\lim_{n\to\infty} \P_n=\P$, and $\lim_{s\to\infty} \P(\CE(s)) = \lim_{s\to\infty} \P(\CE^\prime(s)) = (1+\rho)^{-1}$, and so the argument for Model 1 carries over. For $\rho=\infty$, a bit more care is needed. In particular, we need to distinguish between regimes (I) and (II), which have different macroscopic time scales: $\beta_n = \kappa_{M_n}|\G_n|$, respectively, $\beta_n = \beta^*_n$. 

As shown in Appendix~\ref{appC*},
\begin{equation}
\label{EEprimescal}  
\P_n(\CE(s))\,\P_n(\CE^\prime(s)) \asymp \left\{\begin{array}{ll}
1/\kappa_{s^{1/\beta}}, &0 \leq s \leq \beta^*_n,\\
1/\kappa_{M_n}, &\beta^*_n < s \leq t_n.
\end{array}
\right.
\end{equation}
where we recall from \eqref{scales} that $\beta^*_n = M_n^{\,\beta}$ and from \eqref{betaKndef} $\kappa_M = (1+\sum_{m=0}^M K_m)^2$. Indeed, heuristically, at time $s=M^\beta$ the active component has communicated with the first $M$ dormant components, and so joint activity occurs with probability $1/\kappa_M=1/\kappa_{s^{1/\beta}}$. This holds all the way up to $s=\beta^*_n$, after which the active component has communicated with all $M_n$ dormant components, and so joint activity occurs with probability $1/\kappa_{M_n}$. Now, inserting \eqref{EEprimescal} into the right-hand side of \eqref{est3}, we get
\begin{equation}
\tilde{J}_{t_n} \asymp \tilde{J}^1_{t_n} + \tilde{J}^2_{t_n}
\end{equation}
with
\begin{equation}
\label{J12def}
\tilde{J}^1_{t_n} = |\G_n|^{-1} \int_0^{\beta^*_n} \d s\,(1/\kappa_{s^{1/\beta}}),
\qquad
\tilde{J}^2_{t_n} = |\G_n|^{-1} \int_{\beta^*_n}^{t_n} \d s\,(1/\kappa_{M_n}).
\end{equation}

\medskip\noindent
$\blacktriangleright$ Regime (I). Since $\beta_n = \kappa_{M_n}|\G_n|$, we have $\tilde{J}^2_{t_n} \leq t_n/\beta_n = o(1)$ because $t_n = o(\beta_n)$, and so we need only worry about $\tilde{J}^1_{t_n}$. Recall from \eqref{ass3} that $\gamma=\frac{\alpha+\beta-1}{\beta}$ with $\alpha \leq 1 < \alpha+\beta$. Note that $\alpha=1$ implies $\gamma=1$.
\begin{itemize}
\item
For $\gamma \in (\tfrac12,1]$ and $\alpha \in (-\infty,1)$, we have $\tilde{J}^1_{t_n} \asymp |\G_n|^{-1} (\beta^*_n)^{2\gamma-1} = (\beta^*_n/\beta^{**}_n)^{2\gamma-1}$, where we use that $\kappa_M \asymp M^{2(1-\alpha)}$ as $M\to\infty$, $\gamma = \tfrac{\alpha+\beta-1}{\beta}$, and $\beta^{**}_n = |\G_n|^{1/(2\gamma-1)}$. Since $\beta^*_n = o(\beta^{**}_n)$, we get $\tilde{J}^1_{t_n} = o(1)$. 
\item
For $\gamma = \alpha = 1$, we have $\tilde{J}^1_{t_n} \asymp  |\G_n|^{-1} (\beta^*_n/\log \beta^*_n) = (\beta^*_n/\beta_n)(\kappa_{M_n}/\log \beta^*_n) \asymp \beta^*_n/\beta_n$, where we use that $\kappa_M \asymp \log M$ as $M\to\infty$ and $\beta^*_n = M_n^{\,\beta}$. Since $\beta^*_n = o(\beta_n)$, we again get $\tilde{J}^1_{t_n} = o(1)$. 
\item
For $\gamma=\tfrac12$ and $\alpha \in (-\infty,1)$, we have $\tilde{J}^1_{t_n} \asymp  |\G_n|^{-1} \log \beta^*_n$. Since $\log \beta^*_n = o(\log \beta^{**}_n) = o(|\G_n|)$, we again get $\tilde{J}^1_{t_n} = o(1)$. 
\item
For $\gamma \in (0,\tfrac12)$, we use that $\lim_{n\to\infty} \P_n=\P$ and $\P(\CE(s)) \asymp u^{-(1-\gamma)}$, $u \to \infty$, as shown in \cite[Section 6.2]{GdHOpr1}. Hence, by monotone convergence, $\tilde{J}_{t_n} \asymp C'' |\G_n|^{-1} = o(1)$ with $C'' = \int_0^\infty \d s\,u^{-2(1-\gamma)} < \infty$. 
\end{itemize}

\medskip\noindent
$\blacktriangleright$ Regime (II). This regime is excluded because it exhibits different behaviour.  See Section~\ref{ss.deep}.
\end{proof}


\paragraph{Step 3.}

In \cite[Lemma 6.10]{GdHOpr1} we showed for Model 2 with $\rho=\infty$ that for the system on $\G$ in equilibrium the deep seed-banks are deterministic, i.e.,
\begin{equation}
\label{e4242}
\lim_{m\to\infty} \mathrm{Var}_{\nu_\theta}(y_{0,m}) = 0.
\end{equation}
The following lemma says that the same holds for the scaling limit of the system on $\G_n$.

\begin{lemma}{\bf [Deterministic deep seed-banks: Model 2, $\rho=\infty$]}
\label{lem:deep}
Let $\mu(0) = \mu \in \CR_\theta$. Let $(t_n)_{n\in\N}$ satisfy $\lim_{n\to\infty} t_n = \infty$ and $\lim_{n\to\infty} t_n/\beta_n = 0$. Let $\nu$ be any weak limit point of $(\Phi_n\mu(t_n))_{n\in\N}$. Then
\begin{equation}
\label{e4242fin}
\lim_{m\to\infty} \limsup_{n\to\infty} \mathrm{Var}_{\Phi_n\nu}(y_{0,m}) = 0.
\end{equation}
\end{lemma}

\begin{proof}
(The proof does in fact not use that $\mu \in \CR_\theta$.) Write (see \cite[Lemma 6.1]{GdHOpr1})
\begin{equation}
\begin{aligned}
\mathrm{Var}_{\Phi_n\mu(t_n)}(y_{0,m})
&= \E_{\Phi_n\mu}\left[\big(y_{0,m}(t_n)-\E_{\Phi_n\mu}[y_{0,m}(t_n)]\big)^2\right]\\
& = \int_{0}^{t_n} \d s\, \sum_{k \in \G_n} 
b^{(2),n}_{t_n-s}\big((0,R_m),(k,A)\big)\,b^{(2),n}_{t_n-s}\big((0,R_m),(k,A)\big)\,\E_{\Phi_n\mu}[g(x_k(u))]\\
&\leq \|g\| \int_{0}^{t_n} \d s\, \sum_{k \in \G_n} 
b^{(2),n}_{t_n-s}\big((0,R_m),(k,A)\big)\,b^{(2),n}_{t_n-s}\big((0,R_m),(k,A)\big),
\end{aligned}
\end{equation}
where the estimate is uniform in $\mu$. The sum under the integral is the probability that two Markov processes, both starting from $(0,R_m)$ and moving according to $b^{(2),n}(\cdot,\cdot)$, at time $t_n-s$ are at the same site and both active. Define
\begin{equation}
\label{tau}
\tau=\big\{t\geq 0\colon\,RW(t)=RW^\prime(t)=(i,A)\text{ for some }i\in\G_n\big\},
\end{equation}
i.e., the first time when the two Markov processes are jointly active at the same site. Write $\P^{(2),n}_{(0,D_m),(0,D_m)}$ to denote the joint law of the two Markov processes. Then
\begin{equation}
\label{var1}
\begin{aligned}
&\int_{0}^{t_n} \d s\, \sum_{k \in \G_n} 
b^{(2),n}_{t_n-s}\big((0,R_m),(k,A)\big)\,b^{(2),n}_{t_n-s}\big((0,R_m),(k,A)\big)\\
&= \int_{0}^{t_n}\d s\,\, \E^{(2),n}_{(0,D_m),(0,D_m)}\left[\sum_{k\in\G_n} 
1_{\{RW(s)=RW^\prime(s)=k\}}\,1_{\CE(s)}\,1_{\CE^\prime(s)}\right]\\
&= \E^{(2),n}_{(0,D_m),(0,D_m)}\left[1_{\{\tau \leq t_n\}} \int_{0}^{t_n-\tau}\d s\,\,E^{(2),n}_{(0,A),(0,A)}\left[\sum_{k\in\G_n}  
1_{\{RW(s)=RW^\prime(s)=k\}}\,1_{\CE(s)}\,1_{\CE^\prime(s)}\right]\right]\\
&\leq \P^{(2),n}_{(0,D_m),(0,D_m)}\left(\tau \leq t_n\right) \int_{0}^{t_n} \d s\,\, \E_n\left[\hat{a}^n_{T(s)+T^\prime(s)}(0,0)
\,1_{\CE(s)}\,1_{\CE^\prime(s)}\right],
\end{aligned}
\end{equation}
where for the second equality we use the strong Markov property at time $\tau$, together with the fact that on the event $\{\tau>t_n\}$ the product of the indicators equals $0$ for all $u \in [0,t_n]$, and for the inequality we recall \eqref{covcomp}. By \eqref{Jdefs}, we have
\begin{equation}
\int_0^{t_n} \d s\,\,\E_n\left[\hat{a}^n_{T(s)+T^\prime(s)}(0,0)\,1_{\CE(s)}\,1_{\CE^\prime(s)}\right]
= J^1_{t_n} + J^2_{t_n}(0,0) + J^3_{t_n}(0,0),
\end{equation}
By \eqref{est2}--\eqref{Jtildelim}, the first two terms tends to zero as $n\to\infty$. By \eqref{J3bd1}--\eqref{J3bd3}, 
\begin{equation}
J^3_{t_n}(0,0) \leq C\,\int_0^{t_n} \d s\,\,\E\left[\hat{a}_{T(s)+T^\prime(s)}(0,0)\,1_{\CE(s)}\,1_{\CE^\prime(s)}\right],  
\end{equation}
which is finite in the coexistence regime, uniformly in $(t_n)_{n\in\N}$. Furthermore, recall \eqref{Andef} to estimate 
\begin{equation}
\P^{(2),n}_{(0,D_m),(0,D_m)}\left(\tau \leq t_n\right) =  \P^{(2),n}_{(0,D_m),(0,D_m)}\left(\tau \leq t_n, \CA_n\right) 
= \P^{(2)}_{(0,D_m),(0,D_m)}\left(\tau \leq t_n,\CA_n\right) \leq \P^{(2)}_{(0,D_m),(0,D_m)}\left(\tau<\infty\right).
\end{equation}
Hence, to get the claim in \eqref{e4242fin} it suffices to show that
\begin{equation}
\lim_{m\to\infty} \P^{(2)}_{(0,D_m),(0,D_m)}\left(\tau<\infty\right) = 0.
\end{equation}
But this was done in \cite[Section 6.3, Step 3]{GdHOpr1}, the idea being that two Markov processes starting from deep seed-banks are far apart when they are jointly awake for the first time.
\end{proof}

Steps 1--3 complete the proof of Lemma~\ref{lem:pres}.
\end{proof}


\subsection{Law of large numbers for the macroscopic variable}
\label{ss.LLN}

The following lemma will play a crucial role in Sections~\ref{s.fssrhofin}--\ref{s.fssrhoinffast}.

\begin{lemma}{\bf [$L_2$-convergence of the macroscopic variable]}
\label{lem:L2macr} 
For any $\theta \in [0,1]$, any $\mu(0) = \mu \in \CR_\theta$ and any sequence of times $(t_n)_{n\in\N}$ satisfying $\lim_{n\to\infty} t_n = \infty$ and $\lim_{n\to\infty} t_n/\beta_n = 0$,
\begin{equation}
\label{LLN}
\lim_{n\to\infty} \E_{\Phi_n\mu}\left[(\hat{\theta}^n(t_n) - \theta)^2\right] = 0
\end{equation}
in Model 1, and similarly for $\hat{\theta}^{M_n,n}$ in Model 2.
\end{lemma}

\begin{proof}
Recall \eqref{meanpres}. We again distinguish between Model 1 and 2.


\paragraph{Model 1.}

Use \eqref{e1311} to write  
\begin{equation}
\label{varestmac1}
\E_{\Phi_n\mu(t_n)}\left[(\hat{\theta}^n-\theta)^2\right]
=\frac{1}{(1+K)^2|\G_n|^2} \sum_{u_1,u_2 \in \S_n} K(u_1)\,K(u_2)\,
\text{Cov}_{\Phi_n\mu(t_n)}(z_{u_1},z_{u_2})
\end{equation}
with $K(u)=1$ for $u=(i,A)$ and $K(u)=K$ for $u=(i,D)$. Via \eqref{astcond} and \eqref{IADcomp} (with $\nu$ replaced by $\mu$) it follows that
\begin{equation}
\label{L2bd}
\begin{aligned}
\limsup_{n\to\infty} \E_{\Phi_n\mu(t_n)}\left[(\hat{\theta}^n-\theta)^2\right]
&= \limsup_{n\to\infty} \frac{1}{(1+K)^2|\G_n|^2} \sum_{u_1,u_2 \in \S_n} K(u_1)\,K(u_2)\,I_{t_n}(u_1,u_2)\\ 
&\leq \limsup_{n\to\infty} \frac{1}{|\G_n|^2} \sum_{i,j \in \G_n} I_{t_n}\big((i,A),(j,A)\big).
\end{aligned}
\end{equation}
We have
\begin{equation}
\label{Iav}
\begin{aligned}
\frac{1}{|\G_n|^2} \sum_{i,j \in \G_n} I_{t_n}\big((i,A),(j,A)\big)
&\leq \|g\|\, \frac{1}{|\G_n|^2} \sum_{i,j \in \G_n}
\int_{0}^{t_n} \d s\,\, \E\left[\hat{a}^n_{T(s)+T^\prime(s)}(i-j,0)
\,1_{\CE(s)}\,1_{\CE^\prime(s)}\right]\\
&= \|g\|\, \frac{1}{|\G_n|} \int_{0}^{t_n} \d s\,\, \E\left[1_{\CE(s)}\,1_{\CE^\prime(s)}\right]
= \|g\|\, \tilde{J}_{t_n},
\end{aligned}
\end{equation}
where the inequality uses \eqref{covcomp}, the first equality uses that $\sum_{k\in\G_n} \hat{a}^n_t(k,0) = 1$ for all $t \geq 0$ and all $n\in\N$, and the second equality is \eqref{est3}. We already showed that $\tilde{J}_{t_n} = o(1)$, and so this settles \eqref{LLN}.


\paragraph{Model 2.}

Use \eqref{mo3} to write, similarly as in \eqref{varestmac1},
\begin{equation}
\label{varestmac2}
\begin{aligned}
&\E_{\Phi_n\mu(t_n)}\left[(\hat{\theta}^{M_n,n}-\theta)^2\right]\\
&= \frac{1}{(1+\sum_{m=0}^{M_n} K_m)^2|\G_n|^2} \sum_{u_1,u_2 \in \S_n} 
K(u_1)\,K(u_2)\,\text{Cov}_{\Phi_n\mu(t_n)}(z_{u_1},z_{u_2})
\end{aligned}
\end{equation}
with $K(s)=1$ for $s=(i,A)$ and $K(s)=K_m$ for $s=(i,D_m)$ (recall \eqref{betaKndef}). Via \eqref{astcond} and \eqref{covexpr}--\eqref{IADcomp} (with $\nu$ replaced by $\mu$) it follows that, similarly as in \eqref{L2bd},
\begin{equation}
\label{L2bdalt}
\begin{aligned}
&\limsup_{n\to\infty} \E_{\Phi_n\mu(t_n)}\left[(\hat{\theta}^{M_n,n}-\theta)^2\right]\\
&= \limsup_{n\to\infty} \frac{1}{(1+\sum_{m=0}^{M_n} K_m)^2|\G_n|^2} 
\sum_{u_1,u_2 \in \S_n} K(u_1)\,K(u_2)\,I_{t_n}(u_1,u_2)\\ 
&\leq \limsup_{n\to\infty} \frac{1}{|\G_n|^2} \sum_{i,j \in \G_n} I_{t_n}\big((i,A),(j,A)\big).
\end{aligned}
\end{equation}
The rest of the argument is the same.
\end{proof}


\subsection{Convergence at macroscopic times conditional on the macroscopic variable}
\label{ss.macroconv}

Recall \eqref{laws}, \eqref{lawscr}, \eqref{choicelawn}, \eqref{extop} and \eqref{erglaw1}, \eqref{erglaw2fin} and \eqref{erglaw2inf}. We begin with the observation that, for all $\theta \in [0,1]$,
\begin{equation}
\label{ergast}
\begin{array}{lll}
&\CT^\mathrm{erg}_\theta = \CR^{(1)}_\theta &\text{(Model 1)},\\[0.2cm]
&\CT^\mathrm{erg}_\theta = \CR^{(2)}_\theta &\text{(Model 2 with $\rho<\infty$)},\\[0.2cm]
&\CT^\mathrm{erg,\bullet}_\theta = \CR^{(2),\bullet}_\theta &\text{(Model 2 with $\rho=\infty$)}.
\end{array}
\end{equation} 
Indeed the inclusion $\subset$ was shown in \cite[Lemma 5.7]{GdHOpr1} and \cite[Lemmas 6.6. and 6.9]{GdHOpr1}, while the inclusion $\supset$
was shown in Lemma~\ref{lem:erg}.  

For Model 2 with $\rho=\infty$ we need the topology of uniform weak convergence on the set $\CT^{\bullet,*}$ defined in \eqref{CTcov} (recall also Remark \ref{rem:topology}). We additionally define 
\begin{equation}
\label{ergbulletlaw}
\CT^{\mathrm{erg,\blacksquare}}_\theta = \left\{\mu \in \CT^{\mathrm{erg,\bullet}}_\theta\colon\,
\lim_{m\to\infty} \mathrm{Var}_{\mu}(y_{0,m}) = 0\right\},
\end{equation}
which is a subset of  $\CT^{\mathrm{erg},\bullet,*}$. 

\begin{lemma}{{\bf [Conditional convergence at macroscopic time]}}
\label{lem:presevolve}
Fix $\mu(0) = \mu \in \CT$ (Model 1, Model 2 with $\rho<\infty$) and $\mu \in \CT^\bullet$ (Model 2 with $\rho=\infty$). Fix $s>0$, and let 
\begin{equation}
\mu_n = (\Phi_n\mu)(s\beta_n).
\end{equation}
Then every weak limit point $\bar\nu$ of $(\mu_n)_{n\in\N}$ has the representation 
\begin{equation}
\label{nurepr}
\bar\nu = \int_{[0,1]} Q(\d\theta)\,\bar\nu_\theta,
\end{equation}
where 
\begin{equation}
Q = \left\{\begin{array}{ll}
\lim\limits_{n\to\infty} \CL_{\mu_n}[\hat\theta^n] &\text{{\rm (Model 1)}},\\[0.2cm]
\lim\limits_{n\to\infty} \CL_{\mu_n}[\hat\theta^{M_n,n}] &\text{{\rm (Model 2)}},
\end{array}
\right.
\end{equation}
is the associated limit law of the macroscopic variable, and 
\begin{equation}
\label{nurepralt}
\bar\nu_\theta = \left\{\begin{array}{ll}
\int_{\,\CT^{\mathrm{erg}}_{\theta}} Q_\theta(\d\mu)\,\mu &\text{{\rm (Model 1)}},\\[0.2cm]
\int_{\,\CT^{\mathrm{erg}}_{\theta}} Q_\theta(\d\mu)\,\mu &\text{{\rm (Model 2 with $\rho<\infty$)}},\\[0.2cm]
\int_{\,\CT^{\mathrm{erg},\blacksquare}_{\theta}} Q_\theta(\d\mu)\,\mu &\text{{\rm (Model 2 with $\rho=\infty$)}},
\end{array}
\right.
\end{equation}
for some Choquet measure $Q_\theta$. Moreover, $\theta \mapsto \bar\nu_\theta$ is continuous on $\CT$ in the weak topology, respectively, on $\CT^\bullet$ in the uniform weak topology.
\end{lemma}

\begin{proof}
Fix $s>0$, pick any $(t_n)_{n\in\N}$ with $\lim_{n\to\infty} t_n = \infty$ and $\lim_{n\to\infty} t_n/\beta_n = 0$, and put 
\begin{equation}
\mu^-_n = (\Phi_n\mu)(s\beta_n-t_n).
\end{equation}
Let $\bar\nu^-$ be any weak limit point of $(\mu^-_n)_{n\in\N}$. Let $Q^-=\lim_{n\to\infty} \CL_{\mu^-_n}[\hat\theta^n]$ (Model 1), respectively, $Q^-=\lim_{n\to\infty} \CL_{\mu^-_n}[\hat\theta^{M_n,n}]$ (Model 2) be the associated limit law of the macroscopic variable. Then
\begin{equation}
\label{numin1}
\bar\nu^- = \int_{[0,1]} Q^-(\d\theta)\,\bar\nu^-_\theta,
\end{equation}
where, by Choquet's theorem, 
\begin{equation}
\label{numin2}
\bar\nu^-_\theta = \left\{\begin{array}{ll}
\int_{\,\CT^{\mathrm{erg}}_{\theta}} Q^-_\theta(\d\mu)\,\mu, &\text{{\rm (Model 1, Model 2 with $\rho<\infty$)}},\\[0.2cm]
\int_{\,\CT^{\mathrm{erg},\bullet}_{\theta}} Q^-_\theta(\d\mu)\,\mu, &\text{{\rm (Model 2 with $\rho=\infty$)}}, 
\end{array}
\right.
\end{equation} 
for some Choquet measure $Q^-_\theta$.

Next we evolve the dynamics over time $t_n$ to see what happens at time $s\beta_n$. To that end we consider the limit points $\bar\nu_\theta$ of $((\Phi_n \bar\nu^-_{\theta})S_n(t_n))_{n\in\N}$, where $(S_n(s))_{s \geq 0}$ is the semigroup of the dynamics on $\G_n \times \{A,(D_m)_{0 \leq m \leq M_n}\}$. By \eqref{ergast} and Lemmas~\ref{lem:allmeans}--\ref{lem:deep}, these limit points are in $\CT^{\mathrm{erg}}_\theta$, respectively, $\CT^{\mathrm{erg},\blacksquare}_\theta$. Moreover, by Lemma~\ref{lem:L2macr}, $\hat\theta^n$ and $\hat\theta^{M_n,n}$ have the same limit points under $\mu_n^-$ and $\mu_n$. Hence \eqref{numin1} and \eqref{numin2} imply \eqref{nurepr} and \eqref{nurepralt}. 

The continuity of $\theta\mapsto\bar\nu_\theta$ uses the same coupling argument that was used in \cite{GdHOpr1} for the infinite system, and that was modified in Sections~\ref{ss.class}--\ref{ss.pres} to deal with the finite system. 
\end{proof}

The observation made in the proof of Lemma \ref{lem:erg} guarantees that 
\begin{equation}
\label{ergbulletlawalt}
\CT^{\mathrm{erg,\blacksquare}}_\theta = \left\{\mu \in \CT^{\mathrm{erg,*}}_\theta\colon\,
\lim_{m\to\infty} y_{0,m} = \theta \,\, \mu\text{-a.s.}\right\},
\end{equation}
i.e., in Model 2 with $\rho=\infty$ the deep seed-banks are deterministic in equilibrium.

In Sections~\ref{s.fssrhofin}--\ref{s.fssrhoinffast} we will see that Lemma~\ref{lem:presevolve} is the key to showing that, for any $\theta \in [0,1]$, $\mu(0) = \mu \in \CT$ and $s>0$,
\begin{equation}
\lim_{n\to\infty} \CL[(\Phi_n\mu)(s\beta_n)(\cdot \mid \hat\theta^n = \theta)] = \nu_{\theta}(\cdot),
\end{equation}
i.e., $\bar\nu_\theta = \nu_\theta$. The proof requires the use of an abstract scheme, which is outlined in Appendix~\ref{appB*}. The latter will allow us to identify the limit law of the macroscopic variable on time scale $\beta_n$. A similar statement holds for Model 2 conditional on $\hat\theta^{M_n,n} = \theta$, for $\mu(0) = \mu \in \CT$, respectively, $\mu(0) = \mu \in \CT^\bullet$. 
 

\subsection{Extension to fast growing seed-banks}
\label{ss.deep} 

For Model 2 with $\rho=\infty$ in regime (II), Lemmas~\ref{lem:erg}--\ref{lem:deep} carry over provided we restrict to time intervals of length $o(\beta^{**}_n)$ rather than $o(\beta^*_n)$ (recall that $\beta^*_n$ is the macroscopic time scale for Model 2 with $\rho=\infty$ in regime (II)). Indeed, recall from \eqref{Jtildelim} that 
\begin{equation}
\tilde{J}_{t_n} = |\G_n|^{-1} \int_0^{t_n} \d s\,(1/\kappa_{s^{1/\beta}}). 
\end{equation}
It suffices to check that $\lim_{n\to\infty} \tilde{J}_{t_n}=0$ when $t_n = o(\beta^{**}_n)$, because this allows us to carry through the estimates given in Part 4 of Step 2 in the proof of Lemma~\ref{lem:pres}. For this we refer to the four bullets below \eqref{J12def}.

The reason for the restriction is that \emph{ergodicity breaks down on time scale} $\beta^{**}_n$. In Section~\ref{s.fssrhoinffast} we will analyse Model 2 with $\rho=\infty$ in regime (II). Lemmas~\ref{lem:L2macr}--\ref{lem:presevolve} will not be needed there (even though they are trivially true on times scales $o(\beta^{**}_n)$).


\section{Proofs: $\rho<\infty$}
\label{s.fssrhofin} 

In this section we prove Theorems~\ref{T.finsys1}, \ref{thm:1657} and \ref{T.finsys2fin}. Section~\ref{ss.fssmod1} focusses on Model 1, Section~\ref{ss.fssmod2rhofin} on Model 2 with $\rho<\infty$.

\subsection{Model 1}
\label{ss.fssmod1}

The proof of Theorem~\ref{T.finsys1} will follow from the same argument as given below for Model 2 when $\rho < \infty$.

To prove Theorem~\ref{thm:1657} note that, by using the same Brownian motions for every $n\in \N$, we can realise weak convergence of the path in \eqref{gh16} as a.s.\ uniform convergence on a compact space via the Skorohod representation. Moreover, the increasing process of the macroscopic variable $\hat{\theta}^{n}$ in \eqref{mo3} is the increasing process of the active macroscopic variable $\hat{\theta}_x^n$ in \eqref{act} divided by $1+K$, and both hit the traps $0$ or $1$ if and only if their increasing process hits $0$ and remains $0$. Since $(\hat{\theta}^{n}(t))_{t \geq 0}$ is a continuous martingale, it is a time-transformed Brownian motion, with the time transformation given by its increasing process. We conclude from the path convergence that the increasing processes converge to the increasing process of the $\CF g$-diffusion. The latter has a derivative that converges to zero as the macroscopic tome $s$ tends to infinity, and so the limit path becomes constant. For that reason we can conclude that on time scale $\beta_n$ the hitting time of the traps by the path in \eqref{gh16} converges to the hitting time of the traps for the limit process, which is the $\CF g$-diffusion.


\subsection{Model 2: $\rho<\infty$}
\label{ss.fssmod2rhofin}

In this section we prove Theorem~\ref{T.finsys2fin}. The proof is built on an \emph{abstract scheme} for deriving the finite-systems scheme, developed in \cite[Section 1]{CG94*} and \cite[Section 4]{DGV95} for general spatial systems and summarised in Appendix~\ref{appB*}. This abstract scheme has been applied to several classes of systems, but not yet to systems with seed-banks. 

We will see that we can use this abstract scheme by \emph{incorporating the seed-bank into the single-component state space}, namely, by considering the state space $I^\G$ with $I = [0,1] \times [0,1]^{\N_0}$, respectively, $I^{\G_n}$ with $I = [0,1] \times [0,1]^{M_n+1}$ (see Remark~\ref{rem:translate}). Along the way, various ingredients need to be specified, for which we use the lemmas derived in Section~\ref{s.classpres}. In particular, we need to show that the macroscopic time scale chosen in \eqref{e2033alt}, namely, $\beta_{M_n,n} = \kappa |\G_n|$ with $\kappa=(1+\rho)^2$, emerges as the correct scaling time (which is related to the verification of Assumption (A10) in Corollary~\ref{cor.3636} in Appendix~\ref{appB*}). 

We will see that the abstract scheme allows for a \emph{bootstrappig argument}. Namely, we first use tightness of the macroscopic variable process to establish that the finite system locally converges to the equilibrium of the infinite system with a density that is given by the limiting value of the macroscopic variable. Afterwards, we use this equilibrium to identify how to macroscopic variable process evolves on the macroscopic time scale. 

The proof is organised into 5 Steps. In Steps 1--3 we first carry out the proof for $M_n \equiv M$ and $g=dg_{\text{FW}}$, $d \in (0,\infty)$. In Steps 4--5 we let $M_n\to\infty$ and consider general $g\in\CG$.


\paragraph{Step 1: Checking the assumptions.}

We must check Assumptions (A1)--(A9) in Theorem~\ref{th.3624} and assumption (A10) in Corollary~\ref{cor.3636}. We proceed item per item, after first  checking that our set-up fits into the abstract scheme. 

\medskip\noindent
$\bullet$ {\bf Set-up.} Our geographic space is a countable Abelian group $\G$ endowed with the discrete topology, while our single-component state space $I=[0,1]\times[0,1]^{M+1}$ is a Polish space equipped with the product topology of $[0,1]$. Our assumption that $\G$ is \emph{profinite} yields a projective system $(\G_n)_{n\in\N}$ of finite groups endowed with the discrete topology. Our full state spaces $E,E_{M,n}$ are $I^\G$, $I_M^{\G_n}$. We can apply Theorem~\ref{th.3624}, provided we choose the initial law appropriately, namely, from the class $\CR^{(2)}_\theta$ introduced in Section~\ref{ss.class}, which is the domain of attraction of the equilibrium $\nu_\theta$ and which, by Lemma~\ref{lem:pres}, is preserved on time scale $o(\beta_{M,n})$. 


\paragraph{(A1).}

If $g(x)=dg_{\text{FW}}$, $d \in (0,\infty)$, then the dual on $\G_n$ is the spatial coalescent with truncated transition kernel $a_n(\cdot,\cdot)$ for the random walk on $\G_n$. For this process it is straightforward to see that, as $n \to \infty$, the random walk on $\G_n$ converges to the random walk on $\G$ because $a_n(\cdot,\cdot)$ converges pointwise to $a(\cdot,\cdot)$. Therefore $b^n(\cdot,\cdot)$ converges to $b(\cdot,\cdot)$, and also the dual lineages converge. The seed-bank leads to different waiting times until the next jump occurs in the dual. Therefore, as $n\to\infty$, the spatial coalescent on $\G_n$ converges in law to the spatial coalescent on $\G$. Hence, by the duality relation and the fact that the moments are convergence determining, as $n\to\infty$ the forward process on $\G_n$ converges to the forward process on $\G$ in the sense of marginal distributions. 


\paragraph{(A2).}

Let
\begin{equation}
\label{e2657}
\begin{aligned}
\CT &= \text{set of translation invariant probability measures on } I^\G \text{ under the group action on } \G,\\ 
\CT_n &= \text{set of translation invariant probability measures on } I^{\G_n} \text{ under the group action on } \G_n. 
\end{aligned}
\end{equation}
In our context, $\CT$ is defined in \eqref{laws}, while $\CT_n$ is the image of $\CT$ under the projection given by \eqref{choicelawn}. Since the evolution mechanism is translation invariant, $\CT_n$ and $\CT$ are preserved under the semigroup of the evolution, and so (A2) holds.


\paragraph{(A3).}

The Polish space $\CJ$ is the set $[0,1]$ labelling the density, and $\CT^\mathrm{erg}_\theta$ is the set of translation invariant probability measures that are \emph{ergodic} under translation with density $\theta$ given in \eqref{erglaw2fin}. Thus, (A3) is a consequence of \cite[Theorem 3.2]{GdHOpr1}.


\paragraph{(A4).}

Define the statistics
\begin{equation}
\label{e2872}
\hat{\theta}^{\,M,n} = \frac{1}{|\G_n|} \sum_{i\in\G_n} \frac{x_i
+\sum_{m=0}^M K_m y_{i,m}}{1+\sum_{m=0}^M K_m},
\end{equation}
and recall the \emph{empirical measure} $\CE_{M,n}$ defined in \eqref{e1297}. Recall that $\hat\theta^{M,n}$ is the \emph{average} w.r.t.\ $\CE_{M,n}$ of
\begin{equation}
\label{e4081}
\frac{x_0 +\sum_{m=0}^M K_m y_{0,m}}{1+\sum_{m=0}^M K_m}.
\end{equation}
We get the definition of $\CJ$ in (A3) from the general ergodic theorem in \cite[Section 6.4]{Krengel85}.


\paragraph{(A5).}

The continuity property of the statistic $\hat{\theta}^{\,M,n}$ and the ergodic behaviour of the infinite system can be deduced with the help of $L^2$-theory and moment calculations, given by Lemma~\ref{lem:deccor}.


\paragraph{(A6).}

For the time scale choose
\begin{equation}
\label{e2676}
\beta_{M,n} = \kappa_M |\G_n|, \qquad \kappa_M = \left(1+\sum_{m=0}^M K_m\right)^2.
\end{equation}
We have to show that, with $z^{M,n}=(z_{i,m})_{i \in \G_n,0 \leq m \in M}$, 
\begin{equation}
\label{e2680}
\CL \left[\left(\hat{\theta}^{\,M,n}\big(z^{M,n}(t \beta_{M,n})\big)\right)_{t \geq 0}\right], \qquad n \in \N, 
\end{equation}
is tight in the space of paths $C([0,\infty),[0,1])$. For that purpose we consider the \emph{increasing process} on time scale $\beta_{M,n}$, and compute (recall \eqref{e1732})
\begin{equation}
\label{e2694}
\langle \hat{\theta}^{\,M,n} \rangle = \left\langle \hat \theta \big(z^{M,n}(t\beta_{M,n})\big)\right\rangle_{t \geq 0} 
= \left(\int_0^t \d s\,\frac{1}{|\G_n|} \sum_{i \in \G_n} g\big(x_i(s\beta_{M,n})\big)\right)_{t \geq 0}.
\end{equation}
Our task is to show that as $n\to\infty$ the right-hand side converges to 
\begin{equation}
\label{limproc}
\left(\int_0^t \d s\, (\CF g)\big(\Theta(s)\big)\right)_{t \geq 0},
\end{equation} 
with $(\Theta(u))_{s \geq 0}$ a non-trivial limit process that starts at $\theta$ and as $u \to\infty$ tends to $0$ or $1$, in other words, the choice of time scale $\beta_{M,n}$ in \eqref{e2676} is proper. The associated processes for finite $n$ are \emph{continuous martingales}, namely, time-changed Brownian motions, and the increasing processes are bounded and continuous with a bounded derivative on finite time intervals (uniformly in $n$). Hence the time-changed Brownian motions are tight in $C([0,\infty),\R)$, and so are the martingales. Since $s \mapsto\Theta(s)$ is continuous, it follows that
\begin{equation}
\label{e2684}
s \mapsto \E \left[f \left(\Theta(s)\right)\right] \text{ is continuous for } f \in C([0,1],\R),
\end{equation}
for every weak limit point arising from \eqref{e2680}, where $\E$ is with respect to the law of $\Theta(s)$. 


\paragraph{(A7).}

Translation invariance of laws is preserved under weak convergence. Tightness follows from compactness of the state space.


\paragraph{(A8).}

We choose to work with coupling rather than with duality, because this will be convenient and instructive later, when we include general $g$. In \cite{CGSh95}, assumptions (A8) and (A9) are verified for the model without seed-bank when $\CT_n,\CT$ are translation invariant laws, based on the construction of a coupling of the two finite or the two infinite processes, starting from different initial points, and a coupling of the finite and the infinite system, with the finite system starting in the translation invariant version of the restriction of the infinite system. The coupling is done by constructing the two processes as strong solutions of an SSDE with the same Brownian motions. 

In order to prove (A8), it suffices to have a coupling and a Lyapunov function such that the distance between the two coupled processes can be controlled by a Lyapunov function with non-increasing expectation. For the model with seed-bank this coupling was constructed in \cite[Section 5.3]{GdHOpr1}. Hence, we indeed have (A8). 


\paragraph{(A9).}

In order to prove (A9) (i.e., to establish ergodicity), we proceed as was done for systems without seed-bank in \cite[Proposition 2.4]{CGSh95}. In particular, we define the bivariate process 
\begin{equation}
\label{bivar}
(z^{M,n},z)
\end{equation}
as the strong solutions to the corresponding SSDEs in \eqref{gh1b*}--\eqref{gh2b*}, respectively,  \eqref{gh1**}--\eqref{gh2**} with the same collection $(w_i(t))_{t \geq 0}$, $i \in \G$, of standard Brownian motions, starting from initial laws in $\CT_n,\CT$ that are linked as in \eqref{choicelawn}. As Lyapunov function we use the same quantity as in \cite[Lemma 5.8]{GdHOpr1}. Because of (A8), it suffices to prove Lemma~\ref{l.2752} below. 


\paragraph{(A10).}

This assumption will be verified in Step 3. 


\paragraph{Step 2: Coupling the finite system and the infinite system.}

To compare the finite and the infinite system, we apply the coupling techniques that were developed in \cite{GdHOpr1} to deal with different initial laws and in \cite{GdHOpr3} to deal with different dynamics. Abbreviate $\S = \G \times \{A,(D_m)_{m\in\N_0}\}$, $\S_M = \G \times \{A,(D_m)_{0 \leq m \leq M}\}$ and $\S_{M,n} = \G_n \times \{A,(D_m)_{0 \leq m \leq M}\}$.      

\begin{lemma}{{\bf [Comparison finite-infinite]}}
\label{l.2752}
Let $(z^{M,n},z)$ be the bivariate process on $E_{M,n} \times E$ defined in \eqref{bivar}, with initial laws in $\CT_n,\CT$ linked as in \eqref{choicelawn}. Write $z=(z_u)_{u \in \S}$ and $z^{M,n}  = (z_u)_{u \in \S_{M,n}}$, and denote by $\Lambda_{M,n}$ the diagonal of $\S_{M,n} \times \S_{M,n}$. Let $\CV^{M,n}$ be the collection of all probability measures $\bar\mu$ on $E_{M,n} \times E$ such that
\begin{equation}
\bar \mu\left(z^{M,n}_u = z_u\,\,\forall\,u \in \Lambda_{M,n}\right)=1.
\end{equation} 
Then
\begin{equation}
\lim_{n\to\infty} \sup_{\bar \mu \in \CV^{M,n}} \E_{\bar \mu}|z^{M,n}_u(t)-z_u(t)| = 0
\qquad \forall\,u \in \S_M \times \S_M,\,t >0. 
\end{equation} 
In particular, there exists a sequence of times $(\ell_n)_{n\in\N}$, satisfying $\lim_{n\to\infty} \ell_n = \infty$, such that
\begin{equation}
\label{e.2787}
\lim_{n\to\infty} \sup_{\bar \mu \in \CV^{M,n}} \E_{\bar \mu} |z^{M,n}_u(\ell_n)-z_u(\ell_n)| = 0 
\qquad \forall\,u \in \S_M \times \S_M.
\end{equation}
\end{lemma}

\begin{proof}
We have to estimate the \emph{Lyapunov function}
\begin{equation}
\label{e2760}
t \mapsto \E \left( |x^{M,n}_i(t)-x_i(t) | + \sum_{m=0}^M K_m\,|y^{M,n}_{i,m}(t)-y_{i,m}(t)|\right), \qquad i \in \G,
\end{equation}
and show that the right-hand side tends to zero as $t\to\infty$. Using the computations on the coupling in \cite[Section 5.3]{GdHOpr1}, carried out for the infinite system starting from different initial configurations, we can proceed exactly as in the proof of \cite[Proposition 2.4(a)]{CGSh95}, replacing the Green function $\hat{G}(0,0)$ of $\hat a(\cdot,\cdot)$ appearing there by the hazard integral $\hat{B}(0,0)$ defined in \eqref{hazard}. The fact that the coupling is successful follows from Lemmas~\ref{lem:erg}--\ref{lem:pres}. 
\end{proof}


\paragraph{Step 3: Completion of the proof for $M<\infty$.}

We verify Assumption (A10) in Corollary~\ref{cor.3636} in Appendix~\ref{appB*}, which will complete the proof of Theorem~\ref{T.finsys2fin} for the case where $M < \infty$. A key observation made in \cite[Section 5.1]{GdHOpr1} is the 
fact that 
\begin{equation}
\label{martdef}
\CM_{M,n}(t) = \lim_{n\to\infty} \frac{1}{|\G_n|} \sum_{i\in\G_n} 
\left[\frac{x^{M,n}_i(t) + \sum_{m=0}^M K_m y_{i,m}^{M,n}(t)}{1+\sum_{m=0}^M K_m}\right], \qquad t \geq 0,
\end{equation}
is a martingale that has \eqref{e2694} as increasing process. From this we want to conclude that: (i) the time-transformed processes converge in law and hence so do the martingales themselves; (ii) the limit is the diffusion with the claimed diffusion function. 

No new ideas are needed. Lemmas~\ref{l.3734}--\ref{l.3759} below are versions of \cite[Lemmas 4.8-4.10]{DGV95} reformulated for our purposes (see also Appendix~\ref{appB*}). Their proof is a straightforward adaption of the proofs given in \cite[Section 4(b)]{DGV95}.  Both (i) and (ii) can be treated exactly as in the proofs of~\cite[Lemmas 4.8--4.10]{DGV95}, after we replace the martingale $(\Theta(s))_{s \geq 0}$ used there by the martingale in \eqref{martdef} adapted to the seed-bank, and use the facts derived in \cite[Section 5.3]{GdHOpr1}. Recall the definition of $\CT=\CT([0,1]^{\G \times \N_0})$ in \eqref{laws} and $\CE^{M,n}$ in \eqref{e1721}.

\begin{lemma}{{\bf [Tightness of the macroscopic variable]}}
\label{l.3734}
(a) The sequence
\begin{equation}
\label{e3736}
\CL \left[\left(\left\{\int_0^t \d s\,\frac{1}{|\G_n|} \sum_{i \in \G_n} 
g\big(x^{M,n}_i(s\,\beta_{M,n})\big)\right\},\hat{\theta}_n(t\,\beta_{M,n})\right)_{t \geq 0} \right], \qquad n \in \N,
\end{equation}
is tight in $C([0,\infty),[0,\infty) \times [0,1])$.\\
(b) The sequence $\CL[(\hat \theta^{M,n}(t\,\beta_{M,n}))_{t \in \geq 0}]$, $n\in\N$, is tight in $C([0,\infty),[0,1])$.\\
(c) The sequence $\CL [(\int^t_0 \d s\,\CE^{M,n}(s\beta_{M,n}))_{t \geq 0}]$, $n\in\N$, is tight in $C([0,\infty), 
\CT ([0,1]^{\G \times \{0,\ldots,M\}}))$. 
\end{lemma}

\begin{lemma}{{\bf [Convergence of the macroscopic variable]}}
\label{l.3745}
Suppose that $(n_k)_{k \in \N}$ is such that
\begin{equation}
\label{e3747}
\lim_{k\to\infty} \CL \left[\Big(\hat \theta^{M,n_k}(t\,\beta_{M,n_k})\Big)_{t \geq 0}\right]
= \CL \left[(\Theta(t))_{t \geq 0}\right].
\end{equation}
Then
\begin{eqnarray}
\label{e3761}
&&(\Theta(t))_{t \geq 0} \quad \mbox{is a continuous martingale},\\
\label{e3764}
&&\left(\Theta(t)^2-\int_0^t \d s\,(\CF g)(\Theta(s))\right)_{t \geq 0} \quad \mbox{is a martingale}.
\end{eqnarray}
\end{lemma}

\begin{lemma}{{\bf [Identification of the scaling limit]}}
\label{l.3759}
The martingale problem in \eqref{e3761}--\eqref{e3764} has a unique solution, namely, the diffusion with diffusion function $\CF g$.
\end{lemma}

Lemma~\ref{l.3745} guarantees that the time scale $\beta_{M,n}$ is proper, i.e., $(\Theta(t))_{t \geq 0}$ is a non-trivial random process that starts at $\theta$ and eventually converges to either $0$ or $1$.  

Lemmas~\ref{l.3745}--\ref{l.3759} imply the assumption in Corollary~\ref{cor.3636}, as proved in \cite[Section 4(d)]{DGV95} with the help of stochastic analysis of semi-martingales. To prove \eqref{e2680}, we can follow the strategy in \cite[Section 4(d)]{DGV95} with only minor changes. To identify the time scale $\beta_{M,n}$ and verify \eqref{e3764}, we start with the following observation about the dual process. The fraction of time the two random walks are jointly active equals $\kappa=1/(1+\rho)^2$. Therefore the mixing property in \eqref{pnmix} tells us that $\kappa|\G_n|$ is the average time it takes the two random walks to meet on $\G_n$ and be jointly active. When they do, they coalesce a rate $d$, and so coalescence eventually occurs on time scale $\beta_{M,n} = \kappa|\G_n|$. 

The proof of Lemmas~\ref{l.3734}--\ref{l.3759} follows the line of argument in \cite[Section 4(d)]{DGV95}, and rests on Lemmas~\ref{lem:erg}--\ref{lem:pres} and \ref{lem:L2macr}--\ref{lem:presevolve}.  


\paragraph{Step 4: Extension to $M=\infty$.}

We show how to drop the assumption $M < \infty$. Since $\rho < \infty$, $\hat{\theta}^{n}$ can be approximated uniformly by $\hat{\theta}^{\,M,n}$ as $M\to\infty$ uniformly in $n$ on the time scale $\beta_n$, and we do not need to change the scale other than by the constant $\kappa=(1+\rho)^2$. Therefore we have uniform convergence to the model with $M=\infty$ when $\lim_{n\to\infty} M_n = \infty$. Hence the same argument as in Steps 1-3 goes through, provided we show that the equilibrium for the model with $M_n$ converges to the equilibrium for the model with $M=\infty$. For $g = d g_{\rm FW}$, $d \in (0,\infty)$, this follows from duality, because all moments of the equilibria converge. 

Indeed, denote by $(\Theta^M(s))_{s \geq 0}$ the process in \eqref{gh18} for Model 2 with $M$ seed-banks, and let $\nu_\theta^M$ be its equilibrium on $\G$. We need that
\begin{equation}
\label{e3801}
\lim_{M\to\infty} \nu^M_\theta = \nu_\theta
\end{equation}
and
\begin{equation}
\label{e3804}
\lim_{M\to\infty} \CL \left[(\hat{\theta}^M(t))_{t > 0}\right]
= \CL \left[(\Theta (t))_{t > 0}\right],
\end{equation}
where the latter is the $(\CF g)$-diffusion starting in $\theta$. Clearly, \eqref{e3801} implies \eqref{e3804}. Indeed, \eqref{e3801} yields $\lim_{M\to\infty} \CF_M g = \CF g$ (with $\CF_M$ the renormalisation map for the model with $M$ seed-banks). This in turn yields convergence of the solution of the corresponding martingale problem (by \cite[Lemma 5.1, Chapter 4]{EK86} and \cite[Proposition 3.2.3]{JM86}). It remains to show \eqref{e3801}. For $g=dg_{\rm FW}$, $d \in (0,\infty)$, this follows from convergence of the dual process, which implies convergence of moments. 


\paragraph{Step 5: Extension to general $g$.}

Finally, we point out how to modify the argument for general $g$. Duality entered into the proof of (A1), and also in the proof that the limiting point arising in (A6) leads to the dynamics given by \eqref{e3764}. The former arises in the standard construction of the model (see \cite{GdHOpr1}). The latter requires us to verify that, conditional on $\hat{\theta}^{M,n}$, the increasing process defined in \eqref{e2694} satisfies a law of large numbers in $L_2$. This verification is based on the property that $\nu_\theta$ is ergodic for every $\theta$, irrespectively of $g\in\CG$ (see \cite{GdHOpr1}), which follows from Lemmas~\ref{lem:allmeans}--\ref{lem:deccor}. As to the other assumptions, without duality we can work with the generator and verify these assumptions directly with the help of Lemma~\ref{lem:presevolve}. See \cite[Section 1]{DGV95}. We also need that $\theta \mapsto \nu_\theta$ is continuous, a fact that was established in \cite[Section 6.3.1, Lemma 6.6]{GdHOpr1} with the help of coupling.


\section{Proofs: $\rho=\infty$ and slow growing seed-bank}  
\label{s.fssrhoinfslow}

In this section we prove Theorem~\ref{T.finsys2inf(1)}. We follow the same line of argument as in Section~\ref{s.fssrhofin}, but with a number of adaptations. In order to apply the abstract scheme in Appendix~\ref{appB*}, it no longer suffices to incorporate the seed-bank via an extension of the single-component state space, as was done for $\rho<\infty$ in Section~\ref{s.fssrhofin}. Instead, we need to \emph{extend the geographic space with the seed-bank space}, namely, consider the state space $I^\S$ with $I=[0,1]$, respectively, $I^{\S_n}$ with $I=[0,1]$ (see Remark~\ref{rem:translate}). We will see below that, because the active population is a negligible fraction of the total population, new arguments are needed to control the deep seed-banks, based on the computations in Section~\ref{s.classpres}. As shown in \cite{GdHOpr1}, the validity of the ergodic theorem for the infinite system requires the additional assumption of \emph{colour regularity} for the initial law. 

The main change is that the macroscopic variable and the active macroscopic variable (recall \eqref{mo3} and \eqref{act}) evolve on \emph{different} times scales, namely, $\kappa_{M_n}|\G_n|$, respectively, $|\G_n|$ (provided the initial law is non-degenerate). Since the former is asymptotically larger than the latter, under certain conditions this opens up the possibility that correlations over typical distances in $\G_n$ change via migration before the macroscopic variable is able to move via the Brownian motions, leading to a break down of ergodicity in the geographic space. Consequently, if we try to follow the abstract scheme as for $\rho<\infty$ in Section~\ref{s.fssrhofin}, then we possibly run into problems. 

Fix a sequence $(M_n)_{n\in\N}$ such that $\lim_{n\to\infty}$ $M_n=\infty$. We will see that, as long as $M_n$ is in the slow growth regime (Regime (I) in \eqref{Mncond}), the proper time scale is $\beta_n=\kappa_{M_n}|\G_n|$. The factor $\kappa_{M_n}=(1+\sum^{M_n}_{m=0} K_m)^2$, which is the square of the size of the seed-bank, compensates for the fact that the seed-bank slows down the evolution and reduces the volatility. The proper time scale for the macroscopic variable must be calculated by looking at the \emph{increasing process} $\langle\hat{\theta}^{\,M_n,n}\rangle$ and using It\^o's formula. In particular, we must check tightness of $(\hat{\theta}^{\,M_n,n}(\cdot\,\beta_n))_{n\in\N}$ via \eqref{e2872}. Because $g$ is bounded, the expression we obtain for the increasing process at time $s\beta_n$ is bounded by $\|g\|_\infty s$. Our goal is to prove the analogue of \eqref{limproc} and show that the limit proces $(\Theta(s))_{s \geq 0}$ is non-trivial.  For that we need to show that, on time scale $\beta_n$, $(\Theta(s))_{s \geq 0}$ has non-zero fluctuations, which gradually vanish as the boundary $\{0,1\}$ is approached or hit. The latter is possible because of the slow growth of $M_n$.

To ensure that the line of argument used for $\rho<\infty$ and $M < \infty$ in Section~\ref{s.fssrhofin} carries over to $\rho=\infty$ and $M = \infty$, we need the \emph{new setting} of the abstract scheme mentioned above (extension of the geographic space with the seed-bank space). This also requires an adaptation of the sets from which we draw the initial law. For $\rho<\infty$ these were the sets $\CT,\CT^{\mathrm{erg}}$ defined in \eqref{laws} and $\CT^{\mathrm{erg}}_\theta$ defined in \eqref{erglaw1}, \eqref{erglaw2fin}. For $\rho=\infty$ these sets must be replaced by $\CT^\bullet,\CT^{\mathrm{erg},\bullet}$ defined in \eqref{lawscr} and $\CT^{\mathrm{erg},\bullet}_\theta$ defined \eqref{erglaw2inf}. Similarly, $\CR^{(2)}_\theta \subset  \CT^{\mathrm{erg}}_\theta$ must be replaced by $\CR_\theta^{(2),\bullet} \subset \CT^{\mathrm{erg},\bullet}_\theta$ (recall Definition~\ref{def:class2}). 

In this setting, Assumptions (A1)--(A3) are straightforward in view of the results for the infinite system summarised in Section~\ref{ss.coreinf} (see the paragraph on equilibria) in combination with the preservation property stated in Lemma~\ref{lem:pres}, which uses and extends the computations carried out in \cite{GdHOpr1}. Assumptions (A4)--(A10), however, are not straightforward. There are \emph{four tasks}: 
\begin{itemize}
\item[{\bf (A)}] 
To settle (A4)--(A8), the key obstacle is that the seed-banks are not translation invariant and ergodic in the seed-bank direction. We use the results in Section~\ref{s.classpres} to gain control over the `deep' seed-banks. To verify Assumptions (A4)--(A7), for $\rho=\infty$ we take into account that, in order to control the long-time behaviour of the infinite system, we must require that the initial law $\mu(0)\in\CT^\mathrm{erg}$ (recall \eqref{erglaw2inf}) is \emph{colour regular}, i.e., $\mu(0)\in\CT^{\mathrm{erg},\bullet}$ (recall \eqref{covcond1}). Furthermore, to verify Assumption (A8), we show that the \emph{colour regularity is preserved over time} and that the \emph{conserved quantity} is \emph{continuous in the initial state}. Consequently, we have continuity of the ergodic theorem and we can apply the ergodic theorem that is known for the infinite system.
\item[{\bf (B)}] 
To verify Assumption (A9), we first use a coupling argument to prove \emph{uniformity of the ergodic theorem} in the \emph{initial law}, over a time stretch $o(\beta_n)$. In other words, we show that for the slow growing seed-bank the problem of the breaking of ergodicity mentioned above does not occur, and the macroscopic variable depends continuously on the initial state. 
\item[{\bf (C)}]
To complete the verification of Assumption (A9), we next check that the approximation of the infinite system by the finite system in Lemma~\ref{l.2752} still holds in modified form for the slow growing seed-bank when $\rho=\infty$. For this we construct a coupling of the two systems that is successful because of the properties established in Task (B), and use that as time proceeds the deep seed-banks become deterministic. 
\item[{\bf (D)}]
To verify Assumption (A10), we deal with the fact that $\hat{\theta}$ no longer is a continuous functional of the empirical measure, since the active population is negligible in the limit.   
\end{itemize}
We will address these items in Sections~\ref{ss.firsttaskI}--\ref{ss.fourthtaskI}.


\subsection{Task (A)}
\label{ss.firsttaskI}

We need to verify Assumptions (A4)--(A8). In order to transfer the ergodic theorem that is known for the infinite system to the sequence of finite systems, we use a restarting argument that requires modified versions of Assumptions (A4) and (A8) where the initial laws are drawn from $\CR^{(2),\bullet}_\theta$ for some $\theta \in [0,1]$ (see above) rather than from $\CT^{\mathrm{erg},\bullet}_\theta$, as was done in Section~\ref{s.fssrhofin}. In particular, we need the fact that all weak limit points along sequences of times that are $o(\beta_n)$ satisfy the \emph{Liggett conditions} (recall Definitions~\ref{def:class1}--\ref{def:class2}) and are \emph{colour regular} (recall\eqref{covcond1}), as was shown in Lemma~\ref{lem:pres}.    

Once we have established (A4)--(A5), (A6) will follow from the fact that $\hat{\theta}^{M_n,n}(t\beta_n)$ is a continuous-path martingale with bounded increasing process given by \eqref{e1732}. In what follows we first verify on (A4), (A7) and then (A5), (A8).

To verify (A4) we argue as follows. For the system on $\S_n$ we determine the random variable $\Theta$ (the macroscopic variable) corresponding to the current state as required for (A4) via the estimator $\hat{\theta}^{\,M_n,n}$, which determines an associated equilibrium measure for the evolution on $\S$ that is our candidate for the limit of our sequence of scaled finite systems. In particular, (A4) requires that the estimator $\hat{\theta}^{\,M_n,n}$ defined in \eqref{e2872} converges in probability to a number $\Theta$ as $n\to\infty$, which for ergodic and colour regular initial laws is the parameter $\Theta$ selecting the extremal equilibrium measure $\nu_\Theta$ in whose domain of attraction the system is. Assumption (A4) will follow for $\rho=\infty$, $M=\infty$ if we can establish $L_2$-convergence of the average in \eqref{e4081} as $M \to \infty$, under all possible current laws of our system, i.e.,
\begin{equation}
\label{e4229}
\lim_{n\to\infty} \CL[\hat{\theta}^{\,M_n,n}(t)] = \CL[\hat{\theta}^{\,\infty,\infty}(t)] \qquad \text{ in } L_2 \qquad \forall\,t \geq 0,
\end{equation}
for some random variable $\hat{\theta}^{\,\infty,\infty}(t)$, playing the role of the macroscopic variable associated with the limit law of the finite system as $n\to\infty$.   

As shown in \cite[Section 6]{GdHOpr1}, subject to \eqref{thetadefaltalt} we know that the infinite system is $L_2$-ergodic and that
\begin{equation}
\label{thetaatt}
\lim_{M\to\infty}  \E_{\mu(t)}\left[\frac{x_0 + \sum_{m=0}^M K_m\,y_{0,m}}
{1+\sum_{m=0}^M K_m}\right] = \theta \qquad \forall\,t \geq 0,
\end{equation}
where $\mu(t)$ is the law of the system at time $t$. Via Lemma~\ref{lem:L2macr} this yields Assumption (A4). 

Assumption (A7) amounts to showing that colour regularity is preserved under the evolution up to times of order $\beta_n$. But this follows from Lemma~\ref{lem:presevolve} via ergodic decomposition.

To settle Assumptions (A5) and (A8), which we treat together, we actually need more, namely, the existence of the limit in \eqref{thetaatt} for \emph{sequences} of initial laws $(\mu_n)_{n\in\N}$ induced by the state of the finite system at time $t\beta_{M_n,n}$. Therefore two complications arise: 
\begin{itemize}
\item[(1)]
In \eqref{thetaatt} we need convergence in law rather than in expectation.
\item[(2)]
We need that \emph{weak limit points} as $n \to \infty$ have the property that the limit of the macroscopic variable as $M \to \infty$ exists, i.e., $\hat{\theta}^{\,M_n,n}(s \beta_{M_n,n})$ converges in $L_2$ (and hence in law) to a random variable $\Theta(s)$, and conditional on $\Theta(s)$ colour regularity holds. 
\end{itemize}
We want to show that, for a subsequence $(n_k)_{k \in \N}$ and for the law of the pair $((x^{M_n,n},y^{M_n,n}),\hat{\theta}^{\,M_n,n})$ evaluated at time $s \beta_{M_n,n}$, we have
\begin{equation}
\label{e4241}
\lim_{k\to\infty} \CL_\nu \left[\hat{\theta}^{\,M_{n_k},n_k}(t \beta_{M_{n_k},n_k})\right] 
= \CL_{\nu(t)} [\Theta (t)] \quad \text{ for some limit law } \nu(t). 
\end{equation}
In fact, we want to show the sharper statement
\begin{equation}
\label{e4250}
\lim_{k\to\infty} \CL_{\nu_{n_k}} \left[\hat{\theta}^{\,M_{n_k},n_k}(t \beta_{M_{n_k},n_k})\right] = \CL_{\nu(t)} [\Theta (t)]
\qquad \forall\, (\nu_{n_k})\colon\,\lim_{k\to\infty} \nu_{n_k} = \nu(t).
\end{equation}
This we will get by coupling the evolution on $\S_n$ and the evolution on $\S$, which we provide in Task (C) via a restart argument for $\nu_{n_k}$. For the latter we use the tightness part of Assumption (A6), which gives us convergence in law of the process $\hat\theta^{\,M(n_k),n_k}$ as $k\to\infty$ along a subsequence $(n_k)_{k\in\N}$. From \eqref{e4241}--\eqref{e4250} we conclude that
\begin{equation}
\label{e4245}
\lim_{n\to\infty} \nu_n(t) = \CL \left[ \nu_{\Theta (t)} \right],
\qquad \nu_n(t) = \CL \left[ \left(x^{M_n,n}\big(t\beta_{M_n,n}\big), 
y^{M_n,n}\big(t\beta_{M_n,n}\big)\right)\right],
\end{equation}
which gives Assumption (A5). Note that $\nu_n(s)$ is concentrated on configurations in $\S$ that are \emph{periodic} in space and \emph{constant} (equal to the estimator) in the seed-bank beyond colour $M_n$, since they are lifted from $\S_n$ to $\S$. 

To prove \eqref{e4241}, recall from Section~\ref{ss.fssmod2rhofin} that for $\rho<\infty$ we used $L_2$-theory of stationary random fields. For $\rho=\infty$ this needs to be amended. We need to show that $\hat\theta^{\,M_n,n}(s \beta_n)$ converges in law and in $L_2$ to a random variable $\Theta(s)$ as $n \to \infty$. We also need to show that $\lim_{m\to\infty} \E_{\nu(s)}[y_{i,m}]=\theta$, i.e., on the set of invariant laws $\CI$  (recall \eqref{laws}) the property in \eqref{e4242} is valid even in the limit as $n \to \infty$. But this was already settled in Lemma~\ref{lem:presevolve}. Since the averages take values in $[0,1]$, we have compactness of their laws, so that along a further subsequence we have weak convergence of their laws. The limit point must have mean $\theta$, which arises from the existence as $M \to \infty$ of the mean of the initial law (recall \eqref{thetadefaltalt}), which is an immediate consequence of first moment computations via the dual.

To prove \eqref{e4250}, which is a \emph{continuity property} of the limiting $L_2$-average as a function of $\nu_{n_k}$, we use that, by Lemma~\ref{lem:deep}, the deep seed-banks are \emph{deterministic uniformly in} $n$. In fact, for any sequence of initial measures $(\nu_n)_{n\in\N}$ such that $\lim_{n\to\infty} \nu_n = \nu$ for some $\nu$, we also have that $\lim_{n\to\infty} \hat{\theta}^{M_n,n}(s\beta_n) = \theta(s)$ for some $\theta(s)$. To prove \eqref{e4245}, note that every ergodic component of the state at time $s \beta_n$ satisfies $\lim_{n\to\infty} \lim_{m\to\infty} \E_{\nu_n}[y_{0,m}]=\theta(s)$, which fixes $\nu_{\theta(s)}$, the corresponding equilibrium approached after a restart. This is precisely the requirement in Assumption (A8), and altogether settles Assumption (A5).


\subsection{Task (B)}
\label{ss.secondtaskI}

We prove that, once we have a successful coupling as in Lemma~\ref{l.2752} (which will be established in Task (C)), we can verify Assumption (A9). Here, we need to decompose the sequence $(\mu_n)_{n\in\N}$ and its limit points $\mu$ into ergodic components by using Lemma~\ref{lem:presevolve}. For each of the ergodic components we can apply the version of Lemma~\ref{l.2752} adapted to $\rho=\infty$, to get that $\hat{\theta}(\mu_n(t_n))$ and $\mu_n(t_n)$ have limit points along the \emph{same} subsequential limits. This implies that we can replace $\mu_n(t_n)$ by $\mu(t_n)$ in the limit as $n\to\infty$. 


\subsection{Task (C)}
\label{ss.thirdtaskI}    

We can define the bivariate dynamics in the same way as we did for Lemma~\ref{l.2752}, namely, as in \eqref{bivar} with $M=M_n$ for the finite system and $M=\infty$ for the infinite system. But the quantity in \eqref{e2760} with $M$ replaced by $M_n$ is no longer well defined after we have passed to the limit $n\to\infty$ for the bivariate system, except for configurations drawn for a \emph{restricted class} (which we will achieve by restricting the class of initial laws). This must be resolved through the fact that the deep seed-banks become deterministic after a long time. Two problems come up (which for the infinite system were resolved in \cite{GdHOpr1}): 
\begin{itemize}
\item[(1)] 
The Lyapunov function for the infinite system (recall \eqref{e2760}) is well defined only for certain classes of initial laws in the limit as $M_n \to \infty$.
\item[(2)] 
Instead of the monotone decreasing Lyapunov function for the infinite system, we now have a function where the derivative also involves positive terms, arising from the migration in the infinite system, transporting mass in and out of $\G_n$. In addition, there is exchange with the deep seed-banks of colour beyond $M_n$. 
\end{itemize}

\medskip\noindent
Ad (1): This observation restricts the set of initial measures in $\CT^{\mathrm{erg}}_\theta$ that we can compare via coupling, which was already an issue in the proof of the ergodic theorem for the infinite system in \cite{GdHOpr1}.  Ad (2): For fixed $t$ the extra positive terms vanish when $M_n,n\to\infty$, so this puts a restriction on the growth rate of the sequence $(t_n)_{n\in\N}$.

Therefore the following analogue of Lemma~\ref{l.2752} is needed in order to provide the successful coupling used in Task (B). Recall from \eqref{ergbulletlawalt} that $\CT^{\mathrm{erg},\blacksquare}_\theta \subset \CT^{\mathrm{erg}}_\theta$ is the set of initial laws that are invariant and ergodic under translations such that
$$
(\ast) \qquad \forall\,i \in \G\colon\,\lim_{m\to\infty} y_{i,m} = \theta \quad \text{in probability}.
$$

\begin{lemma}{{\bf [Comparison finite-infinite: $\rho=\infty$, Regime (I)]}}
\label{l.2752alt}
\begin{itemize}
\item[{\rm (a)}] 
For every $\mu \in \CT^{\mathrm{erg},\bullet}_\theta$, and every $(t_n)_{n\in\N}$ such that $\lim_{n\to\infty} t_n=\infty$ and $\lim_{n\to\infty} t_n/\beta_n = 0$, all weak limit points of $((\Phi_n\mu) S_n(t_n))_{n\in\N}$ are in $\CT^{\mathrm{erg},\blacksquare}_\theta$. 
\item[{\rm (b)}] 
For every $\mu \in\CT^{\mathrm{erg},\bullet}_\theta$ the same properties as in Lemma~\ref{l.2752} hold.
\end{itemize}
\end{lemma}

\begin{proof}
Note that (a) and (b) do not imply that the conclusions in Lemma~\ref{l.2752} hold for every $\mu \in \CT^\mathrm{erg}_\theta$, which we used to obtain the claim in Task (B) for finite seed-banks. We need to restrict to $\mu \in \CT^{\mathrm{erg},\bullet}_\theta$ and construct a successful coupling by showing that the limit of $\mu(t)$ as $t\to\infty$ lies in $\CT^{\mathrm{erg},\blacksquare}_\theta$. This is achieved by considering two time scales, which we call $(t_n^*)_{n\in\N}$ and $(\bar{t}_n)_{n\in\N}$.

\medskip\noindent
(a) To prove part (a), we proceed as follows.
\begin{itemize}
\item 
Choose a sequence $(t_n^*)_{n\in\N}$ such that 
\begin{itemize}
\item[(i)] all weak limit points of $\{\mu S(t_n^*)\}_{n\in\N}$ lie in $\CT^{\mathrm{erg},\blacksquare}_\theta$,
\item[(ii)] all weak limit points of $\{\tilde\Phi_n[(\Phi_n\mu) S_n(t_n^*)]\}_{n\in\N}$ lie in $\CT^{\mathrm{erg},\blacksquare}_\theta$,
\item[(iii)] $\lim_{n\to\infty} t_n^* = \infty$ and $\lim_{n\to\infty} t_n^*/\beta_n = 0$.
\end{itemize}  
where $(S(t))_{t \geq 0}$ and $(S_n(t))_{t \geq 0}$ are the semi-groups of the infinite, respectively, finite system, and $(\tilde\Phi_n)_{n\in\N}$ are the extension operators defined in \eqref{extop} (see also the general set-up in Appendix~\ref{appB*}, in particular, \eqref{e3610}). In our setting, $\tilde\Phi_n$ is the periodic continuation from $\G_n$ to $\G$ in the geographic coordinate and the constant continuation by $\hat\theta^{M_n,n}$ from $\{0,\ldots,M_n\}$ to $\N_0$ in the seed-bank coordinate.
\item
Choose a sequence $(\bar{t}_n)_{n\in\N}$ satisfying $\lim_{n\to\infty} \bar{t}_n = \infty$ and $\lim_{n\to\infty} \bar{t}_n/\beta_n = 0$ such that the Markov jump process $Z$ on $\G \times \N_0$ with transition kernel $b^{(2)}(\cdot,\cdot)$ satisfies
$$
(\ast\ast) \qquad \lim_{n\to\infty} \P_0\Big(Z \text{ leaves } \G_n \times \{0,\ldots,M_n\} \text{ before time } \bar{t}_n\Big) = 0. 
$$
\end{itemize}
Property (i) holds for $\mu \in \CT^{\mathrm{erg},\bullet}_\theta$ by the ergodic theorem for $(S(t))_{t \geq 0}$ and properties of the equilibrium, stated in Lemma~\ref{lem:pres}. Property (ii) holds by Lemma~\ref{lem:deep}. Property (iii) is imposed to limit the time scale, so that we are in a position to carry out Task (B).

The choice of $t_n^*$ and $\bar{t}_n$ must achieve that, along any subsequence $(n_k)_{k\in\N}$, 
$$
(\ast\ast\ast) \qquad \lim_{k\to\infty} \CL\big[(z^{M_n,n}(t_{n_k}^*+\bar{t}_{n_k}),z(t_{n_k}^*+\bar{t}_{n_k})\big] = \Gamma
\quad \Longrightarrow \quad \Gamma\big(\{z_1,z_2  \in \mathrm{Diag}(E \times E)\}\big) = 1.
$$
On the time interval  $[0,t_n^*]$ the system converges along subsequences to a state in which we can couple the two sequences of seed-banks at every $i\in\G$ because the two components $(z^{M_n,n}_i,z_i)$ agree on the tail sigma-field even for the limit state, as follows from Lemma~\ref{lem:deep}. On the time interval $[t_n^*,\bar{t}_n]$ we want to use the bivariate dynamics for the finite, respectively, the infinite system with its Lyapunov function to obtain the successful coupling. We want to combine the two to arrive at $(\ast\ast\ast)$. Two problems arise:
\begin{itemize}
\item[--]
$(\ast)$ only holds for the weak limit points in (i) and (ii), and not necessarily for finite $n$.
\item[--] 
During the time interval $[0,t_n^*]$ discrepancies between the finite and the infinite system may build up in $\G_n \times \{A,(D_m)_{0 \leq m \leq M_n}\}$.  
\end{itemize} 
The second item will be dealt with in the proof of part (b). 

Concerning the first item, it suffices to show that at the end of the interval $[t_n^*,t_n^*+\bar{t}_n]$ the finite system looks like the equilibrium measure of the infinite system in whose domain of attraction a weak limit point of $((\Phi_n\mu) S_n(t_n^*))_{n\in\N}$ lies. This property is enough because by running the system for a total time $t_n^*+\bar{t}_n$ we can conclude that $(\ast\ast\ast)$ holds. To obtain said property, we have to argue that we can make the replacement at time $t_n^*$ for any weak limit point. In particular, we have to prove that any weak limit point of $((\Phi_n\mu) S_n(t_n^*))_{n\in\N}$ can be successfully coupled via the bivariate finite dynamics associated with $\CL[\bar{\nu}_{\theta^\bullet}]$, where $\theta^\bullet$ is the limit in probability of the conserved quantity of the measure $\mu$. For this we again use coupling. 

For every $k\in\N$, for the bivariate finite dynamics 
\begin{itemize}
\item[]
starting in $(\Phi_k\mu,\Phi_k\bar\mu)$, with $\mu \in \CR^{(2)}_\theta$ and $\bar\mu$ a weak limit point of $((\Phi_n\mu) S_n(t_n^*))_{n\in\N}$,
\end{itemize}
we have a Lyapunov function (recall \eqref{e2760}) that is well-defined and that is non-decreasing. We want to show that, for the bivariate dynamics of two finite systems,
\begin{equation}
\lim_{k\to\infty} \CL[(\Phi_k\mu)S_k(t^*_k),(\Phi_k\bar\mu) S_k(t^*_k)] = 0
\end{equation}  
when $\lim_{k\to\infty} t^*_k = \infty$. We know that for both marginals, with $\mu$ and $\bar\mu$ as marginal initial laws, 
\begin{equation}
\lim_{k\to\infty} \hat\theta^{M_k,k}(t^*_k) = \theta \quad \text{in law}
\end{equation}
when $\lim_{k\to\infty} t^*_k/\beta_k = 0$. Hence, under this condition, correlations between typical sites tend to zero as $k\to\infty$. This implies that a weak limit point of the bivariate dynamics, at time $t^*_k$ in the limit as $k\to\infty$, lies on the diagonal, and the Lyapunov function tends to zero. Indeed, the two configurations can be ordered in the limit, because they have the same value of $\hat\theta$ and are ergodic. In particular, this also holds when we consider times $(t^*_k +\bar{t}_k)_{k\in\N}$. We next put this together to get $(\ast\ast\ast)$.

Since the finite and the infinite bivariate dynamics constructed above allow for a coupling of $\Phi_k\bar\mu$ and $\mu \in \CR^{(2),\bullet}_\theta$ that is successful when run for a time $\bar{t}_k$ with $k\to\infty$, we use this is an ingredient to conclude that $\Phi_k\mu$ and $\nu_\theta$ can be successfully coupled when the bivariate dynamics is run for a time $t^*_k + \bar{t}_k$ with $k\to\infty$. The latter we may view as quadrivariate dynamics: namely, take the infinite bivariate dynamics whose components start from $\mu,\nu_\theta$ and the finite bivariate dynamics whose components start from $\Phi_k\mu,\Phi_k\nu_\theta$, and place them on the same probability space. The claim now follows by picking $t_k= t^*_k + \bar{t}_k$ and letting $k\to\infty$.

\medskip\noindent
(b) To prove part (b), we need to show that also the positive terms in the derivative of the Lyapunov function tend to zero as $n\to\infty$. Here there are two effects:
\begin{itemize}
\item[--]
The truncated migration on $\G_n$ is different from the migration on $\G$.
\item[--]
The truncated seed-bank $\{0,\ldots,M_n\}$ is different from the seed-bank $\N_0$. 
\end{itemize}
We need to show that both differences are negligible in the limit as $n\to\infty$. 

The first was already dealt with for systems without seed-bank, for which we refer to \cite[Section 4, pp.\ 2319--2322]{DGV95}. The same reasoning carries over to systems with seed-bank because in Regime (I) automatically $\lim_{n\to\infty} t_n/\beta_n^{**} = 0$ (recall that $\beta_n^{**} \gg \beta_n$ is the time scale on which two active lineages in the dual coalesce on the active layer $\G_n$). 

For the second we need the restrictions posed in Regime (I) (recall \eqref{scales}--\eqref{Mncond}) in order to be able to show that perturbations arising from the truncation of the seed-bank are negligible in the limit as $n\to\infty$. If $g$ is a multiple of the Fisher-Wright diffusion function, then we can use duality and restrict to the event $\CA_n$ defined in \eqref{Andef}, which is the event that a single Markov process does not visit a seed-bank with colour $>M_n$ until time $t_n$. The claim now follows from the fact that $\lim_{n\to\infty} \mathbb{P}_n(\CA_n) = 1$ when we choose $t_n$ such that  $\lim_{n\to\infty} t_n/\beta_n^* = 0$ (recall that $\beta_n^* \ll \beta_n$ is the time scale on which a single dormant lineage in the dual starting from the deepest seed-bank with colour $M_n$ becomes active on $\G_n$).

To get the claim for general $g\in\CG$, we need to use the coupling argument in \cite[Sections 5--6]{GdHOpr1}, which uses contraction via a Lyapunov function to show that the coupling is successful when $t_n^*$ has the above properties. As is clear from Section~\ref{s.classpres}, the function $g$ plays no role in the estimates as long as $\|g\|<\infty$.
\end{proof}


\subsection{Task (D)}
\label{ss.fourthtaskI}

To show that there is only \emph{one} limit point, we can again follow the \emph{abstract scheme} developed in \cite[Section 4]{DGV95}. We work with the tightness of the process $(\hat\theta^{\,M(n_k),n_k} (t \beta_{M(n_k),n_k}))_{t \geq 0}$ and the fact that the weak limit points must solve a \emph{well-posed} martingale problem (see Lemmas~\ref{l.3734}--\ref{l.3759}). The latter we can obtain from the very same argument as in Section~\ref{s.fssrhofin} for $\rho < \infty$, $M<\infty$. There we showed that $(\hat\theta^{\,M_n,n}(t \beta_{M_n,n}))_{t \geq 0}$ is a martingale for which we can determine the increasing process. Namely, in the formula for the increasing process in \eqref{e2694} we can apply a \emph{law of large numbers} to the empirical measure, which allows us to identify the limiting martingale problem as before. Here, the difficulty is that the active component of the empirical measure is negligible in the limit as $n\to\infty$. 

To circumvent this obstacle, we look at $\CE_{L_n}^{M_n,n}$, the truncated empirical measure that includes the active component and the first $L_n$ dormant components, with $L_n\in\N_0$ chosen appropriately  (see Section~\ref{sss.fastgrow}). For this quantity it suffices to obtain the increasing process and to describe the empirical measure in the weak topology. Indeed, for every $L\in\N_0$ the truncated empirical measure converges to the truncated equilibrium measure. Since the increasing process only depends on the active component, as before we get a unique limit point and so we have \emph{full convergence} as $n \to \infty$ (not just along subsequences).

To identify the process $(\Theta(t))_{t \geq 0}$ as the $\CF g$-diffusion, we can follow the same route as in Step 3 in Section~\ref{ss.fssmod2rhofin}, without any changes. Therefore we have verified (A10).


\section{Proofs: $\rho=\infty$ and fast growing seed-bank}
\label{s.fssrhoinffast}

In this section we prove Theorem~\ref{T.finsys2inf(2)}. The line of argument is different from that in Sections~\ref{s.fssrhofin}--\ref{s.fssrhoinfslow}, and requires new ideas because the conditional spatial ergodicity used in Section~\ref{s.fssrhoinfslow} for slow growing seed-bank fails at times of order $\beta^{**}_n$. Hence the argument in Appendix \ref{appB*} breaks down. The three time scales in Theorem~\ref{T.finsys2inf(2)} are treated in Section~\ref{ss.convequi}--\ref{ss.fix}, respectively. The initial law satisfies $\mu(0) \in \CT^{\mathrm{erg},\bullet}_\theta = \CR^{(2),\bullet}_\theta$ for some $\theta \in [0,1]$.


\subsection{Convergence to equilibrium on short time scales}
\label{ss.convequi} 

In this section we prove Theorem~\ref{T.finsys2inf(2)}(1). 

\medskip\noindent
{\bf 1.} Consider two independent Markov processes on $\S_n$ with transition kernel $b^{(2),n}(\cdot,\cdot)$, both starting in the active state at the origin. Let $\CE(t)$ and $\CE'(t)$ be the events that the Markov processes are active at time $t$, and let $T(t) = \int_0^t \d s\,1_{\CE(s)}$ and $T'(t) = \int_0^t \d s\,1_{\CE'(s)}$ be their total activity times up to time $t$. Then their \emph{total joint activity time} at time $T$ (i.e., the total time the two Markov processes are jointly active at the same site up to time $T$) is equal in law to  
\begin{equation}
\label{tjatfin}
I^{(2),n}(T) = \int_0^T \d t\,1_{\CE(t)}\,1_{\CE'(t)}\,1_{\{RW_1^{n,\uparrow}(T(t)) = RW_2^{n,\uparrow}(T'(t))\}}, 
\end{equation}
where $(RW_1^{n,\uparrow}(t))_{t \geq 0}$ and $(RW_2^{n,\uparrow}(t))_{t \geq 0}$ are two independent Markov processes on $\G_n$ with transition kernel $b^n(\cdot,\cdot)$. This is true because the transition rates between active and dormant do not depend on the location in geographic space.

We want to compare this quantity with the same quantity for the infinite system, given by
\begin{equation}
\label{tjatinf}
I^{(2)}(T) = \int_0^T \d t\,1_{\CE(t)}\,1_{\CE'(t)}\,1_{\{RW_1^\uparrow(T(t)) = RW_2^\uparrow(T'(t))\}}, 
\end{equation}
where $(RW_1^\uparrow(t))_{t \geq 0}$ and $(RW_2^\uparrow(t))_{t \geq 0}$ are two independent random walks on $\G$ with transition kernel $a(\cdot,\cdot)$. We can couple the transitions between active and dormant, provided no transition from $A \to D_m \to A$ occurs in the infinite system until time $T$ for any $m>M_n$. However, as long as $T = o(\beta^*_n)$, the probability of this event tends to zero as $n\to\infty$ (recall \eqref{Andef}). Therefore, under this coupling we have  
\begin{equation}
\begin{aligned}
&0 \leq I^{(2),n}(T) - I^{(2)}(T)\\ 
&= o(1) + \int_0^T \d t\,1_{\CE(t)}\,1_{\CE'(t)}
\left[1_{\{RW_1^{n,\uparrow}(T(t)) = RW_2^{n,\uparrow}(T'(t))\}} 
- 1_{\{RW_1^\uparrow(T(t)) = RW_2^\uparrow(T'(t))\}}\right]. 
\end{aligned}
\end{equation}
Put 
\begin{equation}
\Delta^n(T) = \hat{\E}_{(0,A),(0,A)}[I^{(2),n}(T) - I^{(2)}(T)],
\end{equation} 
where the expectation is with respect to the law of the coupling. Then 
\begin{equation}
0 \leq \Delta^n(T) = o(1) + \int_0^T \d t\,\,\E\,\left[1_{\CE(t)}\,1_{\CE'(t)}
\left[\sum_{i\in\G_n} a^n_{T(t)}(0,i)\,a^n_{T'(t)}(0,i) - \sum_{i\in\G} a_{T(t)}(0,i)\,a_{T'(t)}(0,i)\right]\right],
\end{equation}
where $\E$ denotes expectation with respect to the law of the process $(1_{\CE(t)},1_{\CE'(t)})_{t \geq 0}$ for the infinite system starting from $(1,1)$. Since for Model 2 with $\rho=\infty$ we have assumed that the transition kernel is symmetric (recall \eqref{sym}), this gives
\begin{equation}
\label{bds}
0 \leq \Delta^n(T) = o(1) + \int_0^T \d t\,\,\E\,\Big[1_{\CE(t)}\,1_{\CE'(t)}\,
\big[a^n_{T(t)+T'(t)}(0,0) - a_{T(t)+T'(t)}(0,0)\big]\Big].
\end{equation}

\medskip\noindent
{\bf 2.}
Our goal is to show that the right-hand side tends to zero as $n\to\infty$ as long as $T=o(\beta^{**}_n)$. To do so we will use that, as shown in \cite[Section 6.2]{GdHOpr1},
\begin{equation}
\label{Tasymp}
\begin{aligned}
&\lim_{t\to\infty} \frac{1}{t^\gamma}\,T(t) = V, 
\qquad 
\lim_{t\to\infty} \frac{1}{t^\gamma}\,T'(t) =  V' 
\quad\text{ in $\mathbb{P}$-probability},\\
&\lim_{t\to\infty} t^{1-\gamma}\,\mathbb{P}\big(\CE(t)\big) = \mathbb{E}[V],
\qquad 
\lim_{t\to\infty} t^{1-\gamma}\,\mathbb{P}\big(\CE'(t)\big) = \mathbb{E}[V'],
\end{aligned}
\end{equation}
where $V=W^{-\gamma}/\chi$, with $W$ a stable law random variable on $(0,\infty)$ with exponent $\gamma$, satisfying $\E[V]<\infty$, and $V'$ is an independent copy of $V$. We split the integral into two parts, 
\begin{equation}
\label{intsplit}
\begin{aligned}
\Delta^n_* &= \int_0^{o(\psi_n^{1/\gamma})} \d t\,\,\E\,\Big[1_{\CE(t)}\,1_{\CE'(t)}\,
\big[a^n_{T(t)+T'(t)}(0,0) - a_{T(t)+T'(t)}(0,0)\big]\Big],\\
\Delta^n_*(T) &= \int_{o(\psi_n^{1/\gamma})}^T \d t\,\,\E\,\Big[1_{\CE(t)}\,1_{\CE'(t)}\,
\big[a^n_{T(t)+T'(t)}(0,0) - a_{T(t)+T'(t)}(0,0)\big]\Big],
\end{aligned}
\end{equation}
and show that both parts vanish. 

\medskip\noindent
{\bf 3.} 
For the first integral in \eqref{intsplit} we use \eqref{afininfcompalt} to estimate
\begin{equation}
\label{bds1}
\Delta^n_* = \int_0^{o(\psi_n^{1/\gamma})} \d t\,\,\E\,\Big[1_{\CE(t)}\,1_{\CE'(t)}\,o\big(a_{T(t)+T'(t)}(0,0)\big)\Big],
\end{equation}
where we use that $T(t)+T'(t) \asymp t^\gamma = o(\psi_n)$ in $\P$-probability. Combining \eqref{Tasymp} and \eqref{bds1}, we get
\begin{equation}
\Delta^n_* = o\left(\int_0^{o(\psi_n^{1/\gamma})} \d t\,(1 \wedge t^{-2(1-\gamma)})\,a_{t^\gamma}(0,0)\right),
\end{equation}
where we use that $t \mapsto a_t(0,0)$ is regularly varying at infinity (recall \eqref{ass2}) to get rid of the factor $V+V'$ and the fluctuations of $T(t)+T'(t)$ on scale $t^\gamma$. Inserting the change of variable $s=t^\gamma$, we obtain
\begin{equation}
\Delta^n_* = o\left(\int_0^{o(\psi_n)} \d s\,(1 \wedge s^{-(1-\gamma)/\gamma})\,a_s(0,0)\right).
\end{equation}
But
\begin{equation}
\int_0^\infty \d s\,(1 \wedge s^{-(1-\gamma)/\gamma})\,a_s(0,0) < \infty 
\end{equation}
in the coexistence regime (recall \eqref{crinf}), and so we have shown that $\lim_{n\to\infty} \Delta^n_* = 0$. 

\medskip\noindent
{\bf 4.}
For the second integral in \eqref{intsplit} we use the definition of the mixing time in \eqref{pnmix} in combination with the comparison property in \eqref{afininfcomp} to estimate $C^{-1} |\G_n|^{-1} \leq a^n_s(0,0) \leq C|\G_n|^{-1}$ for $o(\psi_n) \leq s \leq T$ and some $C<\infty$. This gives
\begin{equation}
\label{intrepr}
\Delta^n_*(T) \asymp C|\G_n|^{-1} \int_{o(\psi_n)}^{T^\gamma} \d s\,(1 \wedge s^{-(1-\gamma)/\gamma}) 
\asymp |\G_n|^{-1} \left\{\begin{array}{ll}
T^{2\gamma-1}, &\gamma \in (\tfrac12,1],\\[0.2cm]
\log T, &\gamma = \tfrac12.
\end{array}
\right.
\end{equation}
Since the latter tends to zero as long as $T = o(|\G_n|^{1/(2\gamma-1)})$, respectively, $T = \e^{o(|\G_n|)}$ (recall \eqref{scales} and the remark below \eqref{Mncond}), we have shown that $\lim_{n\to\infty} \sup_{T = o(\beta^{**}_n)} \Delta^n_*(T) = 0$. Note that we need the assumption $\psi_n = o((\beta^{**}_n)^\gamma)$ because this guarantees that the divergence of the integral in \eqref{intrepr} indeed occurs on time scale $\beta^{**}_n$ and not later.  

\medskip\noindent
{\bf 5.} 
Combining the estimates Steps 2--4, we arrive at
\begin{equation}
\lim_{n\to\infty} \sup_{T = o(\beta^{**}_n)} \big[I^{(2),n}(T) - I^{(2)}(T)\big] = 0 \quad \text{ in probability}.
\end{equation}
Hence, up to time $o(\beta^{**}_n)$ the two Markov processes starting from $(0,A),(0,A)$ see no difference in their total joint activity time between the finite system and the infinite system. It is easy to extend this fact to arbitrary starting points $(0,R_1),(0,R_2)$ with $R_1,R_2 \in \{A,(D_m)_{0 \leq m \leq L_n}\}$. Indeed, we wait until the two Markov processes are jointly active for the first time, which occurs at a finite random time (whose law depends on $L_n$). At that time they start from $(I,A),(J,A)$ with $I,J \in \G_n$ some random locations. This amounts to replacing $a^n_s(0,0)$ and $a_s(0,0)$ by $a^n_s(I,J)$ and $a_s(I,J)$ in \eqref{bds}, \eqref{intsplit} and \eqref{bds1}. However, this does not affect the estimates in Steps 2 and 3 because of the approximation in \eqref{afininfcompalt} and the bound in \eqref{centralbds}, respectively.

\medskip\noindent
{\bf 6.}
The estimates in Steps 2--4 can be trivially extended to $k \in \N$ Markov processes, because for each of the $\binom{k}{2}$ pairs the discrepancy between the total joint activity times tends to zero in probability. In case $g = dg_{\mathrm{FW}}$, $d \in (0,\infty)$, the dual is available and is obtained by letting active Markov processes at the same site coalesce at rate $d$. Consequently, all the mixed moments at time $\bar{\beta}_n$ have the same limit in the finite system as in the infinite system as long as $\bar{\beta}_n \to \infty$ and $\bar{\beta}_n/\beta^{**}_n \to 0$. Hence there is local convergence to $\nu_\theta$, the equilibrium for the infinite system.    

\medskip\noindent
{\bf 7.}
For $g \in \CG$ we do not have a dual, and comparison duality (exploited in Section~\ref{sss.partclus}) does not work either because $\nu_\theta$ depends on $g$. However, as in Section~\ref{s.fssrhoinfslow}, we can follow the abstract scheme from Appendix~\ref{appB*}, as long as we consider times scales $\bar{\beta}_n = o(\beta^{**}_n)$. For this we have to work with the macroscopic variable, which on times scales $\bar{\beta}_n = o(\beta^{**}_n)$ is the constant process equal to $\theta$. We again have to check Assumptions (A1)--(A10). The key tools are once again Lemmas~\ref{lem:erg}--\ref{lem:deep}, which hold for times scales $\bar{\beta}_n = o(\beta^{**}_n)$, as shown in Section~\ref{ss.deep}. Steps 1-5 are needed to couple the finite and the infinite system and show decorrelation. 

\subsection{Partial clustering and switching on intermediate time scales}
\label{ss.partclussw}


\subsubsection{Partial clustering}
\label{sss.partclus} 

In this section we prove the partial clustering part of Theorem~\ref{T.finsys2inf(2)}(2). The proof works for $g\in\CG$ and does not need the dual because it is based on a second moment computation only. Recall \eqref{zudef} and write $u=(u^1,u^2)$ with $u^1 \in \G_n$ and $u^2 \in \{A,(D_m)_{0 \leq m \leq M_n}\}$.  

\begin{proof}
The proof is an adaptation of the proof of \cite[Lemma 5.5]{GdHOpr1}. Pick $\beta^{**}_n \ll \bar{\beta}_n \ll \beta^{*}_n$ ($= \beta_n$)
and $L \in \N_0$.

\medskip\noindent
{\bf 1.}
To prove asymptotic \emph{partial clustering} to depth $L_n$ we must show that, for any $\mu \in \CT$,
\begin{equation}
\label{partcl}
\lim_{n\to\infty} \sup_{u_1,u_2 \in \S_n^{L_n}} \E_{\Phi_n\mu}\left[z_{u_1}(\bar{\beta}_n)\big(1-z_{u_2}(\bar{\beta}_n)\big)\right] = 0.
\end{equation}
Indeed, this implies that, uniformly in $u_1,u_2 \in \S_n^{L_n}$,
\begin{equation}
\begin{aligned}
u_1=u_2\colon &\lim_{n\to\infty} \P_{\Phi_n\mu}\Big(z_{u_1}(\bar{\beta}_n) \in [0,\epsilon) \cup (1-\epsilon,1]\Big)=1 
\quad \forall\,\epsilon>0,\\ 
u_1 \neq u_2\colon &\lim_{n\to\infty} \P_{\Phi_n\mu}\Big(z_{u_2}(\bar{\beta}_n) \in (1-\epsilon,1] 
~\Big|~ z_{u_1}(\bar{\beta}_n) \in (1-\epsilon,1]\Big) = 1
\quad \forall\,\epsilon>0.
\end{aligned}
\end{equation}
Like in \cite{CG94}, we will give the proof by \emph{comparison duality}. Fix $\epsilon>0$. Since $g\in\CG$ we can choose a $c=c(\epsilon)>0$ such that $g(x) \geq \tilde{g}(x) = c(x-\epsilon)(1-\epsilon-x)$, $x\in [0,1]$. Note that $\tilde{g}(x)<0$ for $x \in [0,\epsilon) \cup (1-\epsilon,1]$, so we cannot replace $g$ by $\tilde{g}$ in the evolution equations. Instead we use $\tilde{g}$ as an auxiliary function. 

\medskip\noindent
{\bf 2.}
Let $(B(t))_{t \geq 0}$ be the Markov chain with state space $\{1,2\} \times \S_n \times \S_n$ and $B(t)=(B_0(t),B_1(t),B_2(t))$, evolving according to
\begin{equation}
\label{12MCrates}
\begin{aligned}
(1,u_1,u_1) &\to (1,u_3,u_3), \quad  \text{ at rate } b^{(2),n}(u_1,u_3),\\
(2,u_1,u_2) &\to
\begin{cases}
(2,u_3,u_2), &\text{ at rate } b^{(2),}(u_1,u_3),\\
(2,u_1,u_4), &\text{ at rate } b^{(2),n}(u_2,u_4),\\
(1,u_1,u_1), &\text{ at rate } c\,1_{\{u_1^1=u_2^1\}}\,1_{\{u_1^2=u_2^2=A\}}.
\end{cases}
\end{aligned}
\end{equation}
This describes two random walks, evolving independently according to the transition kernel $b^{(2),n}(\cdot,\cdot)$, that coalesce at rate $c>0$ when they are at the same site and both active. We put $B_0(t)=1$ when the two random walks have already coalesced by time $t$, and $B_0(t)=2$ otherwise. Let $\P_{(2,u_1,u_2)}$ denote the law of the Markov chain $(B(t))_{t \geq 0}$ that starts in $(2,u_1,u_2)$. Note that 
\begin{equation}
\begin{aligned}
\P_{(2,u_1,u_2)}\left(B_1(t)=u_3\right) &= b^{(2),n}_t(u_1,u_3),\\
\P_{(2,u_1,u_2)}\left(B_2(t)=u_4\right) &= b^{(2),n}_t(u_2,u_4).
\end{aligned}
\end{equation}
Below we will show that, for $L_n \ll \bar{\beta}_n^{1/\beta}$,
\begin{equation}
\label{P2lim}
\lim_{n \to \infty} \sup_{u_1,u_2 \in \S_n^{L_n}} \P_{(2,u_1,u_2)}\big(B_0(\bar{\beta}_n)=2\big) = 0.
\end{equation}

\medskip\noindent
{\bf 3.}
The evolution equations read
\begin{equation}
\label{m3}
\d z_{u_1}(t) = \sum_{u_3\in\S_n}b^{(2),n}(u_1,u_3)\,[z_{u_3}(t)-z_{u_1}(t)]\,\d t + \sqrt{g(z_{u_1}(t))}\,1_{\{R_i=A\}}\,\d w_i(t),
\qquad u_1 = (i,R_i) \in \S_n.
\end{equation}
Put
\begin{equation}
\label{zepsdef}
z^-_u = z_u-\epsilon, \qquad z^+_u = z_u+\epsilon.
\end{equation}
For $t\geq 0$, define $F_t\colon\,\{1,2\}\times \S_n\times \S_n \to\R$ by  
\begin{equation}
\label{Fdefs}
F_t(1,u_1,u_1) = \E_{\Phi_n\mu}\big[z^-_{u_1}(t)\big], \qquad 
F_t(2,u_1,u_2) =\E_{\Phi_n\mu}\big[z^-_{u_1}(t)z^+_{u_2}(t)\big].
\end{equation}
Using It\^o-calculus, we obtain from \eqref{m3} that
\begin{equation}
\label{Ito1}
\frac{\d}{\d t} F_1(1,u_1,u_1) = \sum_{u_3\in\S_n} b^{(2),n}(u_1,u_3)\,
\big[F_t(1,u_3,u_3)-F_1(1,u_1,
u_1)\big]
\end{equation}
and
\begin{equation}
\label{Ito2}
\begin{aligned}
&\frac{\d}{\d t} F_t(2,u_1,u_2) = \E_{\Phi_n\mu}\left[g(z_{u_1}(t))\right]\,1_{\{u_1^1=u_2^1\}}\,1_{\{u_1^2=u_2^2 = A\}}\\
&+\sum_{u_3\in\S_n} b^{(2),n}(u_1,u_3)\,\big[F_t(2,u_3,u_2)-F_t(2,u_1,u_2)\big] 
+ \sum_{u_4\in\S_n} b^{(2),n}(u_2,u_4)\,\big[F_t(2,u_1,u_4)-F_t(2,u_1,u_2)\big].
\end{aligned}
\end{equation}
Note that
\begin{equation}
\label{pro1}
\E_{\Phi_n\mu}\left[\tilde{g}(z_{u_1}(t))\right] = \E_{\Phi_n\mu}\left[cz^-_{u_1}(t)(1-z^+_{u_1}(t))\right]
= c F_t(2,u_1,u_1) .
\end{equation}

\medskip\noindent
{\bf 4.}
For $t \geq 0$, define $G_t\colon\,\{1,2\}\times \S_n \times \S_n\to\R$ by
\begin{equation}
\label{Gtdef}
G_t(1,u_1,u_1) =0, \qquad
G_t(2,u_1,u_2) =\E_{\Phi_n\mu}\left[\big(g(z_{u_1}(t))-\tilde{g}(z_{u_1}(t))\big)\right]
\,1_{\{u_1^1=u_2^1\}}\,1_{\{u_1^2=u_2^2A\}}.
\end{equation}
Let $\mathfrak{B}$ denote the generator of $(B(t))_{t \geq 0}$, and let $(\mathfrak{S}_t)_{t \geq 0}$ be the associated semigroup. Then, with the help of \eqref{pro1}--\eqref{Gtdef}, \eqref{Ito1}--\eqref{Ito2} can be written as
\begin{equation}
\frac{\d F_t}{\d t}=\mathfrak{B}F_t+G_t
\end{equation} 
and hence \cite[Theorem I.2.15]{Lig85}
\begin{equation}
F_t=\mathfrak{S}_tF_0+\int_0^t \d s\,\mathfrak{S}_{t-s}G_s.
\end{equation}
Since $G_t \geq 0$ for all $t\geq 0$, we obtain
\begin{equation}
\begin{aligned}
&F_t(2,u_1,u_2) \geq (\mathfrak{S}_t F_0)(2,u_1,u_2) = \E_{(2,u_1,u_2)}\left[F_0(B(t))\right]\\
&= \E_{(2,u_1,u_2)}\left[F_0(B(t))\,1_{\{B_0(t)=1\}}\right] + \E_{(2,u_1,u_2)}\left[F_0(B(t))\,1_{\{B_0(t)=2\}}\right].
\end{aligned}
\end{equation}
But
\begin{equation}
\begin{aligned}
\E_{(2,u_1,u_2)}\left[F_0(B(t))\,1_{\{B_0(t)=1\}}\right]
&= \E_{(2,u_1,u_2)}\left[\E_{\Phi_n\mu}[z^-_{B_1(t)}]\,1_{\{B_0(t)=1\}}\right]\\
&= \E_{(2,u_1,u_2)}\left[\E_{\Phi_n\mu}[z^-_{B_1(t)}]\right] - \E_{(2,u_1,u_2)}\left[\E_{\Phi_n\mu}[z^-_{B_1(t)}]\,1_{\{B_0(t)=2\}}\right],
\end{aligned}
\end{equation}
while
\begin{equation}
\E_{(2,u_1,u_2)}\left[F_0(B(t))\,1_{\{B_0(t)=2\}}\right]
= \E_{(2,u_1,u_2)}\left[\E_{\Phi_n\mu}[z^-_{B_1(t)}\,z^+_{B_1(t)}]\,1_{\{B_0(t)=2\}}\right].
\end{equation}
Since $F_t(1,u_1,u_1) = \E_{(2,u_1,u_2)}[\,\E_{\Phi_n\mu}[z^-_{B_1(t)}]]$, we get
\begin{equation}
\label{estfinal}
\begin{aligned}
&\E_{\Phi_n\mu}\left[z^-_{u_1}(t)\,(1-z^+_{u_2}(t))\right] = F_t(1,u_1,u_1)-F_t(2,u_1,u_2)\\
&\leq \E_{(2,u_1,u_2)}\left[\E_{\Phi_n\mu}[z^-_{B_1(t)}(1-z^+_{B_1(t)})]\,1_{\{B_0(t)=2\}}\right]
\leq \P_{(2,u_1,u_2)}(B_0(t)=2).
\end{aligned}
\end{equation}
Combining \eqref{P2lim} and \eqref{estfinal}, we obtain
\begin{equation}
\liminf_{n \to \infty} \sup_{u_1,u_2 \in \S_n^{L_n}} 
\E_{\Phi_n\mu}\left[(z_{u_1}(\bar{\beta}_n)-\epsilon)(1-\epsilon - z_{u_2}(\bar{\beta}_n))\right] = 0.
\end{equation}
Noting that $z_{u_1}(1-z_{u_2}) \leq (z_{u_1}-\epsilon)(1-\epsilon - z_{u_2}) +\epsilon(2-\epsilon)$ and letting $\epsilon\downarrow 0$, we get \eqref{partcl}.

\medskip\noindent
{\bf 5.} It remains to prove \eqref{P2lim}. Let $\tau_1,\tau_2$ denote the first wake-up times of two independent Markov processes with transition kernel $b^{(2),n}(\cdot,\cdot)$ starting in $u_1,u_2$. If $u_1^2 = u_2^2= A$, then $\tau_1=\tau_2=0$. Otherwise, $\tau_1,\tau_2$ are positive. Since the wake-up time from the deepest seed-bank that is being monitored is of order $1/e_{L_n} \asymp L_n^\beta$, we have that $\tau_1,\tau_2$ are at most of order $L_n^\beta$ uniformly in $u_1,u_2 \in \S^{L_n}_n$. Return to \eqref{tjatfin}. We found in Step 1 in Section~\ref{ss.convequi} that the average total joint activity time of the two Markov processes up to time $T$ in the finite system satisfies
\begin{equation}
\E_{u_1,u_2}[I^{(2),n}(T)] \asymp  |\G_n|^{-1} \int_{\tau_1 \vee \tau_2}^T \d s\,
[1 \wedge (s-\tau_1)^{-(1-\gamma)}]\, [1 \wedge (s-\tau_2)^{-(1-\gamma)}].
\end{equation} 
when $T \gg \psi_n^{1/\gamma}$. Assume without loss of generality that $\tau_1<\tau_2$. Then 
\begin{equation}
\E_{u_1,u_2}[I^{(2),n}(T)] =  |\G_n|^{-1} \int_0^{T-\tau_2}  \d s\,[1 \wedge (s+(\tau_2-\tau_1))^{-(1-\gamma)}]\, [1 \wedge s^{-(1-\gamma)}], 
\end{equation}
Pick $T = \bar{\beta}_n$, and note that $\bar{\beta_n} \gg \psi_n^{1/\gamma}$ by our assumption that $\psi_n = o((\beta^{**}_n)^\gamma)$ because $\bar{\beta}_n \gg \beta^{**}_n$. Since $\bar{\beta}_n \gg L_n^\beta$, we have (recall \eqref{scales})
\begin{equation}
\E_{u_1,u_2}[I^{(2),n}(\bar{\beta}_n)] \asymp |\G_n|^{-1} \int_0^{\bar{\beta}_n}  \d s\,[1 \wedge s^{-2(1-\gamma)}]
\asymp |\G_n|^{-1} (\bar{\beta}_n)^{2\gamma-1} = (\bar{\beta}_n/\beta^{**}_n)^{2\gamma-1}
\end{equation} 
for $\gamma \in (\tfrac12,1]$, which diverges as $n\to\infty$. For $\gamma=\tfrac12$ we have $\E[I^{(2),n}(\bar{\beta}_n) \asymp |\G_n|^{-1} \log \bar{\beta}_n = \log (\bar{\beta}_n/\beta^{**}_n)$, which again diverges as $n\to\infty$. Thus we have proved that the average total joint activity time of the two Markov processes on $\S_n$ diverges as $n\to\infty$. We will complete the proof by showing that the same is true in $\P$-probability. This settles \eqref{P2lim} because the rate of coalescence $c$ in \eqref{12MCrates} is strictly positive. 

\medskip\noindent
{\bf 6.}
The process $(\CE(s))_{s \geq 0}$ marks the times at which the Markov process on $\{A,(D_m)_{m\in\N_0}\}$ is in $A$. Since $A$ is a renewal state for this Markov process, we can think of $(\CE(s))_{s \geq 0}$ as a \emph{Markov renewal process}. Once more return to \eqref{tjatfin}. Let $u_1=u_2=(0,A)$. Pick any $\phi_n$ such that $\psi_n \ll \phi_n \ll \bar{\beta}_n^\gamma$, which is possible because $\psi_n = o((\beta^{**}_n)^\gamma)$ and $\bar{\beta}_n \gg \beta^{**}_n$. Put $t_{k,n} = (k\phi_n)^{1/\gamma}$, $k \in \N_0$, and let $K_n$ be such that $t_{K_n,n} \asymp \bar{\beta}_n$, i.e., $K_n \asymp \bar{\beta}_n^\gamma/\phi_n \gg 1$. By the \emph{sample-path} version of \eqref{Tasymp}, we have 
\begin{equation}
\begin{aligned}
&\lim_{n\to\infty} \left\{\CL\left[\Big(T(t_{k,n})-T(t_{k-1,n}),\CE(t_{k,n})\Big)_{1 \leq k \leq K_n}\right]
- \otimes_{1 \leq k \leq K_n} \left(\delta_{\tau_{k,n}^\gamma-\tau_{k,n}^\gamma}, \left[\big(1-\bar{\tau}_{k,n}^{-(1-\gamma)}\big)\,\delta_0 
+ \bar{\tau}_{k,n}^{-(1-\gamma)}\,\delta_1\right]\right)\right\}\\
&\qquad\qquad  = \delta_{0^{\N_0}}
\end{aligned}
\end{equation}
for some random variables $\tau_{k,n},\bar{\tau}_{k,n}$ satisfying $\tau_{k,n} \asymp t_{k,n} \asymp \bar{\tau}_{k,n}$, $1 \leq k \leq K_n$, in $\P_n$-probability. Here we use that
\begin{equation}
\label{condtkn}
t_{k,n} - t_{k-1,n} \gg (t_{k-1,n})^{1-\gamma}
\end{equation} 
to ensure that during each time interval $[t_{k-1,n},t_{k,n}]$ many renewals to $A$ take place. (The condition in \eqref{condtkn} is satisfied uniformly in $k$ because $t_{k,n}/t_{k-1,n} \asymp 1$ uniformly in $k$ and $\phi_n \gg 1$.) Moreover, by \eqref{pnmix},
\begin{equation}
\lim_{n\to\infty} \left\{\CL\left[\big(RW_1^{n,\uparrow}(T(t_{k,n}))\big)_{1 \leq k \leq K_n}\right] - U(\G_n)^{\otimes K_n} \right\} 
= \delta_{0^{\G \times \N_0}}
\end{equation}
with $U(\G_n)$ the uniform distribution on $\G_n$. Here we use that $(t_{k,n})^\gamma-(t_{k-1,n})^\gamma = \phi_n \gg \psi_n$, $1 \leq k \leq K_n$ to ensure that during each time interval $[t_{k-1,n},t_{k,n}]$ the Markov processes mix well over $\G_n$. It follows via the law of large numbers that  
\begin{equation}
I^{(2),n}(\bar{\beta}_n) \asymp |\G_n|^{-1} \sum_{1 \leq k \leq K_n} \big(t_{k,n}-t_{k-1,n})\,(t_{k,n})^{-2(1-\gamma)}
\quad \text{in $\P_n$-probability}.
\end{equation}
The sum scales like the integral $\int_0^{\bar{\beta}_n} \d t\,[1 \wedge t^{-2(1-\gamma)}] \asymp \bar{\beta}_n^{2\gamma-1}$, and so $I^{(2),n}(\bar{\beta}_n) \asymp \E_{u_1,u_2}[I^{(2),n}(\bar{\beta}_n)]$ in $\P_n$-probability. The same argument works for $u_1,u_2 \in \S_n^{L_n}$ because Markov processes starting at depth $\leq L_n$ become active in time $o(\bar{\beta}_n)$.        
\end{proof}


\subsubsection{Switching}
\label{sss.switch}

In this section we prove the switching part of Theorem~\ref{T.finsys2inf(2)}(2). 

Pick any time scale $\bar\beta_n$ and any seed-bank depths $1 \leq L^-_n < L^+_n \leq M_n$ satisfying 
\begin{equation}
\beta^{**}_n \ll \bar{\beta}_n \ll \beta^*_n, \qquad (L^-_n)^\beta \ll \bar{\beta}_n \ll (L^+_n)^\beta.
\end{equation}
We know from Section~\ref{sss.partclus} that partial clustering has occurred to depth $L^-_n$ at time $\bar{\beta}_n$. We show that none of the seed-banks with colour $L^+_n < m \leq M_n$ has changed up to time $\bar{\beta}_n$.

Consider a Markov process with transition kernel $b^{(2),n}(\cdot,\cdot)$. Because the transitions into and out of the seed-bank do not depend on the location of the Markov process, it suffices to look at the Markov process projected onto $E=\{A,(D_m)_{0 \leq m \leq M_n}\}$. Transitions $A \to D_m$ occur at rate $K_me_m$, transitions $D_m \to A$ occur at rate $e_m$. We know that the total activity time up to time $T$ is $\asymp T^\gamma$. (This is the asymptotics for the infinite seed-bank, but as long as $T\ll \beta^*_n$ the probability that in the infinite seed-bank a jump to a colour $>M_n$ occurs up to time $T$ is $o(1)$.) The rate to jump from $A$ to $D_m$ for some $L^+_n < m \leq M_n$ is (recall \eqref{ass3})
\begin{equation}
\sum_{L^+_n < m \leq M_n} K_me_m = \sum_{m>L^+_n} K_me_m - \sum_{m>M_n} K_me_m
\sim \frac{AB}{\gamma\beta}\big[(L^+_n)^{-\gamma\beta} - (M_n)^{-\gamma\beta}\big] \asymp (L^+_n)^{-\gamma\beta}.
\end{equation}
The total rate accumulated up to time $T$ is $\asymp T^\gamma (L^+_n)^{-\gamma\beta}$. For $T= \bar{\beta}_n$ this equals $(\bar{\beta}_n/(L^+_n)^\beta)^\gamma$, which is $o(1)$, and so the claim follows.

Pick $\mu(0) = \mu \in \CR^{(2),\bullet}_\theta$. We know from the above observations that 
\begin{equation}
\lim_{n\to\infty} \CL\Big[\Phi_n\mu(\bar{\beta}_n) - (1-\theta)\,\Phi_n\mu^{0,L^-_n,L^+_n}(\bar{\beta}_n) 
+ \theta\,\Phi_n\mu_1{1,L^-_n,L^+_n}(\bar{\beta}_n)\Big] 
= \delta_{(0,0^{\N_0})^\G},
\end{equation}
where $\Phi_n\mu^{0,L^-_n,L^+_n}(\bar{\beta}_n)$ and $\Phi_n\mu^{1,L^-_n,L^+_n}(\bar{\beta}_n)$ restricted to $\G_n \times \{(D_m)_{L^+_n < m \leq M_n}\}$ coincide with $\Phi_n\mu(\bar{\beta}_n)$, but restricted to $\G_n \times \{A,(D_m)_{0 \leq m \leq L^-_n}\}$ are `all $0$', respectively, `all $1$', according to an independent draw $\bar\Upsilon \in \{0,1\}$ with mean $\theta$. Clearly, both are elements of $\CR^{(2),\bullet}_\theta$, because the seed-banks beyond depth $L^+_n$ have not changed and $L^+_n \ll M_n$. Therefore, at any time $\tilde{\beta}_n$ satisfying $\bar{\beta}_n \ll \tilde{\beta}_n \ll \beta^*_n$, partial clustering has occurred in both, to some depth $o(\tilde{\beta}_n^{1/\beta})$ that is larger than $L^-_n$ (and possibly larger than $L^+_n$ as well), according to an independent draw $\tilde\Upsilon \in \{0,1\}$ with mean $\theta$. 

\subsection{Complete clustering on large time scales}
\label{ss.fix}

In this section we prove Theorem~\ref{T.finsys2inf(2)}(3). The proof again works for $g\in\CG$ because it uses the comparison duality exploited in Section~\ref{sss.partclus}. The core is the following lemma, which implies that the macroscopic variable $\hat{\theta}^{M_n,n}(\bar{\beta}_n)$ observed at time $\bar{\beta}_n$ converges to $0$ or $1$ in probability as $n\to\infty$ when $\bar{\beta}_n \gg \beta^*_n$. Since $t \mapsto \hat{\theta}^{M_n,n}(t)$ is a martingale, the convergence to 0 and 1 occurs at all times $t \geq \bar{\beta}_n$ simultaneously. From \eqref{mo3} we see that this implies complete fixation (= complete clustering). 

\begin{lemma}
For any $\mu \in \CT$, $g \in \CG$ and $\bar{\beta}_n \gg \beta^*_n$,
\begin{equation}
\lim_{n\to\infty} \E_{\Phi_n\mu}\Big[\hat{\theta}^{M_n,n}(\bar{\beta}_n)\big(1-\hat{\theta}^{M_n,n}(\bar{\beta}_n)\big)\Big] = 0.
\end{equation}
\end{lemma} 

\begin{proof}
Write (recall \eqref{betaKndef})
\begin{equation}
\E_{\Phi_n\mu}\Big[\hat{\theta}^{M_n,n}(t)\big(1-\hat{\theta}^{M_n,n}(t)\big)\Big] 
=  \frac{1}{|\G_n|^2} \frac{1}{\kappa_{M_n}} \sum_{u_1,u_2 \in \S_n} K(u_1)K(u_2)\,\E_{\Phi_n\mu}\big[z_{u_1}(t)\big(1-z_{u_2}(t)\big)\big].
\end{equation}
The comparison duality in Section~\ref{sss.partclus} gave us that, for every $t \geq 0$ and $\epsilon>0$ (recall \eqref{zepsdef}),
\begin{equation}
\E_{\Phi_n\mu}\big[z^-_{u_1}(t)\big(1-z^+_{u_2}(t)\big)\big] \leq \P_{(2,u_1,u_2)}(B_0(t)=2),
\end{equation}
where we recall the definition of the Markov process $(B(t))_{t \geq 0}$ on state space $\{1,2\} \times \S_n \times \S_n$ with transition rates defined in \eqref{12MCrates}. Let $p_n(t) = \sup_{u_1,u_2 \in \S_n} \P_{(2,u_1,u_2)}(B_0(t)=2)$. Then
\begin{equation}
\E_{\Phi_n\mu}\Big[\hat{\theta}^{M_n,n}(t)\big(1-\hat{\theta}^{M_n,n}(t)\big)\Big] 
\leq  [p_n(t)+\epsilon(2-\epsilon)]\,\frac{1}{|\G_n|^2} \frac{1}{\kappa_{M_n}} \sum_{u_1,u_2 \in \S_n} K(u_1)K(u_2) 
= p_n(t)+\epsilon(2-\epsilon).
\end{equation}
Letting $\epsilon \downarrow 0$, we see that it suffices to show that
\begin{equation}
\label{pnlim}
\lim_{n\to\infty} p_n(\bar{\beta}_n) = 0.
\end{equation}
But this follows from the same type of argument as in the last steps of the proof in Section~\ref{sss.partclus} with $L_n$ replaced by $M_n$, where we use that $M_n^\beta = \beta^*_n \ll \bar{\beta}_n$.
\end{proof}


\appendix


\section{Preservation of diffusion function class and reduction of volatility}
\label{appA*}

In this appendix we show that the renormalisation map $\CF$ for the infinite system preserves the class $\CG$ of proper diffusion functions (recall \eqref{gh6}). This is needed for the existence and uniqueness of the scaling limit in the finite-systems scheme. We further show that $\CF$ maps the Fisher-Wright diffusion function to a smaller multiple of itself (recall \eqref{gh20}). This says that the seed-bank reduces the volatility. 

We begin with Model 1, and proceed as in \cite{CGSh95}. Let
\begin{equation}
\label{mrw1}
b^{(1)}\big((i,R_i),(j,R_j)\big) = \left\{ \begin{array}{ll}
\hat{a}(i,j), & R_i = R_j =A,\\
Ke, & i = j, R_i=A, R_j = D, \\
e, & i = j, R_i=D, R_j = A,\\
0, &\mbox{otherwise}.
\end{array}
\right.
\end{equation} 
be the kernel for the motion of a lineage in the dual of Model 1. This describes a Markov process on $\G$ with migration kernel $\hat{a}(\cdot,\cdot)$ that becomes dormant (state $D$) at rate $Ke$ (after which it stops moving) and active (state $A$) at rate $e$ (after which it starts moving again). Let 
\begin{equation}
\label{hazard}
\hat{B}(i,j)=\int_{0}^{\infty} \d t \sum_{k\in\G} b^{(1)}_t\big((k,A),(i,A)\big)\,b^{(1)}_t\big((k,A),(j,A)\big),
\qquad i,j \in \G,
\end{equation}
where $b^{(1)}_t(\cdot,\cdot)$ is the time-$t$ transition kernel of the Markov process with transition kernel $b^{(1)}(\cdot,\cdot)$ in \eqref{mrw1}. This is the mean of the \emph{total joint activity time} of two copies of the Markov process starting from $(i,A)$ and $(j,A)$. The coexistence regime corresponds to $\hat{B}(0,0)<\infty$ (recall Section~\ref{ss.coreinf}). 
   
The following lemma is the seed-bank analogue of \cite[Lemma 2.11]{CGSh95}.

\begin{lemma}{\bf [Moments and coupling]}
\label{lem:cov} 
Suppose that $\hat{B}(0,0)<\infty$. Then
\begin{enumerate}
\item 
For all $i,j \in \G$,
\begin{equation}
\label{mm1}
\E_{\nu_\theta}[x_{i}]=\theta,\qquad \E_{\nu_\theta}[x_{i}x_j]=\theta^2+\hat{B}(i,j)\,\E_{\nu_\theta}[g(x_0)]. 
\end{equation} 
\item 
For every $\theta,\theta^\prime\in [0,1]$ with $\theta < \theta^\prime$ there exist a coupling $\P^{\theta,\theta^\prime}$ of the random variables $z=(z_i)_{i\in\G}$ and $z^\prime=(z_i^\prime)_{i\in\G}$ such that $\CL[z]=\nu_\theta$, $\CL[z^\prime]=\nu_{\theta^\prime}$ and $\P^{\theta,\theta^\prime}(x_i \leq x^\prime_i\,\,\forall\,i\in\G)=1$. Consequently, $\E^{\theta,\theta'}[|x_i-x^\prime_i|] = |\theta-\theta^\prime|$ for all $\theta,\theta^\prime\in [0,1]$.
\end{enumerate}
\end{lemma}

\begin{proof}
Part 1 follows from \cite[Lemma 6.1]{GdHOpr1} by picking the equilibrium measure as the initial measure. Part 2 follows from \cite[Lemma 6.6, Section 6.3.1]{GdHOpr1}, where the coupling is achieved by using the same Brownian motions. 
\end{proof}

For Model 2,  the analogue of Lemma~\ref{lem:cov} holds with $b^{(1)}(\cdot,\cdot)$ replaced by the kernel for the motion of a lineage in the dual of Model 2 (for symmetric migration):
\begin{equation}
\label{mrw2}
b^{(2)}\big((i,R_i),(j,R_j)\big) = \left\{ \begin{array}{ll}
\hat{a}(i,j), &R_i = R_j = A,\\
K_m e_m, &i = j, R_i=A, R_j = D_m, m\in\N_0,\\
e_m, & i = j, R_i = D_m, R_j = A, m\in\N_0,\\
0, &\mbox{otherwise}.
\end{array}
\right.
\end{equation}
and in the dual of Model 3 (for symmetric migration):
\begin{equation}
\label{mrw3}
b^{(3)}((i,R_i), (j,R_j)) = \left\{ \begin{array}{ll}
\hat{a} (i,j), &R_i = R_j = A,\\
K_m e_m  \hat{a}_m(j,i), &R_i=A, R_j = D_m, m\in\N_0,\\
e_m \hat{a}_m(i,j), &R_i=D_m, R_j = A, m\in\N_0.
\end{array}
\right.
\end{equation} 
These kernels define appropriate analogues of $\hat{B}(i,j)$ in \eqref{hazard}.
  
Our main result in this appendix is the following theorem.

\begin{theorem}{{\bf [Preservation and reduction: Models 1--3]}}
\label{T.comp}
\begin{itemize}
\item[{\rm (a)}] 
If $g\in\CG$, then $\CF g\in\CG$.
\item[{\rm (b)}] 
If $g=dg_{\mathrm{FW}}$, $d \in (0,\infty)$, then $\CF g = d^\ast g_{\mathrm{FW}}$ 
with $d^\ast \in (0,d)$. 
\end{itemize}
\end{theorem}

\begin{proof}
(a) We must check that $\CF g$ satisfies the constraints defining the class $\CG$ in \eqref{gh6}.
\begin{itemize}
\item 
If $\theta \in (0,1)$, then $\nu_\theta$ puts positive mass in $(0,1)$. Since $g(x)>0$ for all $x \in (0,1)$, it follows from \eqref{gh19} that $(\CF g)(\theta)>0$. 
\item
By \eqref{mm1},
\begin{equation}
(\CF g)(\theta)=\E_{\nu_\theta}[g(x_0)]=\frac{1}{\hat{B}(0,0)}\,\E_{\nu_\theta}[(x_0-\theta)^2].
\end{equation}  
Thus, if $(\CF g)(\theta)=0$, then by translation invariance we have $\nu_\theta=\delta_\theta$. Since $g(0)=g(1)=0$, this implies  that $(\CF g)(0)=(\CF g)(1) = 0$.
\item 
For any $\theta,\theta^\prime \in [0,1]$,
\begin{equation}
\begin{aligned}
\big|(\CF g)(\theta) - (\CF g)(\theta^\prime)\big| &= \big|\E^{\theta,\theta'}[g(x_0) - g(x^\prime_0)]\big|
\leq \E^{\theta,\theta'}[|g(x_0) - g(x^\prime_0)|]\\
&\leq \text{Lip}[g] \, \E^{\theta,\theta'}[|x_0 - x^\prime_0|] = \text{Lip}[g] \,|\theta-\theta^\prime|,
\end{aligned}
\end{equation}
with $\text{Lip}[g]$ the global Lipschitz constant of $g$. Hence $\text{Lip}[\CF g] \leq \text{Lip}[g]$. Since $g$ is globally Lipschitz, so is $\CF g$.
\end{itemize}

\medskip\noindent
(b) If $g=dg_{\mathrm{FW}}$, then 
\begin{equation}
(\CF g)(\theta) = \E_{\nu_\theta}[g(x_0)] = \E_{\nu_\theta}[dx_0(1-x_0)] 
= d \theta - d\big[\theta^2 + \hat{B}(0,0)(\CF g)(\theta)\big], 
\end{equation}
where we use \eqref{mm1}. Hence $(\CF g)(\theta) = d^\ast g$ with $d^\ast = d/(1+d\hat{B}(0,0))$.
\end{proof}


\section{General conditions for the finite-systems scheme} 
\label{appB*}

In this appendix we recall an \emph{abstract scheme} from \cite{CG94*} that lists general conditions for the validity of the finite-systems scheme. Together with ideas from \cite[Section 4]{DGV95}, this allows us in Sections~\ref{s.fssrhofin}--\ref{s.fssrhoinfslow} to prove the finite-systems scheme described in Section~\ref{s.scaling}. Model 2 with $\rho=\infty$ in regime (II) needs a different route, which is followed in Section~\ref{s.fssrhoinffast}.  

\begin{remark}{{\bf [Conversion of notation for abstract scheme with seed-bank]}}
\label{rem:translate}
{\rm In what follows we stick to the notation used in \cite{CG94*}. Since the latter was written for systems \emph{without} seed-bank, the reader must everywhere insert the following objects:

\bigskip\noindent
$\bullet$ $\rho<\infty$:\\[0.2cm]
\begin{tabular}{lll}
$\G$, $\G_n$        &$\to$     &$\G$, $\G_n$\\
$I$         &$\to$     &$[0,1] \times [0,1]$ (Model 1)\\
             &$\to$     &$[0,1] \times [0,1]^{\N_0}$, $[0,1] \times [0,1]^{M_n+1}$ (Model 2--3)\\
$\CJ$    &$\to$      &$[0,1]$\\
$E$       &$\to$      &$I^\G$\\
$\CT,\CT^\mathrm{erg},\CI$ &$\to$ &$\CT,\CT^\mathrm{erg},\CI$ in \eqref{laws}, \eqref{lawinv}\\
$\CT^\mathrm{erg}_\theta$ &$\to$ &$\CT^\mathrm{erg}_\theta$ in \eqref{erglaw1}, \eqref{erglaw2fin}\\
$\mu_n$       &$\to$     &$\mu_n$ in \eqref{choicelawn}\\
$\CT_n$       &$\to$     &$\CT_n$ in \eqref{trlinvn}
\end{tabular}

\bigskip\noindent
$\bullet$ $\rho=\infty$:\\[0.2cm]
\begin{tabular}{lll}
$\G$, $\G_n$      &$\to$     &$\G \times (A,D)$, $\G_n \times (A,D)$ (Model 1)\\  
                           &$\to$     &$\G \times \{(A,(D_m)_{m\in\N_0}\}$, 
                                             $\G_n \times \{(A,(D_m)_{0 \leq m \leq M_n}\}$ (Model 2--3)\\
$I$         &$\to$     &$[0,1]$\\
$\CJ$    &$\to$     &$[0,1]$\\
$E$       &$\to$     &$I^{\G \times (A,D)}$ (Model 1)\\
              &$\to$    &$I^{\G \times \{(A,(D_m)_{m\in\N_0}\}}$ (Model 2--3)\\
$\CT,\CT^\mathrm{erg},\CI$ &$\to$ &$\CT^\bullet,\CT^{\mathrm{erg},\bullet},\CI$ in \eqref{lawscr}, \eqref{lawinv}\\
$\CT^\mathrm{erg}_\theta$ &$\to$ &$\CT^{\mathrm{erg},\bullet}_\theta,\CT^{\mathrm{erg},\bullet,*},\CT^{\mathrm{erg},\blacksquare}_\theta$ in \eqref{CTcov},\eqref{erglaw2inf}, \eqref{ergbulletlaw}\\
$\mu_n$        &$\to$     &$\mu_n$ in \eqref{choicelawn}\\
$\CT_n$        &$\to$     &$\CT_n$ in \eqref{trlinvn}
\end{tabular}

\medskip\noindent
With these replacements, the abstract scheme applies to systems \emph{with} seed-bank. For $\rho<\infty$ the seed-bank can be incorporated via an extension of the single-component state space $I$, and no extension of the geographic space $\G$ is needed. For $\rho=\infty$, however, this is not enough and the geographic space $\G$ needs to be extended with the seed-bank space $\{(A,(D_m)_{m\in\N_0}\}$. This difference once more underpins the fact that for $\rho<\infty$ the effect of the seed-bank is minor, while for $\rho=\infty$ it is not. In the latter case, colour regularity is needed to control the deep seed-banks.    
}\hfill $\Box$
\end{remark}


\paragraph{General set-up.}

In order to be able to formulate the abstract theorem, we need to set up the framework and introduce the relevant notation. Let $\G$ be a countable set, let $I$ be a Polish space, and let $(Z(t))_{t\geq 0}$ be a Markov processes with state space $I^\G$ (endowed with the product topology) i.e., $Z(t)=(Z_i(t))_{i \in \G}$, with associated semigroup of transition operators $(S(t))_{t\geq 0}$. Let $(\G_n)_{n\in\N}$ be a projective systems of discrete finite groups that increases to $\G$, and let $(Z_n(t))_{t\geq 0}$, $n\in\N$, be Markov processes with state space $E \subseteq I^{\G_n}$ and semigroup $(S_n(t))_{t\geq 0}$, $n\in\N$. For $Z \in I^\G$, define the element $Z |_{\G_n} \in I^{\G_n}$ by the restriction $(Z|_{\G_n})_i=Z_i$, $i \in \G_n$. Suppose that there are extension operators $\tilde\phi_n\colon\,I^{\G_n} \to I^\G$, $n\in\N$, satisfying $(\tilde\phi_nZ_n)|_{\G_n}=Z_n$, and write $\tilde\phi_n$ also for the corresponding push forward acting on measures.
 
If $\mu$ is a measure on $I^\G$ and $\mu_n$ is a measure on $I^{\G_n}$, $n\in\N$, then write $\lim_{n\to\infty} \mu_n = \mu$ as $\lim_{n\to\infty} \tilde\phi_n \mu_n = \mu$. A function $f\colon\,I^\G \to \R$ is called tame when it is bounded, continuous and depends on only \emph{finitely many} coordinates. Write $\lim_{n\to\infty}\CL[\mu_n - \nu_n] = 0$ as $n \to \infty$ when $\lim_{n\to\infty} [\langle \mu_n,f \rangle - \langle \nu_n,f \rangle] = 0$ for all tame $f$.

We need four groups of assumptions.


\paragraph{\emph{$I$: The processes $Z_n(\cdot)$, $n\in\N$, are finite versions of $Z(\cdot)$.}}

This first set of assumptions makes precise the statement that $Z_n(\cdot)$, $n\in\N$, can be considered \emph{finite} versions of $Z(\cdot)$, and identifies certain classes of appropriate initial distributions $\CT$ and $\CT_n$ for $Z(t)$ and $Z_n(t)$.

\begin{itemize}
\item[\textbf{(A1)}] 
For each $Z(0) \in I^\G$, with $Z_n(0)= Z(0)|_{\G_n}$,
\begin{equation}
\label{e3556}
\lim_{n\to\infty} \CL \left[Z_n(t)\right] = \CL \left[Z(t)\right], \qquad \forall\,t \geq 0.
\end{equation}
\item[\textbf{(A2)}] 
There are sets of measures $\CT$ and $\CT_n$ on $I^\G$, respectively, $I^{\G_n}$, such that $\tilde\phi_n(\CT_n) \subset \CT$, and these sets are closed under the actions of the associate semigroups, i.e., for all $t \geq 0$, $\mu_n \in \CT_n$ implies $\mu_n S_n(t) \in \CT_n$ and $\mu \in \CT$ implies $\mu S(t) \in \CT$.
\end{itemize}


\paragraph{\emph{$II$: Basic ergodic properties of the infinite system.}} 

The second set of assumptions are concerned with the equilibrium states and the ergodic properties of the infinite system. Let $\CI$ denote the set of invariant measures for $Z(\cdot)$, i.e., all $\mu \in \CT^\mathrm{erg}$ such that $\mu S(t)=\mu$ for all $t \geq 0$.

\begin{itemize}
\item[\textbf{(A3)}] 
There is a Polish space $\CJ$ and a set of probability measures $\{\nu_\theta\colon\,\theta \in \CJ\}$ on $I^\G$ such that $\{\nu_\theta\colon,\theta \in \CJ\} \subset \CI$. For $\theta \in \CJ$, let $\CT^\mathrm{erg}_\theta$ be the set of all $\mu \in \CT^\mathrm{erg}$ such that $\lim_{t\to\infty} \mu S(t)= \nu_\theta$. (Typically, $\{\nu_\theta\colon\,\theta \in \CJ\}$ is the set of \emph{extreme points} of $\CI \cap \CT^\mathrm{erg}$.)
\item[\textbf{(A4)}] 
There is a random variable $\Theta$ on $I^\G$, called the \emph{macroscopic variable}, and functions $\hat \theta^m \colon\,I^{\G_n} \to \CJ$ that are estimators of $\Theta$, such that for all $\mu \in \CT^\mathrm{erg}$ the following limit exists (saying that $\hat \theta$ is a consistent estimator of $\theta$):
\begin{equation}
\label{e3572}
\lim_{n\to\infty} \hat \theta^n(Z|_{\G_n}) = \Theta(Z), \qquad \text{in law under } \mu.
\end{equation}
Furthermore, if $\Theta(Z)=\theta$ $\mu$-a.s., then $\mu \in \CT^\mathrm{erg}_\theta$ and, for all $\mu \in \CT$,
\begin{equation}
\label{e3576}
\mu= \int_\CJ \Gamma(\d\theta)\,\mu_\theta,
\end{equation}
where $\Gamma$ is the law of $\Theta(Z)$ with respect to $\mu$, and $\mu_\theta \in \CT^\mathrm{erg}_\theta$, $\theta \in \CJ$.
\end{itemize}


\paragraph{\emph{$III$: Time scale.}} 

In the third set of assumptions we impose certain regularity properties on $\hat \theta_n$, and identify a fundamental \emph{time scale} $\beta_n$ connected with $\hat \theta^n$. We write $\bar \theta^n(t)$ for $\hat \theta^n(Z_n(t))$.

\begin{itemize}
\item[\textbf{(A5)}] 
Let $(T_n)_{n\in\N}$ be any sequence tending to infinity. Suppose that $\CL[Z_n(0)] \in \CT_n$, and $\CL[Z(0)] \in \CT^\mathrm{erg}_\theta$ for some $\theta \in \CJ$. Suppose further that, along some subsequence $(n_k)_{k\in\N}$,
\begin{equation}
\label{e3589}
\lim_{k\to\infty} \CL \left[Z_{n_k}\big(T_{n_k}\big)\right] = \CL[Z].
\end{equation}
Then 
\begin{equation}
\label{e3593}
\lim_{k\to\infty} \CL\left[\hat \theta^{n_k} \big(Z_{n_k} (T_{n_k}) \big) \right] 
= \CL \left[\Theta(Z) \right].
\end{equation}
\item[\textbf{(A6)}] 
There is a sequence $(\beta_n)_{n\in\N}$ tending to infinity such that if $\CL[Z_n(0)] \in \CT_n$ and $\lim_{n\to\infty} \CL[Z_n(0)] = \CL[Z(0)] \in \CT^\mathrm{erg}$, then
\begin{itemize}
\item[--] 
the family of measures $\CL[\hat \theta^n(s\beta_n)_{s \geq 0}]$ on $D (\CJ, [0,\infty))$ is \emph{tight}, 
\item[--] 
all weak limit points $(\Theta(s))_{s \geq 0}$ have the property that $s \mapsto E[f(\Theta(s))]$ is continuous for all bounded continuous functions $f$ on $\CJ$.
\end{itemize}
We want that $\beta_n$ is proper, i.e., leads to a non-trivial limit process. This has to be checked afterwards. 
\end{itemize}


\paragraph{\emph{$IV$: Compactness and strengthened convergence properties.}} 

In the fourth set of assumptions we impose compactness and require a strengthened form of (A1), as well as a \emph{strengthened Feller property} for $(Z(t))_{t\geq 0}$.

\begin{itemize}
\item[\textbf{(A7)}] 
Fix $T < \infty$, and let $(t_n)_{n\in\N}$ satisfy $0 \leq t_n \leq T \beta_n$ for $n\in\N$. If $\CL[Z_n(0)] \in \CT^n$, then the family $\CL(Z_n(t_n))$, $n\in\N$, is tight and all its weak limit points belong to $\CT$.
\item[\textbf{(A8)}] 
If $\mu,\mu_n \in \CT^\mathrm{erg}$ are such that $\lim_{n\to\infty} \mu_n = \mu$ and $\lim_{n\to\infty} t_n = \infty$, then
\begin{equation}
\label{e3606}
\lim_{n\to\infty} [\mu_n S(t_n)-\mu S(t_n)] = 0.
\end{equation}
\item[\textbf{(A9)}] 
There is a sequence $(L_n)_{n\in\N}$ tending to infinity such that if $0 \leq t_n \leq L_n$ and $\mu_n \in \CT_n$, then
\begin{equation}
\label{e3610}
\lim_{n \to\infty} [\mu_n S_n(t_n)-(\tilde\phi_n \mu_n)S(t_n)] = 0.
\end{equation}
\end{itemize} 
Assumptions (A8) and (A9) should be viewed as statements about uniformity of the ergodic theorem in the initial measure and uniformity in the approximation of the infinite system by finite systems for large times. The most difficult assumptions to check are (A6), (A8) and (A9), and (A10) below.


\paragraph{Results.}

We are now ready to state our abstract theorem.

\begin{theorem}{\bf [Abstract theorem]}
\label{th.3624}
Suppose that $Z(t),Z_n(t)$, $n\in\N$, satisfy {\rm (A1)--(A9)}, and that $\CL(Z(0)) \in \CT_\theta$ for some $\theta \in \CJ$. Let $(n_k)_{k\in\N}$ be any sequence such that
\begin{equation}
\label{e3627}
\lim_{k\to\infty} \CL \left[\big(\hat \theta^{n_k}(s \beta_{n_k})\big)_{s \geq 0}\right] 
= \CL \left[(\Theta(s))_{s \geq 0}\right] \mbox{ in } D \left(\CJ,[0,\infty)\right).
\end{equation}
For any $s \in (0,\infty)$, if $T(n) \to \infty$ and $T(n)/\beta_n \to s$ as $n \to \infty$, then
\begin{equation}
\label{e3631}
\lim_{k\to\infty} \CL \left[Z_{n_k} \big(sT(n_k)\big)\right] 
= \int_\CJ \; P(\Theta(s) \in d\theta)\,\nu_\theta,
\end{equation}
where $P$ is the law of $(\Theta(s))_{s \geq 0}$.
\end{theorem}

\noindent
Note that, because of (A6), sequences satisfying \eqref{e3627} exist. Suppose that we add the assumption: 
\begin{itemize}
\item[\textbf{(A10)}] 
The weak limit in \eqref{e3627} is unique.
\end{itemize}
Then the finite-systems scheme holds as claimed. The way we are able to check (A10) is as follows.

\begin{corollary}{\bf [Verification of (A10)]}
\label{cor.3636}
If a weak limit point of $\CL[(\hat\theta^{n_k}(s \beta_{n_k}))_{s \geq 0}]$ as $k\to\infty$ satisfies a well-posed martingale problem, independently of $(n_k)_{k \in \N}$, then \eqref{e3631} holds and $\CL[(\hat\theta^n(s\beta_n))_{s \geq 0}]$ converges as $n\to\infty$ to the solution of this martingale problem.
\end{corollary}

Once we have checked (A1)--(A10), we will want to verify that the macroscopic time scale $\beta_n$ is proper, i.e., at time $s\beta_n$ the finite system locally converges to $\nu_{\Theta(s)}$, the equilibrium of the infinite system with density $\Theta(s)$, where $(\Theta(s))_{s \geq 0}$ is a non-trivial random process with cadlag paths hitting the traps of the evolution as $s\to\infty$, which in our models are $0$ or $1$. 


\section{Fraction of time spent in the active state}
\label{appC*}

In this appendix we show that, in the limit as $M\to\infty$, the fraction of time spent in the active state until time $M^\beta$ is $\asymp f_M$ with
\begin{equation}
\label{acfrac}
f_M = \frac{1}{1+\sum_{0 \leq m \leq M} K_m}.
\end{equation} 
Because the transitions into and out of the seed-bank do not depend on the location of the Markov process, it suffices to look at the Markov process projected onto $E=\{A,(D_m)_{\N_0}\}$. Transitions from dormant states with colour $> M$ occur at rate $O(M^{-\beta})$ as $M\to\infty$. Hence these transitions may be ignored, at the cost of a finite correction factor, so that we only need to look at the projection onto $\mathcal{C}_M=\{A,(D_m)_{0 \leq m \leq M}\}$.

Let $(U_k)_{k\in\N}$ be the discrete-time Markov chain on $\mathcal{C}_M$ with transition probabilities $1$ for $D_m \to A$ and $p_m$ for $A \to D_m$ with
\begin{equation}
\label{trpr}
p_m = \frac{K_m e_m}{\sum_{0 \leq \ell \leq M} K_\ell e_\ell}, \qquad 0 \leq m \leq M. 
\end{equation}
The transition times are equal in distribution to 
\begin{equation}
A \to (D_m)_{0 \leq m \leq M}\colon\,\frac{1}{\sum_{0 \leq m \leq M} K_m e_m}\, Z,
\qquad 
D_m \to A\colon\,\frac{1}{e_m}\, Z',
\end{equation}
with $Z,Z' $ independent $\mathrm{Exp}(1)$ random variables. For $n \in \N$, let $T^{A}_n$ and $T^{D}_n$ denote the \emph{total} active time and dormant time after $n$ transitions out of and into $A$, when the Markov process starts in $A$. Then, in distribution,
\begin{equation}
T^{A}_n = \sum_{k=0}^{n-1} \frac{1}{\sum_{0 \leq m \leq M} K_m e_m}\, Z_k,
\qquad
T^{D}_n = \sum_{k=0}^{n-1} \frac{1}{e_{M_k}}\, Z'_k,
\end{equation}
where $(Z_k)_{k\in\N_0},(Z'_k)_{k\in\N_0}$ are sequences of independent $\mathrm{Exp}(1)$ random variables. Since $\mathcal{C}_M$ is finite, by \eqref{trpr} we have
\begin{equation}
\label{LLNalt}
\lim_{n\to\infty} \frac{1}{n} \sum_{k=0}^{n-1} \delta_{U_k} = \sum_{0 \leq m \leq M} p_m \delta_m \quad \text{in probability}. 
\end{equation}
Hence
\begin{equation}
\lim_{n\to\infty} \frac{1}{n} T^{A}_n = \mu^{A}, 
\qquad \lim_{n\to\infty} \frac{1}{n} T^{D}_n = \mu^{D}
\quad \text{in probability},
\end{equation}
with 
\begin{equation}
\mu^{A} = \frac{1}{\sum_{0 \leq m \leq M} K_m e_m},
\qquad  \mu^{D} = \frac{1}{\sum_{0 \leq m \leq M} K_m e_m} \sum_{0 \leq m \leq M} K_m.
\end{equation}
Note that
\begin{equation}
\mu^{D} = \mu^{A} \sum_{0 \leq m \leq M} K_m. 
\end{equation}
Consequently,
\begin{equation}
\lim_{n\to\infty} \frac{T^{A}_n}{T^{A}_n+T^{D}_n} = \frac{1}{1+\sum_{0 \leq m \leq M} K_m} = f_M 
\quad \text{in probability}.
\end{equation}
Hence a fraction $f_M$ of time is spent in the active state, where we note that until time $M^\beta$ \emph{many} transitions $A \to D_m$ occur for \emph{most} $0 \leq m \leq M$, so that the law of large numbers in \eqref{LLNalt} applies, at the cost of a finite correction factor.


\section{Speculations about clustering for fast growing seed-bank}
\label{appD*}

In this appendix we speculate about how partial and complete clustering may come about in Model 2 with $\rho=\infty$ in regime (II). We will argue that both on time scale $\beta^{**}_n$ and time scale $\beta^*_n$ the macroscopic variable moves towards fixation according to a \emph{jump process}, i.e., it follows a piecewise constant path that ends in $0$ or $1$. This is different from the diffusive clustering found in Model 1, Model 2 with $\rho<\infty$ and Model 2 with $\rho=\infty$ in regime (I) (see Fig.~\ref{fig:WFren}). 


\subsection{Partial clustering} 


\paragraph{Step 1: Hazard.}

Let $I(t) \in \{0,1\}$ be the indicator of the event that two independent Markov processes with transition kernel $b^{(2),n}(\cdot,\cdot)$, starting at a distance of order $n$, at time $t$ occupy the same colony in $\G_n$ and are jointly active. Let
\begin{equation}
\label{e2787}
\zeta_0 = 0, \quad \eta_0 = 0, \quad \zeta_k = \inf\{t > \eta_{k-1}\colon\,I(t)=1\}, \quad  
\eta_k = \inf\{t > \zeta_k\colon\,I(t)=0\}, \qquad k \in \N,
\end{equation}
be the successive times at which joint active occupancy is switched on and off. The successive intervals of joint active occupancy are $(\Delta_k)_{k\in\N}$ with $\Delta_k = [\zeta_k,\eta_k)$, and so $I(t) = 1$ if and only if $t \in \cup_{k\in\N} \Delta_k$. The number of completed intervals up to time $t$ is 
\begin{equation}
\label{e2795}
K(t) = \sup\{k\in\N\colon\,\eta_k \leq t\}. 
\end{equation}
We are interested in the \emph{total joint active occupancy time} up to time $t$,
\begin{equation}
\label{e2799}
C(t) = \sum_{k=1}^{K(t)} |\Delta_k| + (t-\zeta_{K(t)+1})\,1_{\{\zeta_{K(t)+1}<t\}}.
\end{equation}
This is the total time up to time $t$ during which the two Markov processes coalesce at rate $d$. For $t \to\infty$ the boundary term is negligible (because $(t-\zeta_{K(t)+1})\,1_{\{\zeta_{K(t)+1}<t\}} \leq |\Delta_{K(t)+1}|$). Therefore, the scaling of $C(t)$ is determined by the scaling of $K(t)$ and the law of $(|\Delta_k|)_{k\in\N}$. However, in order to find the scaling of $C(t)$ it is better to look at a \emph{larger time scale}. Indeed, what happens is that the two Markov processes experience \emph{bursts} of joint active occupancy times until they separate and meet again. These joint activity time intervals come in rapid succession, until the two Markov processes move far apart and loose track of each other. After that they both need to make order $|\G_n|$ steps until they meet again. In view of this scale separation, we divide time into two parts: the long stretches \emph{between} the burst and the short stretches \emph{during} the burst. 


\paragraph{Step 2: Bursts.}

We begin by focussing on the \emph{long stretches}. Let us discretise time and consider a renewal process $\zeta = (0,\zeta_1,\zeta_1+\zeta_2,\ldots)$ on $\N_0$ with i.i.d.\ increments such that $\P(\zeta_1<\infty) = 1$ and $\P(\zeta_1 = n) \sim C n^{-(1+\gamma)}$, $n\to\infty$. Let $\zeta'$ be an independent copy of $\zeta$, and put $\CG = \zeta \cap \zeta'$. Then also $\CG$ is a renewal process on $\N_0$, written $\CG = (0,\CG_1,\CG_1+\CG_2,\ldots)$. It was shown in \cite[Theorems 1.3 and 1.5]{AB16} that, for $\gamma \in (\tfrac12,1)$,
\begin{equation}
\label{jr1}
\P(\CG_1<\infty) = 1, \qquad \P(\CG_1 = n) \sim C^*n^{-(1+\gamma^*)}, \quad n \to\infty, 
\end{equation}
with
\begin{equation}
\label{jr2}
\gamma^* = 2\gamma-1, \qquad C^* =  \pi C^2\,\big[\gamma^*\sin(\pi\gamma^*)\big]\,
\big[\gamma\sin(\pi\gamma)\big]^{-2}.
\end{equation}
Returning to continuous time, we see that \eqref{jr1}--\eqref{jr2} tell us that if we \emph{ignore} the time lapses during which the two Markov processes \emph{are active} (which are negligible because they have finite mean), then the time lapses $U_1,U_2,\ldots$ between the successive times at which they \emph{become jointly active} are i.i.d.\ with
\begin{equation}
\label{Uscal}
\P(U_1<\infty) = 1, \qquad \frac{\P(U_1 \in \d t)}{\d t} \sim C^* t^{-(1+\gamma^*)}, \quad t \to \infty.
\end{equation} 
However, by \eqref{pnmix}, on average at only one out of $|\G_n|$ such times the two Markov processes occupy the same site. Hence, the time lapses $V_1,V_2,\ldots$ between the successive times at which they \emph{become jointly active at the same colony} are i.i.d.\ with 
\begin{equation}
\label{e2844}
\P(V_1 \in \d t) = \sum_{k\in\N} \frac{1}{|\G_n|}\,\left[1-\frac{1}{|\G_n|}\right]^{k-1}\,
\P\left(\sum_{\ell=1}^k U_\ell \in \d t\right).
\end{equation}
Putting $t = s|\G_n|^{1/\gamma^*}$, $k=\bar{s}|\G_n|$ and letting $n\to\infty$, we obtain 
\begin{equation}
\label{e2850}
\lim_{n\to\infty} \frac{\P(|\G_n|^{-1/\gamma^*}V_1 \in \d s)}{\d s} 
= \int_0^\infty \d \bar{s}\,\e^{-\bar{s}} \lim_{n\to\infty}
\frac{\P(|\G_n|^{-1/\gamma^*} \sum_{\ell=1}^{\bar{s}|\G_n|} U_\ell \in \d s)}{\d s}.
\end{equation}
It follows from \eqref{Uscal} that
\begin{equation}
\label{e2856}
\lim_{N\to\infty} \CL\left[N^{-1/\gamma^*} \sum_{\ell=1}^N U_\ell\right] = \CL[U^*]
\end{equation}
with $U^*$ a positive stable law random variable with exponent $\gamma^*$. Hence, putting $T^n_1 = |\G_n|^{-1/\gamma^*} V_1=V_1/\beta^{**}_n$, we find that
\begin{equation}
\label{Zlim}
\lim_{n\to\infty}  \frac{\P(T^n_1 \in \d s)}{\d s} = \int_0^\infty \d \bar{s}\,\e^{-\bar{s}}\, 
\frac{\P(\bar{s}^{1/\gamma^*} U^* \in d s)}{\d s}.
\end{equation}  
This says that $T^n_1$ converges in distribution to a random variable $T_1$ whose density is given by the right-hand side of \eqref{Zlim}. In terms of the Laplace transform, we get
\begin{equation}
\label{e2868}
\begin{aligned}
\E[\e^{-\lambda T_1}] &= \int_0^\infty \d s\,\e^{-\lambda s}\,\P(T_1 \in \d s)
= \int_0^\infty \d s\,\e^{-\lambda s}\,\int_0^\infty \d \bar{s}\,\e^{-\bar{s}}\,\P(\bar{s}^{1/\gamma^*} U^* \in d s)\\
&= \int_0^\infty \d \bar{s}\,\e^{-\bar{s}}\,\E\left[e^{-\lambda \bar{s}^{1/\gamma^*} U^*}\right]
=  \int_0^\infty \d \bar{s}\,\e^{-\bar{s}}\,\e^{-D(\lambda \bar{s}^{1/\gamma^*})^{\gamma^*}}\\
&= \int_0^\infty \d \bar{s}\,\e^{-\bar{s}}\,e^{-\bar{s}D\lambda^{\gamma^*}}
= (1+D\lambda^{\gamma^*})^{-1}, \qquad \lambda>0,
\end{aligned}
\end{equation}
for some $D \in (0,\infty)$ depending on $C^*$.


\paragraph{Step 3: Short stretches.}

Next we turn to the \emph{short stretches}. In the limit as $n\to\infty$, the total joint active occupancy times during the successive bursts, which we denote by $W_1,W_2,\ldots$, are equal in distribution to the total joint active occupancy time of two Markov processes on $\G$ starting from the origin. Indeed, for a burst to end the two Markov processes must move apart and loose track of each other (as they do on $\G$). It is only after they meet again (as they do on $\G_n$) and are jointly active that the next burst starts, which takes a much longer time than the typical time of a burst. Thus, the law of $W_1$ is the same as that of the \emph{total joint active occupancy times} we analyzed in \cite[Section 6.2]{GdHOpr1}. In particular, $\E[W_1]=I$ with $I$ the integral in \cite[Eq.~(5.19)]{GdHOpr1}, which by \cite[Eq.~(6.25)]{GdHOpr1} is finite because $I_{\hat{a},\gamma}<\infty$ in the regime of coexistence (recall \eqref{crinf}).

\begin{figure}[htbp]
\begin{center}
\setlength{\unitlength}{.5cm}
\begin{picture}(10,10)(1,0)
\put(0,0){\line(1,0){13}}
\put(0,0){\line(0,1){8}}
\put(0,0){\circle*{.25}}
\put(2,2){\circle*{.25}}
\put(5,3){\circle*{.25}}
\put(7,5){\circle*{.25}}
\put(10,6){\circle*{.25}}
\put(13.5,-.1){$s$} 
\put(-.6,8.5){$H(s)$}
\put(.9,-.7){$T_1$}
\put(3.2,-.7){$T_2$}
\put(5.8,-.7){$T_3$}
\put(8.2,-.7){$T_4$}
\put(10.8,-.7){$T_5$}
\put(2.2,0.9){$W_1$}
\put(5.2,2.3){$W_2$}
\put(7.2,3.9){$W_3$}
\put(10.2,5.2){$W_4$}
\qbezier[10](2,0)(2,1)(2,2)
\qbezier[15](5,0)(5,1.5)(5,3)
\qbezier[25](7,0)(7,2.5)(7,5)
\qbezier[30](10,0)(10,3)(10,6)
\qbezier[30](12,0)(12,3)(12,6)
{\thicklines
\qbezier(0,0)(1,0)(2,0)
\qbezier(2,2)(3.5,2)(5,2)
\qbezier(5,3)(6,3)(7,3)
\qbezier(7,5)(8.5,5)(10,5)
\qbezier(10,6)(11,6)(12,6)
}
\end{picture}
\vspace{0.7cm}
\caption{\small Accumulation of the joint activity time $s \mapsto H(s)$ of two Markov processes in the finite system on time scale $\beta^{**}_n$ in the limit as $n\to\infty$. The time lapses $T_1,T_2,\ldots$ are i.i.d.\ with law given by \eqref{e2868}, the increments $W_1,W_2,\ldots$ are i.i.d.\ with law given by the total joint activity time of two Markov processes in the infinite system.}
\label{fig:accum}
\end{center}
\end{figure}
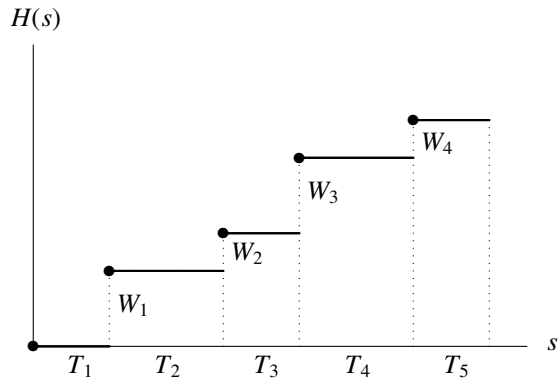


\paragraph{Step 4: Towards partial clustering.}

In view of Steps 1--3 we expect that, for $L_n = o((\beta^{**}_n)^{1/\beta})$,
\begin{equation}
\label{scaldens}
\lim_{n\to\infty} \CL\big[\big(\hat{\theta}^{L_n,n}(s\beta^{**}_n)\big)_{s>0}\big] = (\Theta^*(s))_{s>0}
\end{equation} 
with 
\begin{equation}
\Theta^*(s) = \Theta(H(s)), \qquad s>0, 
\end{equation}
where $(\Theta(s))_{s>0}$ is the diffusion on $[0,1]$ defined in \eqref{gh18}, with diffusion function $\CF g$. Thus, partial clustering is achieved via i.i.d.\ random jumps $W_1,W_2,\ldots$ occurring after i.i.d.\ random time lapses $T_1,T_2,\ldots$.   


\subsection{Complete clustering} 

We expect the exact same behaviour as in \eqref{scaldens} for $(\hat{\theta}^{M_n,n}(s\beta^*_n))_{s>0}$ on time scale $\beta^*_n$, modulo a constant that multiplies $T_1,T_2,\ldots$ and depends on the parameters controlling the deep seed-banks (recall in \eqref{ass3}). The hazard at time $s$ is again growing in a bursty manner, now for Markov processes that start from the deepest seed-banks.


\bibliography{seedbank}

\begin{thebibliography}{GHO22b}

\bibitem[AB16]{AB16}
K.~Alexander and Q.~Berger.
\newblock Local asymptotics for the first intersection of two independent
  renewals.
\newblock {\em Electr. J. Probab.}, 21(66):1—20, 2016.

\bibitem[BR76]{BR76}
R.N. Bhattacharya and R.~Ranga Rao.
\newblock {\em Normal Approximation and Asymptotic Expansions}.
\newblock Wiley \& Sons, New York, 1976.

\bibitem[Bre68]{B68}
L.~Breiman.
\newblock {\em Probability}.
\newblock Addison-Wesley, Reading, Massachusetts, 1968.

\bibitem[CG90]{CG90}
J.T. Cox and A.~Greven.
\newblock On the long term behaviour of some finite particle systems.
\newblock {\em Probab. Theory Relat. Fields}, 85:195--237, 1990.

\bibitem[CG91]{CG91}
J.T. Cox and A.~Greven.
\newblock On the long term behavior of finite particle systems:a critical
  dimension example.
\newblock In {\em Random {W}alks, {B}rownian {M}otion, and {I}nteracting
  {P}article {S}ystems}, volume~28 of {\em Progr. Probab.}, pages 203--213.
  Birkh\"{a}user Boston, Boston, MA, 1991.

\bibitem[CG94a]{CG94}
J.T. Cox and A.~Greven.
\newblock Ergodic theorems for infinite systems of locally interacting
  diffusions.
\newblock {\em Ann. Probab.}, 22(2):833--853, 1994.

\bibitem[CG94b]{CG94*}
J.T. Cox and A.~Greven.
\newblock The finite systems scheme: {A}n abstract theorem and a new example.
\newblock {\em CRM Proceedings and Lecture Notes}, 5:55--66, 1994.

\bibitem[CGS95]{CGSh95}
J.T. Cox, A.~Greven, and T.~Shiga.
\newblock Finite and infinite systems of interacting diffusions.
\newblock {\em Probab. Theory Relat. Fields}, 102:165--197, 1995.

\bibitem[Cox89]{C89}
J.T. Cox.
\newblock Coalescing random walks and voter model consensus times on the torus
  in zd.
\newblock {\em Ann. Probab.}, 17(4):1333--1366, 1989.

\bibitem[DG93]{DG93b}
D.A. Dawson and A.~Greven.
\newblock Multiple scale analysis of interacting diffusions.
\newblock {\em Probab. Th. Relat. Fields}, 95:467--508, 1993.

\bibitem[DGV95]{DGV95}
D.A. Dawson, A.~Greven, and J.~Vaillancourt.
\newblock Equilibria and quasi-equilibria for infinite collections of
  interacting {F}leming-{V}iot processes.
\newblock {\em Trans. Amer. Math. Soc.}, 347(7):2277--2360, 1995.

\bibitem[DGW04]{DGW04a}
D.A. Dawson, L.G. Gorostiza, and A.~Wakolbinger.
\newblock Hierarchical random walks.
\newblock In {\em Asymptotic methods in stochastics}, volume~44 of {\em Fields
  Inst. Commun.}, pages 173--193. Amer. Math. Soc., Providence, RI, 2004.

\bibitem[DGW05]{DGW05}
D.A. Dawson, L.~Gorostiza, and A.~Wakolbinger.
\newblock Degrees of transience and recurrence and hierarchical random walk.
\newblock {\em Potential Anal.}, 22(4):305--350, 2005.

\bibitem[EK86]{EK86}
S.N. Ethier and T.~Kurtz.
\newblock {\em Markov {P}rocesses. {C}haracterization and {C}onvergence}.
\newblock John Wiley, New York, 1986.

\bibitem[GHO22a]{GdHOpr3}
A.~Greven, F.~den Hollander, and M.~Oomen.
\newblock Spatial populations with seed-bank: renormalisation on the
  hierarchical group.
\newblock {\em To appear in Memoirs Amer. Math. Soc.}, 2022.

\bibitem[GHO22b]{GdHOpr1}
A.~Greven, F.~den Hollander, and M.~Oomen.
\newblock Spatial populations with seed-bank: well-posedness, duality and
  equilibrium.
\newblock {\em Electron. J. Probab.}, 27:(paper no. 18),1--88, 2022.

\bibitem[Hug95]{H95}
B.D. Hughes.
\newblock {\em Random Walk in Random Environment}, volume~I.
\newblock Clarendon Press, 1995.

\bibitem[JM86]{JM86}
A.~Joffe and M.~M\'{e}tivier.
\newblock Weak convergence of sequences of semimartingales with applications to
  multitype branching processes.
\newblock {\em Adv. in Appl. Probab.}, 18(1):20--65, 1986.

\bibitem[Kre85]{Krengel85}
U.~Krengel.
\newblock {\em Ergodic Theorems}, volume~6 of {\em De Gruyter Studies in
  Mathematics}.
\newblock Walter de Gruyter \& Co., Berlin, 1985.

\bibitem[Lig85]{Lig85}
T.M. Liggett.
\newblock {\em Interacting Particle Systems}.
\newblock Springer, New York, 1985.

\bibitem[OK73]{ohtakimura1973}
T.~Ohta and M~Kimura.
\newblock A model of mutation appropriate to estimate the number of
  electrophoretically detectable alleles in a finite population.
\newblock {\em Genet. Res.}, 22:201--204, 1973.

\bibitem[Spi64]{Sp64}
F.~Spitzer.
\newblock {\em Principles of Random Walk}, volume~34.
\newblock Springer New York, New York, 1964.

\end{thebibliography}
\bibliographystyle{alpha}


\end{document}